\documentclass[a4paper,12pt]{article}

\usepackage{amsthm,amsmath,stmaryrd,bbm,hyperref,geometry,color,authblk}
\usepackage[utf8]{inputenc}
\usepackage{amssymb}
\usepackage[english]{babel}
\usepackage{graphicx}
\usepackage{amsfonts,amssymb}
\usepackage{verbatim}
\usepackage{enumitem}

\usepackage{xcolor} 

\newcommand{\po}{\left(}
\newcommand{\pf}{\right)}
\newcommand{\co}{\left[}
\newcommand{\cf}{\right]}
\newcommand{\cco}{\llbracket}
\newcommand{\ccf}{\rrbracket}
\newcommand{\R}{\mathbb R}

\newcommand{\N}{\mathbb N} 
\newcommand{\dd}{\mathrm{d}}
\newcommand{\na}{\nabla}
 
\newcommand{\hil}{\mathbb{H}}
\newcommand{\oeta}{\overline{\eta}}

\newcommand{\new}[1]{#1}
\newcommand{\rev}[1]{#1}

\newtheorem{thm}{Theorem}
\newtheorem{assu}{Assumption}
\newtheorem*{assu*}{Assumption}
\newtheorem{lem}[thm]{Lemma}

\newtheorem{cor}[thm]{Corollary}
\newtheorem{prop}[thm]{Proposition}
\newtheorem{rem}{Remark}
\newtheorem{Example}{Example}

\title{Local  convergence rates for Wasserstein gradient flows and McKean-Vlasov equations with multiple stationary solutions}
\author[1]{Pierre Monmarché}
\author[2]{Julien Reygner}
\affil[1]{LJLL and LCT, Sorbonne Université, Paris, France}
\affil[2]{CERMICS, Ecole des Ponts, Marne-la-Vallée, France}

\begin{document}

\maketitle

\begin{abstract}
Non-linear versions of log-Sobolev inequalities, that link a free energy to its dissipation along the corresponding Wasserstein gradient flow (i.e. corresponds to Polyak-Lojasiewicz inequalities in this context), are known to provide global exponential long-time convergence to the free energy minimizers, and have been shown to hold in various contexts. However they cannot hold when the free energy admits critical points which are not global minimizers, which is for instance the case of the granular media equation in a double-well potential with quadratic attractive interaction at low temperature. This work addresses such cases, extending the general arguments when a log-Sobolev inequality only holds locally and, as an example, establishing such local inequalities for the granular media equation with quadratic interaction either in the one-dimensional symmetric double-well case or in higher dimension in the low temperature regime. The method provides quantitative convergence rates for initial conditions in a Wasserstein ball around the stationary solutions. The same analysis is carried out for the kinetic counterpart of the gradient flow, i.e. the corresponding Vlasov-Fokker-Planck equation. The local exponential convergence to stationary solutions for the mean-field equations, both elliptic and kinetic, is shown to induce for the corresponding particle systems a fast (i.e. uniform in the number \rev{of} particles) decay of the particle system free energy toward the level of the non-linear limit.
\end{abstract}




\section{Introduction}\label{sec:intro}

\subsection{\rev{Motivation}}

We are interested in the long-time behavior of the granular media equation
\begin{equation}
\label{eq:granular_media}
\partial_t \rho_t = \na\cdot \po \sigma^2\nabla \rho_t + \po \na V + \rho_t \star \na W\pf \rho_t\pf
\end{equation}
or of more general McKean-Vlasov semilinear equations (see \eqref{eq:granular_media_2} below). Here, $\rho_t$ is a probability density over $\R^d$, $V\in\mathcal C^2(\R^d,\R)$, $W\in\mathcal C^2(\R^d\times\R^d,\R)$ are respectively a confining and interaction potential, $W(x,y)=W(y,x)$ for all $x,y\in\R^d$, $\sigma^2>0$ stands for the temperature (or diffusivity) and $\rho \star \na W(x) = \int_{\R^d} \na_x W(x,y)\rho(y)\dd y$. We have typically in mind the double-well case where $d=1$,
\begin{equation}
\label{eq:double_well}
V(x) = \frac{x^4}{4} - \frac{x^2}{2}\,,\qquad W(x,y) = \theta  (x-y)^2
\end{equation}
for some $\theta\in\R$. When $\theta\leqslant 0$ (repulsive interaction), or when $\theta> 0$ and $\sigma^2/\theta$ is large enough (attractive interaction at high temperature or small interaction), there is a unique stationary solution to \eqref{eq:granular_media}, which is globally attractive with exponential rate. However for a fixed $\theta>0$ there is a phase transition at some critical temperature $\sigma_c^2>0$ such that, for $\sigma^2 <\sigma_c^2$, \eqref{eq:granular_media} with \eqref{eq:double_well} admits three distinct stationary solutions (one symmetric, unstable, and two non-symmetric, stable) \cite{Tugautdoublewell}. We are interested in this second case, and more precisely on obtaining \emph{local} convergence rates, namely to prove that solutions which start close to a stable stationary solution converge exponentially fast to it. Among other motivations, this is an important question both for the theoretical understanding of metastable interacting particle systems, where the stability property of the non-linear limit drives the short-time behavior and induces the metastable behavior \cite{Pavliotis,carrillo2020long,BBM,dawson1986large,gvalani2020barriers,bashiri2021metastability}, and for practical questions of optimization in the Wasserstein space, for instance in the mean-field modelling of artificial neural networks \cite{chizat,Szpruch,mei2018mean,mei2019mean} (the loss function being convex only in toy models, like one-layer networks\rev{, cf.~\eqref{eq:neurones}).}

\medskip

To illustrate the point, let us state a result obtained with our approach in the specific case of the symmetric double well \eqref{eq:double_well}. In the next statement (proven in Section~\ref{sec:doublepuit1D}, see Remarks~\ref{rem:decomposition} and \ref{rem:double-puit-degenere}) as in the rest of the work we write $\mathcal W_2$ and $\|\cdot\|_{TV}$ respectively the $L2$ Wasserstein distance and total variation norm and $\mathcal P_2(\R^d)$ the set of probability measures on $\R^d$ with finite second moment.

\begin{prop}\label{prop:double-well}
Consider the granular media \eqref{eq:granular_media} in the case \eqref{eq:double_well} with $\theta>0$.
\begin{enumerate}
    \item If $\sigma^2 < \sigma_c^2$, let $\rho_*$ be one of the two non-symmetric stationary solutions. There exist $\delta,\lambda,C>0$ such that for any initial condition $\rho_0\in\mathcal P_2(\R)$ with $\mathcal W_2(\rho_0,\rho_*) \leqslant \delta$, the corresponding solution to \eqref{eq:granular_media} satisfies, for all $t\geqslant 0$
\[\mathcal W_2(\rho_t,\rho_*) + \|\rho_t - \rho_*\|_{TV} \leqslant C e^{-\lambda t}\,.\]
 \item If $\sigma^2 = \sigma_c^2$, denote by $\rho_*$ the unique stationary solution. For any $\rho_0\in\mathcal P_2(\R)$, there exists $C_0>0$ such that the corresponding solution to \eqref{eq:granular_media} satisfies, for all $t\geqslant 0$,  
\[\mathcal W_2(\rho_t,\rho_*) + \|\rho_t - \rho_*\|_{TV} \leqslant \frac{C_0}{t^{1/3}}\,.\]
\end{enumerate}

\end{prop}

Notice that, in the sub-critical case, there is no assumption on $\sigma^2$ other than $\sigma^2 < \sigma_c^2$, namely, the result holds arbitrarily close to the critical temperature (our results also apply in the simpler super-critical regime and provide, as soon as $\sigma^2 > \sigma_c^2$, a global quantitative exponential convergence toward the unique stationary solution, see Remark~\ref{rem:supercritical}). This is in contrast to the results of \cite{tugaut2023steady}, also concerned with quantitative local convergence of \eqref{eq:granular_media}, which require the temperature to be sufficiently small (for a given $\theta$). Moreover, the Wasserstein convergence in \cite[Theorem 2.3]{tugaut2023steady} only gives a convergence speed of order $e^{-\lambda \sqrt{t}}$ for some $\lambda>0$, and additionally it requires $\theta > -\inf V''$, which our result does not (under this additional assumption we get that the rate $\lambda$ is uniform over $\sigma^2\in(0,\sigma_0^2]$ for any $\sigma_0^2\in(0,\sigma_c^2)$, as in Proposition~\ref{prop:multiwell_new}). The constants $\delta,\lambda,C$ in Proposition~\ref{prop:double-well} can be made explicit (see Remark~\ref{rem:explicit}), namely the result is quantitative.

\subsection{\rev{The strategy in finite dimension}}

Our method is the following. As recalled in the next section, it is now well-known that \eqref{eq:granular_media} can be seen as the gradient flow in the Wasserstein space of some functional. To fix ideas, consider on $\R^d$ such a gradient flow
\[\dot x_t = -\na f(x_t)\,.\]
Let $x_*$ be a local minimizer of $f$. Assuming a Polyak-Lojasiewicz inequality, namely that there exists $\eta>0$ such that
\begin{equation}\label{eq:LSIgradientflow}
0 \leqslant  f(x) - f(x_*) \leqslant \eta |\na f(x)|^2 
\end{equation} 
in a neighborhood $\mathcal A$ of $x_*$, then by differentiating $f(x_t)-f(x_*)$ we immediately get that, as long as $x_t $ stays within $\mathcal A$, 
  \begin{equation}
  \label{eq:decayFinidim}
  f(x_t)-f(x_*) \leqslant e^{-t/\eta}(f(x_0)-f(x_*))
    \end{equation}
  and
  \begin{equation}
      \label{eq:loca3}
\frac{\dd}{\dd t} \sqrt{f(x_t) - f(x_*)}  \leqslant -  \frac1{2\sqrt{\eta}}|\na f(x_t)|\,,
\end{equation}
  from which
  \begin{equation}
  \label{eq:transportFinidim}
|x_t-x_0| =  \left|\int_0^t \na f(x_s)\dd s \right| \leqslant 2\sqrt{\eta\po f(x_0) - f(x_*)\pf }  \,.
  \end{equation}
Using that $f$ is continuous so that the right hand side can be made arbitrarily small by taking $x_0$ sufficiently close to $x_*$, this shows that $\mathcal A$ contains a ball centered at $x_*$ such that, starting with an initial condition in this ball, the gradient flows remains in $\mathcal A$ and thus the previous inequalities hold for all times $t\geqslant 0$. Assuming furthermore that $x_*$ is the unique critical point of $f$ in $\mathcal A$, by the LaSalle invariance principle, $x_t$ converges to $x_*$ and then letting $t\rightarrow \infty$ in \eqref{eq:transportFinidim} and applying it with $x_0$ replaced by $x_t$ gives
\[|x_*-x_t|^2 \leqslant 4 \eta\po f(x_t) - f(x_*)\pf \leqslant 4 \eta e^{-t/\eta}\po f(x_0) - f(x_*)\pf  \,,\]
which proves the exponential convergence to $x_*$. 

\subsection{\rev{Extension to the Wasserstein space}}

We show that \rev{the previous} argument extends to the infinite dimensional settings of gradient flows over the Wasserstein space  (the lack of continuity of $f$ with respect to the Wasserstein distance in this case being circumvented by the regularization properties of~\eqref{eq:granular_media} which, with the present notations, amounts to say that $x_0 \mapsto f(x_t)$ is continuous at $x_*$ as soon as $t>0$). The key new ingredient is thus the (local) dissipation inequality~\eqref{eq:LSIgradientflow}, which for elliptic McKean-Vlasov equations as \eqref{eq:granular_media} corresponds to a (local) non-linear log-Sobolev inequality (LSI). Such inequalities  have been investigated in a number of works (see e.g. \cite{CMCV,Pavliotis,GuillinWuZhang} and references within) but, to our knowledge, only in cases where they are global, corresponding in the finite-dimensional settings above to the case where   \eqref{eq:LSIgradientflow} holds for all $x\in\R^d$ (for instance when $f$ is uniformly strongly convex). Our main contribution is thus to show that the method also works locally and, more importantly, to show that   it is indeed possible to establish such local dissipation inequalities in some cases where the global inequality fails. 

\rev{
As an illustration in a specific case, let us sum up the main steps of the proof of Proposition~\ref{prop:double-well}.
\begin{enumerate}
    \item The first step is to prove the analogue of~\eqref{eq:LSIgradientflow} in the present context. Since the dissipation of the free energy along the flow reads
    \[\frac{\dd }{\dd t } \mathcal F(\rho_t) = - \sigma^4 \mathcal I\po \rho_t|\Gamma(\rho_t)\pf  \]
    with $\mathcal I$ the Fisher information and $\Gamma(\rho)$ the local equilibrium (see \eqref{eq:Fisher} and \eqref{eq:Gamma}), the goal is to find a constant $\eta>0$ and a set of probability measures $\mathcal A$ such that 
    \begin{equation}
        \label{eq:introPL}
        \forall \rho \in \mathcal A,\qquad \mathcal F(\rho) - \inf_{\mathcal A}\mathcal F \leqslant \eta \mathcal I\po \rho|\Gamma(\rho)\pf\,.
    \end{equation}
    To do so, we notice that $\mathcal F$ can be decomposed as
    \[\mathcal F(\rho) = \mathcal H \po \rho |\Gamma(\rho)\pf + g\po m(\rho)\pf \]
    for some function $g:\R\mapsto\R$ with the mean $m(\rho) = \int_{\R}x \rho(x)\dd x $, where $\mathcal H$ stands for the relative entropy (cf. Section~\ref{sec:decomposition}). Classical arguments show that $\Gamma(\rho)$ satisfies a so-called log-Sobolev inequality, with a constant independent from $\rho$, which shows that $\mathcal H(\rho|\Gamma(\rho)) \leqslant \eta' \mathcal I(\rho|\Gamma(\rho))$ for all $\rho$ for some $\eta'>0$. Hence, it only remains to analyze a $1d$ function $g$. In the context of Proposition~\ref{prop:double-well}, we can actually bound
    \begin{multline*}
     \mathcal F(\rho) - \inf\mathcal F = \mathcal F(\rho) - \mathcal F(\rho_*) = \mathcal H(\rho|\Gamma(\rho) + g(m(\rho)) - g(m(\rho_*))\\ \leqslant \mathcal H(\rho|\Gamma(\rho)) + C |m(\rho) - m(\rho_*)|^2    
    \end{multline*}
    for some $C>0$. The main step is thus to analyze a $1d$ fixed-point problem associated to $g$ leading to a bound of the form
    \begin{equation}
        \label{ref:loceq}
        |m(\rho) - m(\rho_*)| \leqslant  C_\varepsilon |m(\rho) - m\po \Gamma(\rho)\pf |
    \end{equation}
    for some $C_\varepsilon>0$, for  all $\rho$ with $|m(\rho)|\geqslant \varepsilon>0$, for all $\varepsilon>0$ (see Proposition~\ref{prop:doublepuit}). Since $|m(\rho)-m(\Gamma(\rho))|\leqslant \mathcal W_2(\rho,\Gamma(\rho))| \leqslant 2\eta'\sqrt{ \mathcal I(\rho|\Gamma(\rho))}$ thanks to the log-Sobolev inequality satisfied by $\Gamma(\rho)$ (cf. the proof of Proposition~\ref{prop:New_fixed_point} for details), this concludes the proof of~\eqref{eq:introPL} over the set $\mathcal A=\{\rho \in \mathcal P_2(\R^d), |m(\rho)|\geqslant \varepsilon\}$, for an $\varepsilon>0$ smaller than $|m(\rho_*)|$. In the criticial case where $\sigma=\sigma_c$ , the fixed point problem associated to $g$ is degenerate and the  bound~\eqref{ref:loceq} has to be replaced by
    \begin{equation*}
        |m(\rho) - m(\rho_*)| \leqslant  C_\varepsilon \po |m(\rho) - m\po \Gamma(\rho)\pf | +  |m(\rho) - m\po \Gamma(\rho)\pf |^{1/3}\pf\,, 
    \end{equation*}
cf.    Proposition~\ref{prop:doublewelldegenerate}.
\item The Benamou-Brenier formulation of the Wasserstein distance gives the analogue of \eqref{eq:transportFinidim}, namely
\begin{equation}
    \label{eq:locW2}
\mathcal W_2^2(\rho_t,\rho_0) \leqslant  C \po \mathcal F(\rho_0) - \mathcal F(\rho_*)\pf 
\end{equation}
 for some $C>0$, as long as $\rho_t \in \mathcal A$ (cf. Lemma~\ref{lem:localLSI_consequence}).
\item Thanks to the regularization property of the elliptic PDE~\eqref{eq:granular_media}, a Wang-Harnak inequality holds : for some $C>0$,
\[\mathcal F(\rho_1) - \mathcal F(\rho_*) \leqslant C \mathcal W_2^2 (\rho_*,\rho_0)\,, \]
cf. Corollary~\ref{Cor:FW2}. Combining this with~\eqref{eq:locW2} and a simple stability bound $\mathcal W_2(\rho_t,\rho_*) \leqslant C \mathcal W_2(\rho_0,\rho_*)$ for some $C>0$ for $t\in[0,1]$ gives
\[\mathcal W_2^2 (\rho_t,\rho_0) \leqslant C '\mathcal W_2^2 (\rho_*,\rho_0)\,, \]
for some $C'>0$, as long as as $\rho_t\in\mathcal A$. This allows to conclude as in the finite dimensional case with~\eqref{eq:transportFinidim}, despite $\mathcal F$ not being continuous with respect to $\mathcal W_2$. In particular, since $\mathcal A$ contains a Wasserstein ball centered on $\rho_*$, reasoning as in finite dimension, we get that there exists a smaller ball from which solutions started there remain in $\mathcal A$ for all times. Conclusion follows as in the finite dimensional case.
\end{enumerate}
}

\subsection{\rev{Main contributions and organization}}

\rev{To conclude this introduction, let us describe the main results of this work}.
\begin{itemize}
\item In the general framework of gradient flows with respect to $\mathcal W_2$, under suitable regularity conditions, the local non-linear LSI that is the analogue of \eqref{eq:LSIgradientflow} is shown to imply the exponential convergence (in $\mathcal W_2$ and relative entropy) towards the local minimizer for all initial conditions in  a suitable $\mathcal W_2$ ball (this is Theorem~\ref{thm:main}).
\item The same result is established in the kinetic case, i.e. for the Vlasov-Fokker-Planck equation (this is Theorem~\ref{thm:mainVFP}).
\item In the particular case of the granular media equation, when the interaction  is parametrized by some moments of the measure, a simple criterion for the local non-linear LSI is given in Proposition~\ref{prop:New_fixed_point}. It is then illustrated in two cases with quadratic interactions, the one-dimensional double-well case (Proposition~\ref{prop:doublepuit}) and the multi-well case in $\R^d$ (Proposition~\ref{prop:multiwell_new}).
\item Under the same conditions as Theorem~\ref{thm:main} (or Theorem~\ref{thm:mainVFP} in the kinetic case), the free energy of the corresponding system of interacting particles is shown to decay fast below the level of the limit of the non-linear limit, as stated in Proposition \ref{prop:particules_overdamped} (or Proposition~\ref{prop:particules_kinetic} in the kinetic case).
\end{itemize}

This work is organized as follows. The general framework is introduced in Section~\ref{sec:general}, where the main general result (Theorem~\ref{thm:main}) is stated and proven. Section~\ref{sec:LSIparametric} addresses the question of establishing local non-Linear LSI, covering the granular media case (with a detailed study of the one-dimensional double-well potential). The Vlasov-Fokker-Planck equation is studied in Section~\ref{sec:VFP}, and Section~\ref{sec:particules} is devoted to interacting particle systems.

To conclude this introduction, let us mention some perspectives of this work. The gradient descent structure underlying our study is quite flexible, as one may restrict the space over which the free energy is minimized (e.g. tensorized distributions, Gaussian distributions, distributions with some fixed marginals\dots), which is of interest for variational inference \cite{lambert2022variational,Lacker1,Lacker2}, and one can modify the metric with respect to which the gradient is taken, which amounts to add some non-constant (possibly non-linear) diffusion matrix, allowing e.g. for slow or fast diffusion processes and congestion effects \cite{dolbeault2009new,carrillo2022primal}. We expect most  of our analysis to extend to this kind of settings. Non-asymptotic bounds for discrete-time numerical schemes can also   be obtained as in \cite{idealized,Camrud}. Extension to time-varying temperature (for annealing or as a surrogate to stochastic gradient descent \cite{shi2023learning}) is straightforward, following e.g. \cite{journel2022convergence,monmarche2018hypocoercivity} and references within.

Last, let us mention that for McKean-Vlasov equations which do not necessarily write as Wasserstein gradient flows, but under different assumptions than ours, a similar statement to Proposition~\ref{prop:double-well} was recently obtained in~\cite{cormier2022stability}. It provides exponential convergence when the initial condition is in a small $\mathcal{W}_1$ neighborhood of a stationary solution. The method is completely different and relies on the differentiation, in the sense of Lions derivatives, of the drift of the underlying nonlinear SDE in the neighborhood of the considered invariant measure, which then provides a criterion for the stability of the invariant measure. \rev{In the spirit of earlier works such as~\cite{Tamura}, it doesn't give very explicit convergence rates or stability radius.}

\section{Local convergence rates with log-Sobolev inequalities}\label{sec:general}

The main result of this section is Theorem~\ref{thm:main},  stated and proven in Section~\ref{sec:mainthm}. Before that, the relevant  notions, conditions  and lemmas are gradually introduced.

\subsection{General settings, assumptions and notations}\label{sec:settings}

We consider an energy functional $\mathcal E:\mathcal P_2(\R^d) \rightarrow (-\infty,\infty]$ and, for a temperature $\sigma^2>0$, the free energy
\[\mathcal F(\rho) = \mathcal E(\rho) + \sigma^2 \mathcal H(\rho)\]
for $\rho\in\mathcal P_2(\R^d)$, where $\mathcal H(\rho)  = \int_{\R^d} \rho \ln \rho$ stands for the entropy (taken as $+\infty$ if $\rho$ is not absolutely continuous; otherwise we   also write $\rho$ its density). 

\begin{assu}[boundedness from below of the free energy]\label{assu:F}
  The free energy $\mathcal F$ is bounded from below.
\end{assu}

Given a functional $\mathcal G:\mathcal P_2(\R^d) \rightarrow (-\infty,\infty]$, a measurable function $\frac{\delta\mathcal{G}}{\delta\mu}:\mathcal P_2(\R^d)\times\R^d \rightarrow \R$ is called a linear functional derivative of $\mathcal{G}$ if, for all  $\mu_1,\mu_0 \in \mathcal P_2(\R^d)$ with $\mathcal{G}(\mu_0)+\mathcal{G}(\mu_1)<\infty$,
\begin{equation}
    \label{eq:defFunctionDerivative}
\mathcal G(\mu_1) - \mathcal G(\mu_0) = \int_0^1 \int_{\R^d} \frac{\delta  \mathcal G(\mu_t)}{\delta \mu}(x) (\mu_1-\mu_0)(\dd x)\dd t\,,
\end{equation}
where $\mu_t = t\mu_1 +(1-t) \mu_0$. 

\begin{assu}[linear functional derivative of the energy]\label{assu:lfdE}
  $\mathcal E$ admits a linear functional derivative that we denote  $E_{\rho}(x) = \frac{\delta  \mathcal E(\rho)}{\delta \mu}(x)$. Moreover,
  \begin{itemize}
      \item[(i)] the function $(x,\rho) \in \R^d \times \mathcal{P}_2(\R^d) \mapsto E_\rho(x)$ is continuous, where $\mathcal{P}_2(\R^d)$ is endowed with the $\mathcal{W}_2$ distance;
      \item[(ii)] for any $\rho \in \mathcal{P}_2(\R^d)$, $E_\rho \in \mathcal{C}^2(\R^d)$;
      \item[(iii)] for all $\mu_0, \mu_1 \in \mathcal{P}_2(\R^d)$ with $\mathcal{E}(\mu_0) + \mathcal{E}(\mu_1) < \infty$, $x \mapsto \sup_{t \in [0,1]} |E_{\mu_t}(x)|$ is in $L^1(\mu_0)\cap L^1(\mu_1)$.
  \end{itemize}
\end{assu}
The last item of Assumption~\ref{assu:lfdE} ensures that the integral in the right hand side of~\eqref{eq:defFunctionDerivative} is well defined when $\mathcal G=\mathcal E$, and moreover that
\begin{equation}\label{eq:cv-lfdE}
  \lim_{t \to 0} \frac{\mathcal{E}(\mu_0 + t(\mu_1-\mu_0))-\mathcal{E}(\mu_0)}{t} = \int_{\R^d} E_{\mu_0}(x)(\mu_1-\mu_0)(\dd x).
\end{equation}

We are interested in the McKean-Vlasov equation
\begin{equation}
\label{eq:granular_media_1}
\partial_t \rho_t = \sigma^2\Delta  \rho_t +   \na\cdot \po \rho_t\na E_{\rho_t}\pf\,,
\end{equation}
equivalently 
\begin{equation}
\label{eq:granular_media_2}
\partial_t \rho_t =  \na\cdot \po \rho_t \na  \frac{\delta\mathcal{F}(\rho_t) }{\delta\mu} \pf\,.
\end{equation}
In particular, the granular media equation  \eqref{eq:granular_media} corresponds to the free energy
\begin{equation}
\label{eq:def_F}
\mathcal{F}(\rho) =  \sigma^2 \int_{\R^d} \rho \ln \rho + \int_{\R^d} V\rho + \frac12 \int_{\R^d\times\R^d}   W \rho^{\otimes 2}\,,
\end{equation} 
for which $E_{\rho}(x) = V(x) + \rho \star W(x)$ and Assumptions~\ref{assu:F} and~\ref{assu:lfdE} are satisfied as soon as, for instance, $V$ is convex outside a compact set and $W$ is lower bounded with at most a quadratic growth at infinity (other conditions can be considered, for instance in repulsive cases $W$ may not be bounded below but Assumption~\ref{assu:F} still holds if $V$ grows faster than $-W$ at infinity)   Under suitable regularity  condition\rev{s}, a straightforward computation gives
\begin{equation}\label{eq:dissipationF}
\frac{\dd}{\dd t} \mathcal F( \rho_t)  = -  \int_{\R^d} \left|\na \frac{\delta\mathcal{F}(\rho_t)}{\delta\mu}\right|^2 \dd \rho_t  \,,
\end{equation}
\rev{see e.g. \cite[Proposition 2.1]{CMCV}.} In particular, the free energy is decreasing along time, and the stationary solutions $\rho_*$ are the critical points of $\mathcal F$, characterized by the fact that  $\frac{\delta\mathcal{F}}{\delta\mu}(\rho_*)$ is constant, which, using that  $\frac{\delta\mathcal{H}}{\delta\mu}(\rho) = \ln \rho$,  is equivalent to the self-consistency equation
\begin{equation}
\label{eq:self_consistency}
\rho_* \propto \exp\po - \frac{1}{\sigma^{2}} E_{\rho_*}\pf\,.
\end{equation}
In cases where $\mathcal F$ satisfies a global non-linear log-Sobolev inequality, in the sense that there exists a constant $\oeta>0$ such that 
\begin{equation}
\label{eq:LSInonlineaire}
\forall \rho\in\mathcal P_2(\R^d)\,,\qquad \mathcal F( \rho) - \inf \mathcal F  \leqslant \oeta   \int_{\R^d} \left|\na \frac{\delta\mathcal{F}(\rho)}{\delta\mu}\right|^2 \dd \rho \,,
\tag{\textbf{G-NL-LSI}}
\end{equation}
 we immediately get from \eqref{eq:dissipationF} that  the free energy decays exponentially fast toward its infimum. In particular such an inequality is clearly false when $\mathcal F$ admits critical points which are not global minimizers, which is precisely the case we are interested in.
 
 \begin{assu}[local equilibrium]\label{assu:loc-eq}
    For any $\rho \in \mathcal{P}_2(\R^d)$, 
    \begin{equation*}
        Z_\rho = \int_{\R^d} \exp\po - \frac1{\sigma^{2}}E_\rho \pf < \infty.
    \end{equation*}
\end{assu}
For the granular media equation~\eqref{eq:granular_media} with free energy given by~\eqref{eq:def_F}, this assumption holds  if for instance $W$ is lower bounded and $\ln|x| = o(V(x))$ at infinity.  Under Assumption~\ref{assu:loc-eq}, denoting
 \begin{equation}
 \label{eq:Gamma}
 \Gamma(\rho) = Z_\rho^{-1} \exp\po - \frac1{\sigma^{2}}E_\rho \pf\,,
 \end{equation}
which we call the local equilibrium as it is the stationary solution of the linear equation
\[\partial_t \tilde \rho_t = \sigma^2 \Delta \tilde \rho_t + \na\cdot \po \tilde \rho_t \na  E_\rho   \pf \,, \]
we can also interpret the free energy dissipation in \eqref{eq:dissipationF} as 
\[\int_{\R^d} \left|\na \frac{\delta\mathcal{F}}{\delta\mu}(\rho)\right|^2 \dd \rho = \sigma^4 \mathcal I\po \rho|\Gamma(\rho)\pf\,,\]
where $\mathcal I(\nu|\mu)$ stands for the Fisher information of $\nu$ with respect to $\mu$,
\begin{equation}
\label{eq:Fisher} 
    \mathcal I(\nu|\mu) = \int_{\R^d} \left|\na \ln \frac{\dd \nu}{\dd \mu}\right|^2 \dd \nu\,.
\end{equation}

 We write
\[\mathcal K = \{\rho_* \in\mathcal P_2(\R^d),\ \mathcal F(\rho_*)<\infty,\  \rho_* = \Gamma(\rho_*)\}\]
the set of critical points of $\mathcal F$.

 In this work we are interested in cases where the local equilibria satisfy a uniform (classical) log-Sobolev inequality, in the sense that there exist $\eta >0$ such that 
\begin{equation}
\label{eq:LSIlineaire_locale}
\forall \nu,\rho\in\mathcal P_2(\R^d)\,,\qquad \mathcal H\po \nu|\Gamma(\rho)\pf  \leqslant \eta  \mathcal I\po \nu|\Gamma(\rho)\pf \,,
\tag{\textbf{U-LSI}}
\end{equation}
where
\[\mathcal H(\nu|\mu) = \int_{\R^d} \ln \frac{\dd\nu}{\dd \mu} \dd \nu\]
stands for the relative entropy of $\nu$ with respect to $\mu$. Indeed, contrary to \eqref{eq:LSInonlineaire}, there are many tools to establish \eqref{eq:LSIlineaire_locale} and it can typically hold even if $\mathcal F$ has several critical points. For instance, for the granular media equation \eqref{eq:granular_media} in the double well case \eqref{eq:double_well} with attractive interaction (namely $\theta>0$), we can decompose $V=V_0+V_1$ where $V_0$ is strongly convex and $V_1$ is bounded, so that,  for any $\rho$,
\[V  + \rho \star W   = V_0 + V_1 + \rho\star W\]
is the sum of a strongly convex potential $V_0+\rho\star W$ (with a lower bound on the curvature independent from $\rho$) and a bounded potential $V_1$ (independent from $\rho$), so that \eqref{eq:LSIlineaire_locale} follows from classical Bakry-Emery and Holley-Stroock results \cite{BakryGentilLedoux}. One motivation of the present work is to understand the difference between the uniform classical log-Sobolev inequality \eqref{eq:LSIlineaire_locale} and the  non-linear log-Sobolev inequality \eqref{eq:LSInonlineaire}.

\medskip

 The well-posedness for \eqref{eq:granular_media_2} for specific cases is a standard question that we do not address here, we refer the interested reader to \cite{ambrosio2005gradient} for general considerations on Wasserstein gradient flows.

 \begin{assu}[well-posedness and regularity of~\eqref{eq:granular_media_2}]\label{assu:pde}
 For all $\rho_0\in\mathcal P_2(\R^d)$, \eqref{eq:granular_media_2} has a unique strong solution, continuous in time for $\mathcal W_2$, which is the gradient flow of $\mathcal F$ in the sense of \cite[Definition 11.1.1]{ambrosio2005gradient}; for all $t>0$, $\rho_t$ has a continuous positive density, $\mathcal F(\rho_t)$ and $\mathcal H(\rho_t|\Gamma(\rho_t))$ are finite; for almost all times $t>0$, $\mathcal I(\rho_t|\Gamma(\rho_t))$ is finite and \eqref{eq:dissipationF} holds. 
 \end{assu}
 
The fact that the free energy, local relative entropy and Fisher information become finite instantaneously will be stated in more quantitative ways along the study. In particular, the results stated in Section~\ref{sec:auxiliary}, combined with the next statement   from~\cite{ambrosio2005gradient}, show  that Assumption~\ref{assu:pde} is met in particular in the granular media case~\eqref{eq:granular_media} in the settings of Proposition~\ref{prop:double-well}.

\begin{prop}
\label{prop:Ambrosio}
In the case~\eqref{eq:def_F} with $\nabla ^2 V \geq \lambda I_d$ for some $\lambda \in \R$, $W(x,y) = w(x-y)$ with an even and convex function $w : \R^d \to [0,\infty)$ which has the doubling property $w(x+y) \leq c_w(1+w(x)+w(y))$, the following holds. For any initial condition $\rho_0 \in \mathcal{P}_2(\R^d)$, there is a unique distributional solution to~\eqref{eq:granular_media}, which is the gradient flow of $\mathcal F$ and  possesses the following properties:
\begin{enumerate}[label=(\roman*)]
    \item for any $t > 0$, $\rho_t$ has a density with respect to the Lebesgue measure on $\R^d$;
    \item $\rho_t \to \rho_0$ when $t \to 0$ in $\mathcal{P}_2(\R^d)$;
    \item $\rho_\cdot \in L^1_\mathrm{loc}((0,\infty),W^{1,1}_\mathrm{loc}(\R^d))$; \item for any $0 < s < t < \infty$,
    \begin{equation*}
        \mathcal F(\rho_s) = \mathcal F(\rho_t) + \sigma^4 \int_s^t \mathcal{I}(\rho_r|\Gamma(\rho_r)) \dd r < \infty.
    \end{equation*}
\end{enumerate}
\end{prop}
\begin{proof}
With the present conditions on $V$ and $W$ the assumptions of~\cite[Example~11.2.7]{ambrosio2005gradient} are satisfied, so that \cite[Theorem~11.2.8]{ambrosio2005gradient} applies, which gives all the points of Proposition~\ref{prop:Ambrosio}.
\end{proof}


\medskip

From now on, Assumptions~\ref{assu:F}, \ref{assu:lfdE}, \ref{assu:loc-eq} and~\ref{assu:pde} are systematically enforced.

\subsection{Two known cases with global convergence}\label{sec:convex_interaction_faible}

To further motivate our study, let us now highlight that it is known that \eqref{eq:LSIlineaire_locale} is in fact sufficient to conclude (concerning the long-time behavior of the equation) in two cases: when $\mathcal E$ is functional-convex, or when the interaction is sufficiently small (two cases which, of course, do not allow for multiple stationary solutions).

\medskip

We start with functional-convexity (to be distinguished from displacement-convexity), which by definition means that
\begin{equation}
\label{eq:convexity}
\mathcal E(t\mu_0 + (1-t) \mu_1) \leqslant t \mathcal E(\mu_0) +(1-t)\mathcal E(\mu_1)\qquad t\in[0,1]\,,
\end{equation}
for all  $\mu_0,\mu_1 \in \mathcal P_2(\R^d)$. As stated in e.g. \cite{chizat}, assuming furthermore that $\mathcal F$ admits a minimizer $\rho_*\in\mathcal K$ then the latter is unique (thanks to the strict convexity of the entropy), and for all $\rho \in \mathcal P_2(\R^d)$, the following entropy sandwich inequalities hold:
\begin{equation}
\label{eq:sandwich}
\sigma^2 \mathcal H(\rho|\rho_*) \leqslant \mathcal F(\rho) - \mathcal F(\rho_*) \leqslant \sigma^2\mathcal H \po \rho|\Gamma(\rho)\pf\,,
\end{equation}
see also Lemma~\ref{lem:sandwich} below.  In particular, the second inequality, together with  \eqref{eq:LSIlineaire_locale}, implies \eqref{eq:LSInonlineaire} with $\oeta = \eta/\sigma^2$, hence the exponential decay of $\mathcal F(\rho_t)$ to its global minimum which, by the first inequality of \eqref{eq:sandwich}, gives the exponential convergence of $\mathcal H(\rho_t|\rho_*)$ to zero (which in turns implies the exponential convergence of $\rho_t$ to $\rho_*$ in total variation and, because $\rho_*$ satisfies a classical LSI thanks to \eqref{eq:LSIlineaire_locale} and thus a Talagrand inequality  \cite{OttoVillani}, in Wasserstein~2 distance). 

The convexity condition \eqref{eq:convexity} is known to hold in various settings, in particular for mean-field models of one-layer neuron networks \cite{Szpruch,chizat}\rev{, where
\begin{equation}
    \label{eq:neurones}
    \mathcal E\po \mu \pf = \int_{\R^{n+m}} \left| z - \int_{\R^d}\Phi(x,y)\mu(\dd x) \right|^2 m(\dd y\dd z) + c\int_{\R^d} |x|^2 \mu(\dd x)\,,  
\end{equation}
where $m$ is the distribution of the data ($y\in\R^n$ being the input, $z\in\R^m$ the output), $x$ is the parameter of a single neuron, $\Phi(x,\cdot)$ is the map parametrized by a neuron with parameters $x$, $\mu$ is the distribution of the parameters in the neurons of the layer, and $c>0$ is a penalization parameter.
}
In the granular media case \eqref{eq:def_F}, assume that the interaction potential is of the form
\begin{equation}
\label{eq:Wa_kr_k}
W(x,y) = W_0(x) + W_0(y) +2\sum_{k\in\N} r_k(x)r_k(y) -2 \sum_{k\in\N} a_k(x)a_k(y)
\end{equation}
for some   functions $W_0,a_k,r_k$ (as in e.g. the quadratic case \eqref{eq:double_well}; the letters $a$ and $r$ refers to \emph{attractive} and \emph{repulsive} by analogy with the quadratic case). Then we see that, for $\mathcal E(\rho)= \int (V + \frac12 \rho \star W)\rho$ and $t\in[0,1]$,
\begin{multline}\label{eq:courbureWa_kr_k}
\mathcal E(t\mu_0 + (1-t) \mu_1) -  t \mathcal E(\mu_0) - (1-t)\mathcal E(\mu_1) \\
= t(1-t) \sum_{k\in\N} \co \po \int_{\R^d} a_k (\mu_0-\mu_1)\pf^2  - \po \int_{\R^d} r_k (\mu_0-\mu_1)\pf^2\cf\,,
\end{multline}
so that \eqref{eq:convexity}  holds for all $\mu_0,\mu_1,t$  if and only if $a_k=0$ for all $k\in\N$, which corresponds to repulsive interaction (and for quadratic potentials \eqref{eq:double_well} to the case where $\theta\leqslant 0$, i.e. $W$ is concave). See also \cite[Section 3.1]{Songbo} for further examples satisfying \eqref{eq:convexity}.

In the general non-functional-convex cases, assuming that $\rho_*\in\mathcal K$ is a global minimizer of $\mathcal F$,   notice that the first inequality in \eqref{eq:sandwich}  cannot hold if there is another global minimizer $\rho'\in\mathcal K$ (since we would have $\mathcal H(\rho'|\rho_*)>0$ while $\mathcal F(\rho')=\mathcal F(\rho_*)$) and the second inequality (which is the one we used to get the global non-linear LSI \eqref{eq:LSInonlineaire} from the uniform classical LSI \eqref{eq:LSIlineaire_locale}) cannot hold as soon as there is a critical point of $\mathcal F$ which is not a global minimizer (since $\mathcal H(\rho|\Gamma(\rho))=0$ when $\rho\in\mathcal K$). 

\medskip

Second, consider the non-convex but small Lipschitz interaction settings. We assume that  $\mu \mapsto \na E_\mu$ is uniformly Lipschitz, namely there exists $L>0$ such that
\begin{equation}
\label{eq:Lipshitz}
\forall \rho,\rho'\in\mathcal P_2(\R^d)\,,\qquad \left\|\na E_{\rho} - \na E_{\rho'}\right\|_\infty \leqslant L \mathcal W_2(\rho,\rho')\,.\tag{\textbf{Lip}}
\end{equation}  
 We still assume   \eqref{eq:LSIlineaire_locale}, which by \cite{OttoVillani} implies  the uniform $T_2$ Talagrand inequality
\begin{equation}
\label{eq:NLT2}
\forall \nu,\rho\in\mathcal P_2(\R^d)\,,\qquad \mathcal W_2^2\po \nu, \Gamma(\rho)\pf \leqslant 4 \eta \mathcal H\po \nu|\Gamma(\rho)\pf\,.
\end{equation}
 Let $\rho_* \in\mathcal K$.  A classical computation shows that
\[\frac{\dd}{\dd t} \mathcal H(\rho_t |\rho_*) = -\sigma^2 \mathcal I(\rho_t|\rho_*) + \int_{\R^d} \po \na E_{\rho_*}- \na E_{\rho_t} \pf \cdot \na \ln\po  \frac{\rho_t}{\rho_*}\pf  \dd\rho_t\,.\]
 Using the Cauchy-Schwarz, log-Sobolev  and Talagrand inequalities,
\begin{eqnarray*}
\frac{\dd}{\dd t} \mathcal H(\rho_t |\rho_*) &\leqslant & -\sigma^2 \mathcal I(\rho_t|\rho_*) + \left\| \na E_{\rho_*}- \na E_{\rho_t}\right\|_\infty \sqrt{ \mathcal I(\rho_t|\rho_*)}\\
& \leqslant &  -\frac{\sigma^2}2 \mathcal I(\rho_t|\rho_*) + \frac{L^2}{2\sigma^2}   \mathcal W_2^2(\rho_t,\rho_*)\\
& \leqslant & \po -\frac{\sigma^2}{2\eta} + \frac{2 L^2\eta }{\sigma^2}\pf \mathcal H(\rho_t|\rho_*)\,.
\end{eqnarray*}
 In particular, as soon as $L\eta <2 \sigma^2$, we get an exponential convergence in relative entropy (hence total variation and $\mathcal W_2$) of $\rho_t$ to $\rho_*$, and in particular uniqueness of the critical point $\rho_*$.

\subsection{Lower-bounded curvature}

To revisit and generalize the two previous cases, let us consider the case where there exists a cost functional $\mathcal C:\mathcal P_2(\R^d)\times\mathcal P_2(\R^d)\rightarrow (-\infty,\infty]$ with $\mathcal C(\mu_0,\mu_1)=\mathcal C(\mu_1,\mu_0)$ and such that 
\begin{multline}
\label{eq:convexity-c}
\forall \mu_0,\mu_1\in\mathcal P_2(\R^d),t\in[0,1],\\
\mathcal E(t\mu_0 + (1-t) \mu_1) \leqslant t \mathcal E(\mu_0) +(1-t)\mathcal E(\mu_1) + t(1-t)\mathcal C(\mu_0,\mu_1) \,.\qquad
\tag{$\mathcal C$\textbf{-curv-}$\mathcal E$}
\end{multline}
Functional-convexity corresponds to $\mathcal C=0$. If $\mathcal C = \lambda \mathcal W_2^2$ for some $\lambda\in\R$, namely if
\begin{multline}
\label{eq:convexity-c_W2}
\forall \mu_0,\mu_1\in\mathcal P_2(\R^d),t\in[0,1],\\
\mathcal E(t\mu_0 + (1-t) \mu_1) \leqslant t \mathcal E(\mu_0) +(1-t)\mathcal E(\mu_1) + \lambda t(1-t)\mathcal W_2^2(\mu_0,\mu_1) \,,\qquad
\tag{$\mathcal W_2^2$\textbf{-curv-}$\mathcal E$}
\end{multline}
 it  can be interpreted in terms of Otto calculus by saying that the Hessian of $\mathcal E$ is lower-bounded by $-\lambda$. In the case \eqref{eq:Wa_kr_k} a bound \eqref{eq:convexity-c} follows from \eqref{eq:courbureWa_kr_k}, and if $a_k$ is $L_k$-Lipschitz for all $k\in\N$ with $\lambda=\sum_{k\in\N} L_k^2 <\infty$, \eqref{eq:convexity-c_W2} holds (in fact, a stronger inequality holds, with $\mathcal W_2$ replaced by $\mathcal W_1$, which could be of interest in some cases; in this work we use $\mathcal W_2$ everywhere for simplicity).

\begin{Example}
\new{Consider the Gaussian kernel interaction potential $ W(x,y) = a e^{-|x-y|^2} $ for some $a\in\R$, which appears for instance in the Adaptive Biasing Potential method \cite{ABP} and more generally in regularized approximations of processes which are influenced by the local density of particles (e.g. \cite{li2022sampling}). It is attractive for $a<0$ and repulsive for $a>0$. We can write it in the form \eqref{eq:Wa_kr_k}  by decomposing
\[W(x,y) =   a e^{-|x|^2-|y|^2}\prod_{i=1}^d \sum_{k\in\N}  \frac{(2x_i y_i)^k}{k!} = a \sum_{k\in\N^d} w_k(x)w_k(y)  \,,\]
with $w_k(x) = e^{-|x|^2}\prod_{i=1}^d \frac{2^{k_i/2}x_i^{k_i}}{\sqrt{k_i!}} $, so that \eqref{eq:convexity-c} holds with $\mathcal C(\mu_0,\mu_1) = 0$ if $a\geqslant 0$ and, if $a<0$,
\[\mathcal C(\mu_0,\mu_1)  = \frac{|a|}2 \sum_{k\in\N^d} \po \int_{\R^d} w_k (\mu_0-\mu_1)\pf^2 \,.\]
Although it can be checked that $\sum_{k\in\N^d} \|\na w_k\|_\infty^2 =+\infty$,  we can get \eqref{eq:convexity-c_W2} as follows: given $(X,Y)$ a coupling of $\mu_0$ and $\mu_1$,
\[
\mathcal C(\mu_0,\mu_1)  = \frac{|a|}2 \sum_{k\in\N^d} \po \mathbb E \po w_k(X)-w_k(Y) \pf\pf ^2 \leqslant \frac{|a|}2 \mathbb E \po h(X,Y)\pf \] 
with
\begin{eqnarray*}
h(x,y) & = & \sum_{k\in\N^d} |w_k(x)-w_k(y)|^2  \\
&= & \sum_{k\in\N^d}  \po w_k^2(x) + w_k^2(y) - 2 w_k(x)w_k(y)\pf \\
& = & 2  - 2 e^{-|x-y|^2 } \ \leqslant \ 2|x-y|^2\,. 
\end{eqnarray*}
Taking the infimum over the couplings $(X,Y)$ gives $\mathcal C(\mu_0,\mu_1) \leqslant |a| \mathcal W_2^2(\mu_0,\mu_1)$.
}
\end{Example} 

\begin{Example}\label{ex:hilbert}
Considering the settings of \cite{chizat2018global}, let $\hil$ be a Hilbert space with norm $\|\cdot\|_{\hil}$, $\varphi:\R^d\rightarrow \hil$, $R:\hil \rightarrow \R$ and $V_0:\R^d\rightarrow \R$. Consider 
\[\mathcal E(\mu) = \int_{\R^d} V_0(x) \mu(\dd x) +  R\po \int_{\R^d} \varphi(x) \mu(\dd x)\pf\,,\]
where the integral of an $\hil$-valued function is understood in Bochner's sense, and we take the convention that $\mathcal{E}(\mu)=\infty$ if $\int_{\R^d} \|\varphi(x)\|_{\hil}\mu(\dd x)=\infty$. As discussed in \cite{chizat2018global}, this encompasses many cases of interest in optimization, machine learning and stastistics in high dimension. Assume that the curvature of $R$ is lower-bounded in the sense that there exists $\theta>0$ such that $  R+\theta \|\cdot\|_{\hil}^2$ is convex. This  implies  that \eqref{eq:convexity-c} holds with 
\[\mathcal C(\mu_0,\mu_1)  = \theta \left\|\int_{\R^d} \varphi \dd  \mu_0 - \int_{\R^d} \varphi \dd  \mu_1\right\|^2_{\hil}\,.\]
This generalizes the quadratic finite-dimensional case \eqref{eq:courbureWa_kr_k}. Assuming furthemore that $\varphi$ is Lipschitz continuous (namely $\|\varphi(x)-\varphi(y)\|_{\hil} \leqslant \ell |x-y|$ for all $x,y\in\R^d$), we get \eqref{eq:convexity-c_W2} with $\lambda = \theta\ell^2$. 
\end{Example}

  The following generalizes the entropy sandwich inequality \eqref{eq:sandwich} when $\mathcal C\neq 0$.
 
 \begin{lem}
 \label{lem:sandwich}
 Under \eqref{eq:convexity-c}, for all $\mu_0,\mu_1\in\mathcal P_2(\R^d)$ and $\rho_*\in\mathcal K$,
 \begin{equation}
 \mathcal F(\mu_0)   \leqslant \mathcal F(\mu_1) +  \sigma^2 \mathcal H\po \mu_0|\Gamma(\mu_0)\pf+   \mathcal C(\mu_0,\mu_1)\,,\label{eq:Fmu0Fmu1<}
 \end{equation}
 and
 \begin{equation}
\label{eq:Fmu0Fmu1>}
\mathcal F(\rho_*) +  \sigma^2 \mathcal H\po \mu_0|\rho_*\pf  \leqslant 
\mathcal F(\mu_0)  + \mathcal C(\mu_0,\rho_*)    \,.
\end{equation}
 \end{lem}
 \begin{proof} To prove each inequality, we assume that the right-hand side is finite, otherwise the result is trivial.
 
   Dividing by $t$ and sending $t$ to zero in \eqref{eq:convexity-c} yields (using~\eqref{eq:cv-lfdE}),
\[\int_{\R^d} E_{\mu_1} (\mu_0-\mu_1) + \mathcal E(\mu_1) \leqslant  \mathcal E(\mu_0)  +     \mathcal C(\mu_0,\mu_1)\,,
\]
and thus, switching the roles of $\mu_1$ and $\mu_0$,
\[\int_{\R^d} E_{\mu_1} (\mu_0-\mu_1) -     \mathcal C(\mu_0,\mu_1) \leqslant  \mathcal E(\mu_0) -  \mathcal E(\mu_1) \leqslant \int_{\R^d} E_{\mu_0} (\mu_0-\mu_1) +     \mathcal C(\mu_0,\mu_1)\,.
\]
Then
\begin{eqnarray*}
\mathcal F(\mu_0) -  \mathcal F(\mu_1) & \leqslant  & \int_{\R^d} E_{\mu_0} (\mu_0-\mu_1) +     \mathcal C(\mu_0,\mu_1) + \sigma^2 \mathcal H  (\mu_0) - \sigma^2 \mathcal H(\mu_1)\nonumber\\
& = &  \sigma^2 \mathcal H\po \mu_0|\Gamma(\mu_0)\pf - \sigma^2 \mathcal H\po \mu_1|\Gamma(\mu_0)\pf +   \mathcal C(\mu_0,\mu_1) \nonumber\\
& \leqslant & \sigma^2 \mathcal H\po \mu_0|\Gamma(\mu_0)\pf+   \mathcal C(\mu_0,\mu_1)\,,
\end{eqnarray*}
and
\begin{eqnarray*}
\mathcal F(\mu_0) -  \mathcal F(\mu_1) & \geqslant  & \int_{\R^d} E_{\mu_1} (\mu_0-\mu_1) -     \mathcal C(\mu_0,\mu_1) + \sigma^2 \mathcal H  (\mu_0) - \sigma^2 \mathcal H(\mu_1)\\
& = &  \sigma^2 \mathcal H\po \mu_0|\Gamma(\mu_1)\pf - \sigma^2 \mathcal H\po \mu_1|\Gamma(\mu_1)\pf -  \mathcal C(\mu_0,\mu_1) \,.
\end{eqnarray*}
The second term of the right hand side vanishes if $\mu_1\in\mathcal K$, which concludes.
 \end{proof}
 
This result will prove useful in our cases of interest, namely when $\mathcal F$ has several critical points. Nevertheless, for now, in the spirit of the previous section, let us discuss how \eqref{eq:LSIlineaire_locale} may already give a global convergence in some cases under \eqref{eq:convexity-c}.

Indeed, assuming a Talagrand type inequality
\[\mathcal C(\mu_0,\rho_*) \leqslant \eta' \mathcal H(\mu_0|\rho_*)\]
for some $\eta'>0$ (which, if $\mathcal C \leqslant \lambda \mathcal W_2^2$ for some $\lambda>0$, is implied by  \eqref{eq:LSIlineaire_locale} with $\eta'=4\eta\lambda$), using \eqref{eq:Fmu0Fmu1>} and then \eqref{eq:Fmu0Fmu1<},
\begin{eqnarray*}
\mathcal C(\mu_0,\rho_*)  \leqslant  \eta' \mathcal H(\mu_0|\rho_*) & \leqslant &  \frac{\eta'}{\sigma^2}  \po \mathcal F(\mu_0) -  \mathcal F(\rho_*) +   \mathcal C(\mu_0,\rho_*) \pf \\
& \leqslant & \eta' \mathcal H\po \mu_0|\Gamma(\mu_0)\pf +  \frac{2\eta'}{\sigma^2}     \mathcal C(\mu_0,\rho_*) \,.
\end{eqnarray*}
Hence, under the condition $2\eta ' < \sigma^2$ (i.e., again, high temperature or small interaction), we obtain a non-linear transport inequality
\begin{equation}
\label{eq:non-lin-transport}
\mathcal C(\mu_0,\rho_*) \leqslant \frac{\eta' }{1-2\eta '/\sigma^2 }\mathcal H\po \mu_0|\Gamma(\mu_0)\pf \,.
\end{equation}
Together with the LSI for $\rho_*$ and \eqref{eq:Fmu0Fmu1<}, this proves \eqref{eq:LSInonlineaire}, since
\[\mathcal F(\mu_0) -  \mathcal F(\rho_*) \leqslant \po \sigma^2 + \frac{\eta ' }{1-2\eta '/\sigma^2 }\pf  \mathcal H\po \mu_0|\Gamma(\mu_0)\pf  \leqslant \eta \po \sigma^2 + \frac{\eta ' }{1-2\eta' /\sigma^2 }\pf  \mathcal I\po \mu_0|\Gamma(\mu_0)\pf\,. \]
Notice that, conversely, the non-linear transport inequality \eqref{eq:non-lin-transport} with $\mathcal C=\mathcal W_2^2$ is shown to be a consequence of \eqref{eq:LSInonlineaire} in \cite{Pavliotis} (see also Lemma~\ref{lem:localLSI_consequence} below).

\subsection{Local non-linear log-Sobolev inequalities}

Consider a non-empty set $\mathcal A\subset \mathcal P_2(\R^d)$ with  $\inf_\mathcal{A} \mathcal{F}<\infty$. Instead of the global condition~\eqref{eq:LSInonlineaire}, the key ingredient in our approach is to obtain (local) non-linear LSI of the form
\begin{equation}
\label{eq:loc_LSInonlineaire}
\forall \rho\in\mathcal A\,,\qquad \mathcal F( \rho) - \inf_{\mu\in\mathcal A} \mathcal F(\mu)  \leqslant  \oeta \sigma^4  \mathcal I\po \rho|\Gamma(\rho)\pf 
\tag{\textbf{NL-LSI}}
\end{equation}
for some $\oeta>0$. The fact that it is possible to get this in some subsets $\mathcal A$ even in cases where $\mathcal F$ has several critical points is the topic of Section~\ref{sec:LSIparametric}. For now we assume that this holds in some   set $\mathcal A$ and we discuss its consequences.

First, since the left hand side of  \eqref{eq:loc_LSInonlineaire} is non-negative and the  right hand side vanishes when $\rho \in \mathcal K$, we get that
\begin{equation}
\label{eq:rho*min}
\rho_* \in \mathcal A \cap \mathcal K \qquad \Rightarrow \qquad \mathcal F(\rho_*) = \inf_{\mu\in\mathcal A} \mathcal F(\mu) \,.
\end{equation}
Second, the results which were known to hold under \eqref{eq:LSInonlineaire} still hold as long as the solution of \eqref{eq:granular_media_2} stays within $\mathcal A$, as we state in the next lemma. For conciseness we write
\[\mathcal F_{\mathcal A}(\rho) = \mathcal F( \rho) - \inf_{\mu\in\mathcal A} \mathcal F(\mu)\,.\]
For $\rho_0\in\mathcal P_2(\R^d)$, we write
\[T_{\mathcal A}(\rho_0) = \inf\{t\geqslant 0,\ \rho_t\notin \mathcal A\}\,,\]
with $(\rho_t)_{t\geqslant 0}$ the solution to \eqref{eq:granular_media_2}. The next lemma is the infinite-dimensional version of \eqref{eq:decayFinidim} and \eqref{eq:transportFinidim} and follows the same proof.

\begin{lem}\label{lem:localLSI_consequence}
Assume that \eqref{eq:loc_LSInonlineaire} holds on $\mathcal A$. Then, for all $\rho_0\in\mathcal A$ and $t\leqslant T_{\mathcal A}(\rho_0)$,
\begin{equation}
\label{eq:lem-F-decay}
\mathcal F_{\mathcal A}(\rho_t) \leqslant e^{-t/\oeta } \mathcal F_{\mathcal A}(\rho_0)
\end{equation}
and
\begin{equation}
\label{eq:lem-W2transport}
\mathcal W_2^2(\rho_t,\rho_0) \leqslant 4 \oeta \mathcal F_{\mathcal A}(\rho_0)\,.
\end{equation}
\end{lem}
\begin{proof}
The first inequality is straightforward from \eqref{eq:dissipationF} and \eqref{eq:loc_LSInonlineaire}. The proof of the second inequality is the same as the proof of \cite[Theorem 3.2]{Pavliotis} (although a factor 2 seems to disappear wrongly in the latter when using the Benamou-Brenier formula, which explains the difference with our result; notice that our factor 4 is consistent with the classical linear case \cite{OttoVillani}, or the finite-dimensional settings \eqref{eq:transportFinidim})\rev{, as we now detail}. Since \eqref{eq:granular_media_2} can be written as the continuity equation
\[\partial_t \rho_t = \nabla \cdot \po \rho_t v_t\pf\qquad v_t := \na  \frac{\delta\mathcal{F}(\rho_t) }{\delta\mu} \,,\]
by the Benamou-Brenier formulation of the Wasserstein distance~\cite{BenamouBrenier}, for any $0 \leq s < t$,
\begin{equation*}
    \mathcal{W}_2^2(\rho_s,\rho_t) \leq (t-s) \int_{s}^t \int_{\R^d} |v_r|^2 \rho_r \dd r = \sigma^4(t-s)\int_s^t \mathcal{I}(\rho_r|\Gamma(\rho_r))\dd r.
\end{equation*}
Thus,
\begin{equation*}
    \frac{\mathcal{W}_2(\rho_s,\rho_t)}{t-s} \leq \sigma^2 \sqrt{\frac{1}{t-s}\int_s^t \mathcal{I}(\rho_r|\Gamma(\rho_r))\dd r},
\end{equation*}
so by Lebesgue's differentiation theorem, the metric derivative of $(\rho_t)_{t \geq 0}$ defined in~\cite[Theorem~1.1.2]{ambrosio2005gradient} satisfies
\begin{equation*}
    |\rho'_t| \leq \sigma^2 \sqrt{\mathcal{I}(\rho_t|\Gamma(\rho_t))},
\end{equation*}
$\dd t$-almost everywhere. Since, by~\cite[Theorem~1.1.2]{ambrosio2005gradient}, we have on the other hand
\begin{equation*}
    \mathcal{W}_2(\rho_s,\rho_t) \leq \int_s^t |\rho'_r| \dd r,
\end{equation*}
we get
\[\mathcal W_2(\rho_0,\rho_t)  \leqslant  \sigma^2 \int_0^t \sqrt{ \mathcal I\po \rho_s|\Gamma(\rho_s)\pf  }\dd s\,. \] 
Besides, for $t\leqslant T_{\mathcal A}(\rho_0)$, 
\[
\frac{\dd}{\dd t} \mathcal F_{\mathcal A}(\rho_t)  =  - \sigma^4 \mathcal I\po \rho_t|\Gamma(\rho_t)\pf \\
 \leqslant  -  \sqrt{ \sigma^ 4 \mathcal I\po \rho_t|\Gamma(\rho_t)\pf \mathcal F_{\mathcal A}(\rho_t) /\oeta  }\,,\]
 from which
\[ \sqrt{\mathcal F_{\mathcal A}(\rho_t)} - \sqrt{\mathcal F_{\mathcal A}(\rho_0)}  \leqslant  -\frac{\sigma^2}{2\sqrt{\oeta}} \int_0^t \sqrt{ \mathcal I\po \rho_s|\Gamma(\rho_s)\pf } \dd s   \leqslant  -\frac{1}{2\sqrt{\oeta}} \mathcal W_2(\rho_t,\rho_0)\,.
 \]
  Rearranging the terms and using that $ \sqrt{\mathcal F_{\mathcal A}(\rho_t)}  \geqslant 0$ concludes \rev{the proof of \eqref{eq:lem-W2transport}}.
\end{proof}

As a consequence, when  \eqref{eq:loc_LSInonlineaire} holds in a set $\mathcal A\subset \mathcal P_2(\R^d)$ that contains a unique \rev{a local minimizer $\rho_*\in\mathcal K$ and no other critical point}, in order to get the long-time convergence to $\rho_*$ at exponential speed starting from an initial condition $\rho_0\in\mathcal A$,  it mainly remains to show that $T_{\mathcal A}(\rho_0)=+\infty$. In specific cases, this can be shown by various means, for instance in dimension 1 by monotonicity arguments as in \cite{Tugautmonotone}, or in general using that the free energy decreases as in \cite{tugaut2014self}, etc. In the following, we focus on a self-contained argument which establishes the exponential convergence towards $\rho_*$ under the natural condition that $\mathcal W_2(\rho_*,\rho_0)$ is sufficiently small (in particular, without assuming that $\mathcal F(\rho_0)<\infty$).

 \subsection{Auxiliary regularization results}\label{sec:auxiliary}
 
Let us recall two results. The first is proven in \cite[Lemma 4.9]{Songbo}.

\begin{lem}\label{lem:Songbo}
Under \eqref{eq:Lipshitz} and \eqref{eq:LSIlineaire_locale}, for all $\rho \in\mathcal P_2(\R^d)$ and $\rho_*\in\mathcal K$, 
\[\mathcal H\po \rho |\Gamma(\rho)\pf \leqslant \po 1+ \eta L + \frac{L^2 \eta^2}{2}\pf \mathcal H(\rho|\rho_*)\,.\] 
\end{lem}

The second result is from \cite[Corollary 4.3]{Wang}.
\begin{prop}\label{prop:Wang}
Assume that there exists $\kappa_0,\kappa_1,\kappa_2>0$ such that for all $\nu,\mu\in\mathcal P_2(\R^d) $ and all $x,y\in\R^d$,
\[|\na E_{\mu}(0)|^2 + \sigma^2  \leqslant \kappa_0\po 1+ \int_{\R^d} |x|^2 \mu(\dd x)\pf  \]
and
\begin{equation}
\label{eq:Wang2}
-2\po \na E_{\mu}(x) - \na E_{\nu}(y) \pf \cdot (x-y) \leqslant \kappa_1 |x-y|^2 + \kappa_2 |x-y|\mathcal W_2(\nu,\mu)\,.
\end{equation}
Then, given two solutions $(\rho_t)_{t\geqslant 0}$ and $(\rho_t')_{t\geqslant 0}$ of \eqref{eq:granular_media_2}, for all $t>0$,
\[\mathcal H\po \rho_t|\rho_t'\pf \leqslant s_t\mathcal W_2^2 (\rho_0,\rho_0')\,,\]
where
\begin{equation*}
s_t=\frac{1}{\sigma^2 } \po \frac{\kappa_1}{1-e^{-\kappa_1 t}} + \frac12 t\kappa_2^2 e^{2(\kappa_1+\kappa_2)t}\pf\,. 
\end{equation*}
\end{prop}
Notice that, under the uniform Lipschitz condition \eqref{eq:Lipshitz}, since we can bound $|\na E_{\mu}(0)|\leqslant |\na E_{\delta_0}(0)| + L \mathcal W_2(\mu,\delta_0)$ for any $\mu\in\mathcal P_2(\R^d)$, the assumptions of Proposition~\ref{prop:Wang} are fulfilled if additionally there exists $\kappa_1>0$ such that  the one-sided Lipschitz condition
\begin{equation}
\label{eq:onesided}
\forall \mu\in\mathcal P_2(\R^d),x,y\in\R^d,\qquad 2\po \na E_{\mu}(x) - \na E_{\mu}(y) \pf \cdot (x-y) \geqslant  - \kappa_1 |x-y|^2\tag{\textbf{o-s-Lip}}
\end{equation}
holds, and then we can take $\kappa_2=2 L$ to have \eqref{eq:Wang2}.

\begin{cor}\label{Cor:FW2}
Assume \eqref{eq:convexity-c_W2}, \eqref{eq:Lipshitz}, \eqref{eq:onesided} and \eqref{eq:LSIlineaire_locale} for some $\lambda, L,\kappa_1,\eta>0$. Then, for all $(\rho_t)_{t\geqslant 0}$ solution to \eqref{eq:granular_media_2}, all $\rho_*\in\mathcal K$ and all $t>0$,
\[\mathcal F(\rho_t) - \mathcal F(\rho_*) \leqslant q_t \mathcal W_2^2(\rho_0,\rho_*)\]
with 
\begin{equation}
\label{eq:qt}
q_t =  \po 1+ \eta L+ \frac{4 \eta\lambda}{\sigma^2} + \frac{L^2 \eta^2}{2}\pf \po \frac{\kappa_1}{1-e^{-\kappa_1 t}} + \rev{2} tL^2 e^{(2\kappa_1+4L)t}\pf\,.
\end{equation}
\end{cor}
Notice that $q_t$ is of order $1/t$ as $t\rightarrow 0$. See \cite[Theorem 4.0.4]{ambrosio2005gradient} for a similar result.

\begin{proof}
Using \eqref{eq:Fmu0Fmu1<} in Lemma~\ref{lem:sandwich} with $\mathcal C = \lambda \mathcal W_2^2$ and then Lemma~\ref{lem:Songbo} and the Talagrand inequality \eqref{eq:NLT2} for $\rho_* = \Gamma(\rho_*)$ implied by \eqref{eq:LSIlineaire_locale},
\begin{eqnarray*}
 \mathcal F(\rho_t) -  \mathcal F(\rho_*) &  \leqslant & \sigma^2 \mathcal H\po \rho_t|\Gamma(\rho_t)\pf+   \lambda \mathcal W_2^2(\rho_t,\rho_*) \\
 & \leqslant & \sigma^2  \po 1+ \eta L+ \frac{4 \eta\lambda}{\sigma^2} + \frac{L^2 \eta^2}{2}\pf \mathcal H(\rho_t|\rho_*)\,.
\end{eqnarray*} 
Since $t\mapsto \rho_*$ solves \eqref{eq:granular_media_2}, conclusion follows from Proposition~\ref{prop:Wang}.
\end{proof}

\subsection{Conclusion}\label{sec:mainthm}

For $\mu\in\mathcal P_2(\R^d)$ and $r>0$, write $\mathcal B_{\mathcal W_2}(\mu,r) = \{\nu\in\mathcal P_2(\R^d),\mathcal W_2(\nu,\mu)\leqslant r\}$.

\begin{assu}\label{assu:total}
The conditions \eqref{eq:convexity-c_W2}, \eqref{eq:Lipshitz}, \eqref{eq:onesided} and \eqref{eq:LSIlineaire_locale} hold for some $\lambda, L,\kappa_1,\eta>0$, and \eqref{eq:loc_LSInonlineaire} holds on some non-empty set  $\mathcal A \subset \mathcal P_2(\R^d)$ for some $\oeta>0$. Furthermore  there exist $\rho_* \in \mathcal A\cap \mathcal K$ and $\delta>0$ such that $\mathcal B_{\mathcal W_2}(\rho_*,\delta) \subset \mathcal A$.
\end{assu}

\begin{thm}\label{thm:main}
Under Assumption \ref{assu:total}, set 
\[\delta' =  \frac{\delta }{2( 2 \sqrt{\oeta q_1} + e^{(\kappa_1/2+L)})}\,,  \]
where $q_1$ is given by \eqref{eq:qt} (with $t=1$). Then, for all $\rho_0 \in \mathcal B_{\mathcal W_2}(\rho_*,\delta')$, $T_{\mathcal A}(\rho_0)=\infty$ and  in particular
\[\mathcal F_{\mathcal A}(\rho_t)    \leqslant   e^{-t/\oeta}   \mathcal F_{\mathcal A}(\rho_0)  \]
for all $t\geqslant 0$. Furthermore, if we assume additionally that  $\mathcal B_{\mathcal W_2}(\rho_*,\delta)\cap \mathcal K=\{\rho_*\}$, then:
\begin{itemize}
\item The following non-linear Talagrand inequality holds:
\begin{equation}
\label{eq:thmTalagrand}
\mathcal W_2^2(\rho_0,\rho_*) \leqslant 4 \oeta \mathcal F_{\mathcal A}(\rho_0)\,. 
\end{equation}
\item Setting $C= e^{1/\oeta} \max(  4 \oeta q_1,e^{\kappa_1+2L} ) $, for all $t\geqslant 0$,
\begin{eqnarray}
\mathcal W_2^2(\rho_t,\rho_*)  &\leqslant &  C e^{-t/\oeta}  \mathcal W_2^2(\rho_0,\rho_*)\label{eq:W2Thm}\\
\mathcal F_{\mathcal A}(\rho_t)  & \leqslant &e^{-t/\oeta} \min \po  \mathcal F_{\mathcal A}(\rho_0)\ ,\ q_{\min(t,1)} e^{1/\oeta}  \mathcal W_2^2(\rho_0,\rho_*)\pf \label{eq:FdecayThm} \\
\sigma^2 \mathcal H(\rho_t|\rho_*) & \leqslant & \mathcal F_{\mathcal A}(\rho_t) + \lambda \mathcal W_2^2(\rho_t,\rho_*)\,.\label{eq:HThm}
\end{eqnarray}
\end{itemize}

\end{thm}

 The last inequality on the relative entropy is just \eqref{eq:Fmu0Fmu1>} in Lemma~\ref{lem:sandwich}, we simply recall it as it gives here the exponential convergence of $\mathcal H(\rho_t|\rho_*)$ (hence of $\|\rho_t-\rho_*\|_{TV}$) to zero.

 \rev{The relation between the transport inequality~\eqref{eq:thmTalagrand} and  PL inequalities is studied in more general settings for convex functionals in \cite{BlanchetBolte}. }
 
 \begin{rem}\label{rem:Ainfini}
 Write $\mathcal A_\infty(\rho_*) = \{\rho_0\in\mathcal A: T_A(\rho_0)=\infty\text{ and } \mathcal W_2(\rho_t,\rho_*)\rightarrow 0 \text{ as }t\rightarrow \infty\}$. Then it is clear by following the proof of Theorem~\ref{thm:main} that its conclusion (namely \eqref{eq:thmTalagrand}, \eqref{eq:W2Thm} and \eqref{eq:FdecayThm}) hold for any initial condition $\rho_0\in\mathcal A_{\infty}(\rho_*)$. Moreover, for any $\rho_0 \in\mathcal P_2(\R^d)$ such that $\mathcal W_2(\rho_t,\rho_*)\rightarrow 0$ as $t\rightarrow \infty$, there exists a $t_*>0$ such that $\rho_{t_*} \in \mathcal B_{\mathcal W_2}(\rho_*,\delta) \subset \mathcal A_\infty(\rho_*)$, from which there exists $C_0>0$ such that for all $t\geqslant 1$,
 \begin{equation}
 \label{eq:convergence_temps-long}
 \mathcal W_2^2(\rho_t,\rho_*) + \mathcal H(\rho_t|\rho_*) + \mathcal F(\rho_t) - \mathcal F(\rho_*) \leqslant C_0 e^{-t/\oeta}\,.
 \end{equation}
 \end{rem}

\begin{proof}
Given a solution $(\rho_t)_{t\geqslant 0}$ of \eqref{eq:granular_media_1}, we consider the time-inhomogeneous Markov diffusion process 
\[\dd X_t = -\na E_{\rho_t} (X_t) \dd t +\sqrt{2}\sigma \dd B_t\]
where $B$ is a Brownian motion  and $X_0 $ is distributed according to $\rho_0$, so that $X_t \sim \rho_t$ for all $t\geqslant 0$. Existence of strong solutions for this SDE is justified by the fact $(t,x)\mapsto \na E_{\rho_t}$ is continuous in time (by continuity of $t\mapsto \rho_t$ and \eqref{eq:Lipshitz}) and $\mathcal C^1$ in $x$, and explosion in finite time is prevented by \eqref{eq:onesided}.  Considering a synchronous coupling $(X,Y)$ (i.e. using the same Brownian motion for two processes) of such diffusions, the first one being associated with an arbitrary solution $\rho$, the second being associated to the stationary solution $t\mapsto \rho_*$, we get that
\[\frac{\dd}{\dd t} |X_t-Y_t|^2 = -2 (X_t-Y_t) \cdot \po \na E_{\rho_t} (X_t) - \na E_{\rho_*} (Y_t)\pf \leqslant (\kappa_1+L) |X_t-Y_t|^2 + L \mathcal W_2^2\po \rho_t,\rho_*\pf\,, \]
where we used \eqref{eq:Lipshitz} and \eqref{eq:onesided}. Taking the expectation, using the Gronwall Lemma, that $\mathcal W_2^2(\rho_t,\rho_*) \leqslant \mathbb E \po|X_t-Y_t|^2\pf$ and taking the infimum over the coupling of the initial conditions, we obtain
\begin{equation}
\label{eq:W2rhotrho*}
\mathcal W_2(\rho_t, \rho_*) \leqslant e^{(\kappa_1/2+L)t} \mathcal W_2(\rho_0, \rho_*)\,.
\end{equation}
In view of the definition of $\delta'$, this implies that for $\rho_0 \in \mathcal B_{\mathcal W_2}(\rho_*,\delta')$, $T_{\mathcal A}(\rho_0)\geqslant 1$ and $\rho_1\in\mathcal A$. On the other hand, using Corollary~\ref{Cor:FW2} and the decay of the free energy  along time, for all $t\geqslant 1$,
\[\mathcal F(\rho_t) - \mathcal F(\rho_*) \leqslant q_1 (\delta')^2\,.\]
Hence, for all $t \in [1,T_{\mathcal A}(\rho_0)]$,  using \eqref{eq:lem-W2transport}, 
\[\mathcal W_2(\rho_t,\rho_*) \leqslant \mathcal W_2(\rho_t,\rho_1) + \mathcal W_2(\rho_1,\rho_*) \leqslant    \delta' \po  2 \sqrt{\oeta q_1} + e^{(\kappa_1/2+L)} \pf  = \frac{\delta}{2} \,.\]
By contradiction, since $t\mapsto \mathcal W_2(\rho_t,\rho_*)$ is continuous, we get that $T_{\mathcal A}(\rho_0)=\infty$.

Now, assume that $\mathcal B_{\mathcal W_2}(\rho_*,\delta)\cap \mathcal K=\{\rho_*\}$. Since Equation~\eqref{eq:granular_media_2} admits $\mathcal F$ as a strict Lyapunov function, the LaSalle
invariance principle established in \cite{Carrillo} shows that the $\mathcal W_2$ distance between $\rho_t$ and $\mathcal K$ vanishes as $t\rightarrow \infty$. Since the trajectory stays within $\mathcal B_{\mathcal W_2}(\rho_*,\delta/2)$ and thus at a distance at least $\delta/2$ of $\mathcal K\setminus\{\rho_*\}$, this implies the convergence towards $\rho_*$. We can then let $t\rightarrow \infty$ in \eqref{eq:lem-W2transport} to get \eqref{eq:thmTalagrand}, that we use then to bound, for $t\geqslant 1$,
\[\mathcal W_2^2(\rho_t,\rho_*) \leqslant 4 \oeta \mathcal F_{\mathcal A}(\rho_t) \leqslant 4 \oeta e^{-(t-1)/\oeta}\mathcal F_{\mathcal A}(\rho_1)  \leqslant 4 \oeta e^{-(t-1)/\oeta} q_1\mathcal W_2^2(\rho_t,\rho_*)\,.\] 
For $t\leqslant 1$, we simply use \eqref{eq:W2rhotrho*}. The bound on $\mathcal F_{\mathcal A}(\rho_t)$ is straightforward from \eqref{eq:lem-F-decay} and Corollary~\ref{Cor:FW2}.
\end{proof}

\section{Local non-linear LSI in a parametric case}\label{sec:LSIparametric}

In this section we focus on the settings of Example~\ref{ex:hilbert}, where we recall that
\begin{equation}
    \label{eq:EHilbert}
    \mathcal E(\mu) = \int_{\R^d} V_0(x) \mu(\dd x) +  R\po \int_{\R^d} \varphi(x) \mu(\dd x)\pf.
\end{equation}
We assume that $V_0\in\mathcal C^2(\R^d,\R)$, $R\in\mathcal C^1(\hil,\R)$ and $\varphi\in\mathcal{C}^2(\R^d, \hil)$. In this case, denoting by $\na_{\hil}$  and $\langle \cdot,\cdot \rangle_{\hil}$ the gradient and scalar product over $\hil$,
\begin{equation*}
    E_\rho(x) = V_0(x) + \langle \nabla_{\hil} R(\varphi(\rho)), \varphi(x)\rangle_{\hil},
\end{equation*}
with $\varphi(\rho):=\int_{\R^d} \varphi \rho \in \hil$,
and Assumption~\ref{assu:lfdE} is satisfied if, for any $x \in \R^d$, $\|\varphi(x)\|_{\hil} \leq C(1+|x|^2)$.

The general idea that we develop in this section is that, since the non-linearity is somehow parametrized by $\varphi(\rho)$, information for the Wasserstein gradient flow can be deduced from standard analysis of fixed point or optimization on Hilbert spaces. General considerations are gathered in Section~\ref{subsec:Hilbert}, which are then applied to the granular media case \eqref{eq:def_F} in Section~\ref{subsec:LSIgranular}.

\begin{rem}
If $R$ is convex then $\mathcal E$ is functional convex. As discussed in Section~\ref{sec:convex_interaction_faible}, thanks to Lemma~\ref{lem:sandwich}, in that case \eqref{eq:LSIlineaire_locale} implies \eqref{eq:LSInonlineaire} and then global convergence at constant rate towards a unique stationary solution.  Hence, this is not the case we are interested in.
\end{rem}

\subsection{General results}\label{subsec:Hilbert}
\subsubsection{The associated fixed-point problem}

In the case \eqref{eq:EHilbert}, we get that $\Gamma(\rho) = \rho_{\varphi(\rho)}$ where, for $\psi\in\hil$,
\begin{equation}
    \label{eq:rho_psi}
    \rho_{\psi}(x) \propto \exp\po -\frac{1}{\sigma^2}\co V_0(x) + \langle\na_{\hil} R(\psi), \varphi(x)\rangle_{\hil}\cf \pf \,.
\end{equation}
Let us assume that $\mathcal{F}(\rho_\psi)<\infty$ for all $\psi \in \hil$. This implies in particular that \begin{equation}
    \label{eq:f}
    f(\psi) = \varphi\po \rho_{\psi}\pf = \int_{\R^d} \varphi(x) \rho_{\psi}(x)\dd x
\end{equation}
is well-defined, and we see that $\varphi\po \Gamma(\rho)\pf  =f\po \varphi(\rho)\pf$. In particular, any $\rho_*\in\mathcal K$ satisfies $\rho_* = \rho_{\varphi(\rho_*)}$. Thus, denoting by $\mathcal K' = \{\psi\in\hil,\  f(\psi)=\psi \}$ the set of fixed-points of $f$, 
\[\mathcal K = \{\rho_{\psi},\ \psi\in\mathcal K'\}\,.\]
One may hope that the stability properties of $\rho_* = \rho_{\psi_*} \in \mathcal K$ as a stationary solution of the (continuous-time) Wasserstein gradient descent is related to the stability of $\psi_*$ as a fixed-point of the (discrete-time) dynamical system $\psi \mapsto f(\psi)$ on $\hil$, which is a   classical question. In fact, we can find examples where $\rho_{\psi_*}$ is stable although $\psi_*$ is not (see the granular media equation with repulsive interaction in Figure~\ref{fig:repulsif} below). However, on the contrary, the next result shows that \eqref{eq:loc_LSInonlineaire} holds in a neighborhood $\rho_{\psi_*}$ (which is thus stable thanks to Theorem~\ref{thm:main}) if $\psi_*$ is   a geometrically attracting fixed-point.

\begin{prop}\label{prop:New_fixed_point}
Let $\psi_*\in\mathcal K'$, $\mathcal A' \subset \hil$ and $\alpha\in[0,1)$ be such that for all $\psi\in\mathcal A'$,
\begin{equation}
    \label{eq:contract_fixedpoint}
\|f(\psi)-\psi_*\|_{\hil} \leqslant \alpha \|\psi - \psi_*\|_{\hil}\,.
\end{equation}
Then, for any $\psi \in \mathcal A'$, 
\begin{equation}
    \label{eq:psipsi'} 
\|\psi-\psi_*\|_{\hil} \leqslant \frac{1}{1-\alpha} \|f(\psi)-\psi\|_{\hil}\,.
\end{equation}
Assuming moreover that $\varphi$ is $ \ell$-Lipschitz continuous, then, for all $\rho\in\mathcal P_2(\R^d)$ such that $\varphi(\rho)\in\mathcal A'$,
\[\|\varphi(\rho)-\psi_*\|_{\hil} \leqslant \frac{\ell}{1-\alpha} \mathcal W_2\po \rho,\Gamma(\rho)\pf\,.\]
Finally, assuming furthermore \eqref{eq:LSIlineaire_locale} and that $R+\theta\|\cdot\|_{\hil}^2$ is convex for some $\theta>0$, then, for all $\rho\in\mathcal P_2(\R^d)$ such that $\varphi(\rho)\in\mathcal A'$,
\[\mathcal F(\rho) - \mathcal F(\rho_{\psi_*}) \leqslant \eta \po \sigma^2   +  \frac{4\eta\theta \ell^2}{(1-\alpha)^2}\pf\mathcal I\po \rho|\Gamma(\rho)\pf\,.\]
\rev{In particular, provided  $\mathcal F(\rho_{\psi_*})$ is the minimum of $\mathcal F$ over $\mathcal A=\{\rho\in\mathcal P_2(\R^d),\ \varphi(\rho)\in\mathcal A'\}$, \eqref{eq:loc_LSInonlineaire} holds over $\mathcal A$.}
\end{prop}

\begin{proof}
The first inequality~\eqref{eq:psipsi'}  simply follows from the triangular inequality
\[\|\psi-\psi_*\|_{\hil} \leqslant \|\psi-f(\psi)\|_{\hil} + \|f(\psi)-\psi_*\|_{\hil}  \leqslant \|\psi-f(\psi)\|_{\hil} + \alpha \|\psi-\psi_*\|_{\hil} \,. \]
Applying this to $\psi=\varphi(\rho)$ reads
\[\|\varphi(\rho)-\psi_*\|_{\hil}
\leqslant \frac{1}{1-\alpha}  \|\varphi(\rho)-\varphi\po \Gamma(\rho)\pf \|_{\hil} 
\leqslant \frac{\ell}{1-\alpha} \mathcal W_2\po \rho,\Gamma(\rho)\pf\,.\]
Finally, applying Lemma~\ref{lem:sandwich} (with $\mathcal C(\mu_0,\mu_1)= \theta\|\varphi(\mu_0)-\varphi(\mu_1)\|_{\hil}^2$ as in Example~\ref{ex:hilbert}),
\begin{eqnarray*}
\mathcal F(\rho) - \mathcal F(\rho_{\psi_*}) & \leqslant & \sigma^2 \mathcal H\po \rho|\Gamma(\rho)\pf  + \theta\|\varphi(\rho)-\psi_*\|_{\hil}^2\\ 
& \leqslant & \sigma^2 \mathcal H\po \rho|\Gamma(\rho)\pf  +  \frac{\theta \ell^2}{(1-\alpha)^2} \mathcal W_2^2\po \rho,\Gamma(\rho)\pf\\
& \leqslant & \po \sigma^2   +  \frac{4\eta\theta \ell^2}{(1-\alpha)^2}\pf \mathcal H\po \rho|\Gamma(\rho)\pf 
\end{eqnarray*}
thanks to Talagrand inequality \eqref{eq:NLT2}. Conclusion follows from \eqref{eq:LSIlineaire_locale}.
\end{proof}

\begin{rem}
The constant $\oeta$ in  \eqref{eq:loc_LSInonlineaire} obtained 
when applying Proposition~\ref{prop:New_fixed_point}, and thus the quantitative convergence estimates obtained when applying Theorem~\ref{thm:main}, are explicit in terms of the constants in the assumptions. In particular there is no additionnal dependency in the dimension.
\end{rem}

We have thus identified the following general conditions:

\begin{assu}
\label{assu:fixedpoint}
There exist $\theta,\ell,\eta>0$ such that $\varphi$ is $\ell$-Lipschitz continuous, $R+\theta\|\cdot\|_{\hil}^2$ is convex and \eqref{eq:LSIlineaire_locale} holds. Moreover, $\mathcal{F}(\rho_\psi) < \infty$ for any $\psi \in \hil$.
\end{assu}

\begin{rem}\label{rem:LellL'}
As discussed in Example~\ref{ex:hilbert}, these conditions imply~\eqref{eq:convexity-c_W2} with $\lambda=\theta\ell^2$. Moreover, since 
\[\na E_{\rho}(x) = \na V_0(x) + \langle \na_{\hil} R\po \varphi(\rho)\pf, \na \varphi(x)\rangle_{\hil}\,, \]
assuming that $\na_{\hil} R$ is $L'$-Lipschitz continuous, we get that, for any $x\in\R^d$,
\[|\na E_{\mu_0}(x)  - \na E_{\mu_1}(x) |\leqslant \ell L' \| \varphi(\mu_0)-\varphi(\mu_1)\|_{\hil} \leqslant \ell^2 L ' \mathcal W_2(\mu_0,\mu_1)\,, \]
which is \eqref{eq:Lipshitz}.
\end{rem}

Under the general conditions of Assumption~\ref{assu:fixedpoint}, given a specific problem, what remains to be done is  to establish the contraction~\eqref{eq:contract_fixedpoint}. A simple way to get it locally at some fixed point $\psi_*$ is to check that $|\na_{\hil} f(\psi_*)|$, the operator norm of the Jacobian matrix of $f$, is strictly less than $1$ at $\psi_*$. We end up with the following.

\begin{cor}
\label{cor:|nabla_f|<1}
Under Assumption~\ref{assu:fixedpoint}, let $\psi_*\in\mathcal K'$ \rev{be such that $\rho_{\psi_*}$ is a local minimizer of $\mathcal F$}. Assume that $f$ is differentiable at $\psi_*$ with $|\na_{\hil} f(\psi_*)|<1$. Then there exist $\delta,\oeta>0$ such that \eqref{eq:loc_LSInonlineaire} holds on $\mathcal A= \mathcal B_{\mathcal W_2}(\rho_{\psi_*},\delta)$.
\end{cor}
\begin{proof}
Using that $\psi_*$ is a fixed point of $f$ and a Taylor expansion,
\[\|f(\psi)- \psi_*\|_{\hil}  = \| \na_{\hil} f(\psi_*) (\psi-\psi_*)\|_{\hil} + o(\|\psi-\psi_*\|_{\hil}) \leqslant \alpha \|\psi-\psi_*\|_{\hil}\]
for some $\alpha<1$ uniformly over some neighborhood $\mathcal N$ of $\psi_*$. The function $\varphi$ being Lipschitz continuous, $\rho \mapsto \varphi(\rho)$ is Lipschitz continuous for the $\mathcal W_2$ distance, which means that $\varphi(\rho)$ lies in $\mathcal N$ for all $\rho \in \mathcal B_{\mathcal W_2}(\rho_{\psi_*},\delta)$ for $\delta$ small enough. Conclusion follows from Proposition~\ref{prop:New_fixed_point}.
\end{proof} 

\subsubsection{Decomposition of the free energy}\label{sec:decomposition}

The difference between \eqref{eq:LSIlineaire_locale} and \eqref{eq:LSInonlineaire} obviously lies in the difference between $\mathcal H(\rho|\Gamma(\rho))$ and $\mathcal F(\rho)$. In the case \eqref{eq:EHilbert}, we can decompose
\[\mathcal F(\rho) = \sigma^2  \mathcal H\po \rho|\Gamma(\rho)\pf + g\po \varphi(\rho) \pf \]
where
\begin{equation}
    \label{eq:gsigma}
    g(\psi) =  R\po \psi \pf - \langle \na_{\hil} R\po \psi \pf , \psi \rangle_{\hil} - \sigma^2 \ln \int_{\R^d} e^{- \frac{1}{\sigma^2}\co  V_0(x) +  \langle \na_{\hil} R\po \psi \pf , \varphi(x) \rangle_{\hil} \cf }\dd x\,.
\end{equation}
This second part of the free energy only depends on $\rho$ through the parameter $\varphi(\rho)$.

\begin{rem}\label{rem:Bregman}
If $R$ is concave, then $M(\psi_1,\psi_2):= R(\psi_2)-R(\psi_1) + \langle \na_{\hil} R(\psi_2),\psi_1-\psi_2\rangle_{\hil} \geqslant 0 $ (this is the so-called Bregman divergence associated with $-R$). Assuming furthermore that $V_0(x) = V(x) - R(\varphi(x))$ with $V$ such that $e^{-V/\sigma^2 }$ is integrable, and that $M(\psi_1,\psi_2) \rightarrow +\infty$ when $\psi_2\rightarrow \infty$, we get by dominated convergence that
\[e^{-\frac{1}{\sigma^2} g(\psi)} =    \int_{\R^d} e^{- \frac{1}{\sigma^2}\co  V(x) + M \po \varphi(x), \psi\pf \cf}\dd x  \underset{\|\psi\|_{\hil} \rightarrow +\infty}\longrightarrow  0\,. \]
As a consequence, $g$ is lower bounded and goes to infinity at infinity. Moreover, it is continuous, and thus in particular when $\hil$ has finite dimension then $g$ reaches its minimum. Notice that, in the quadratic case where $R(\psi)=-\theta \|\psi\|_{\hil}^2$, $M(\psi_1,\psi_2)=\theta\|\psi_1-\psi_2\|_{\hil}^2$.
\end{rem}

Assuming that $R\in\mathcal C^2(\hil ,\R)$,
\begin{equation}
    \label{eq:nag}
\na_{\hil} g(\psi) =  \na_{\hil}^2 R(\psi)  \po f(\psi) - \psi\pf  \,,
\end{equation}
where we used that
\[\frac{\int_{\R^d} \na_{\hil}^2 R(\psi)  \varphi e^{-\frac{1}{\sigma^2}[V_0(x) + \langle \na_\hil R\po \psi \pf,   \varphi(x)\rangle_{\hil}]}\dd x}{\int_{\R^d} e^{-\frac{1}{\sigma^2}[V_0(x) + \langle \na_\hil R\po \psi \pf,   \varphi(x)\rangle_{\hil}]}\dd x} = \na_{\hil}^2 R(\psi)   \varphi\po \rho_{\psi}\pf =  \na_{\hil}^2 R(\psi)   f(\psi)\,. \]
As a consequence, if $\psi_*$ is a critical point of $g$ such that $\na_{\hil}^2 R(\psi_*)$ is non-singular (which is in particular the case for all critical points of $g$ if $R$ is strictly convex or concave, which covers many cases of interest) then necessarily $\psi_* \in \mathcal K'$. In that case, if moreover $\psi_*$ is a local minimizer of $g$, then $\rho_{\psi_*}$ is a minimizer of both parts of the free energy, $\rho \mapsto \mathcal H(\rho|\Gamma(\rho))$ and $\rho \mapsto g\po \varphi(\rho)\pf$ (at least locally in the latter). It is thus a good candidate to be a stable stationary solution for the Wasserstein gradient descent for the free energy. This is indeed the case:

\begin{prop}\label{prop:g_minima}
Assume \eqref{eq:convexity-c_W2}, \eqref{eq:Lipshitz}, \eqref{eq:onesided} and \eqref{eq:LSIlineaire_locale} for some $\lambda, L,\kappa_1,\eta>0$   and $\varphi$ is $\ell$-Lipschitz continuous. 
Let $\psi_*\in\mathcal K'$ be a  proper isolated local minimizer of $g$, in the sense that for all $\varepsilon>0$ small enough, $\zeta(\varepsilon):=\inf\{g(\psi)-g(\psi_*),\ \psi\in\hil,\ \|\psi-\psi_*\|_{\hil}=\varepsilon\}>0$ and $\mathcal{K}' \cap \mathcal B(\psi_*,\varepsilon) = \{\psi_*\}$. Then there exists $\delta>0$ such that for all initial conditions $\rho_0 \in \mathcal B_{\mathcal W_2}(\rho_{\psi_*},\delta)$, the flow \eqref{eq:granular_media_2} converges in long-time to $\rho_{\psi_*}$.

\end{prop}

\begin{proof}
Let $\varepsilon>0$ be small enough so that $\zeta(\varepsilon)> 0$ and  $\mathcal{K}' \cap \mathcal B(\psi_*,\varepsilon) = \{\psi_*\}$.  Notice that, for $\psi,\tilde\psi \in \mathcal K'$, $\|\psi-\tilde{\psi}\|_\hil = \|\varphi(\rho_\psi)-\varphi(\rho_{\tilde{\psi}})\|_\hil \leq \ell \mathcal{W}_2(\rho_\psi,\rho_{\tilde{\psi}})$, so that the second condition implies that $\mathcal K\cap \mathcal B_{\mathcal W_2}(\rho_{\psi_*},\varepsilon/\ell) = \{\rho_{\psi_*}\}$, and thus $\rho_{\psi_*}$ is an isolated critical point of $\mathcal F$. For any $\rho \in\mathcal P_2$ with $\|\varphi(\rho)-\psi_*\|_{\hil}=\varepsilon$, 
\[\mathcal F(\rho) \geqslant g\po \varphi(\rho)\pf \geqslant g(\psi_*) + \zeta(\varepsilon) = \mathcal F(\rho_{\psi_*}) + \zeta(\varepsilon)\,.\]
As a consequence, given a solution $(\rho_t)_{t\geqslant 0}$ of \eqref{eq:granular_media}, if there is a time $t_0$ such that $\|\varphi(\rho_{t_0}) - \psi_*\|_{\hil} < \varepsilon$ and $\mathcal F(\rho_{t_0}) - \mathcal F(\rho_{\psi_*}) < \zeta(\varepsilon) $ then the monotonicity of the free energy along the flow together with the continuity of $t\mapsto \|\varphi(\rho_t)-\psi_*\|_{\hil} $ implies that $\|\varphi(\rho_t)-\psi_*\|_{\hil}<\varepsilon$ for all $t\geqslant t_0$. Conclusion would thus follow from  LaSalle
invariance principle~\cite{Carrillo} using that $\rho_{\psi_*}$ is isolated.

Hence, let us prove that the previous conditions are met for initial conditions that are sufficiently close to $\rho_{\psi_*}$. We follow arguments similar to the proof of Theorem~\ref{thm:main}. Indeed, for all $t\geqslant 0$, using \eqref{eq:W2rhotrho*}
,
\[\|\varphi(\rho_{t}) - \psi_*\|_{\hil} \leq \ell  e^{(\kappa_1/2+L)t}  \mathcal W_2\po \rho_0 ,\rho_{\psi_*}\pf \,.\]
By taking $\delta <  \varepsilon e^{-\kappa_1/2-L} /\ell$, we ensure that 
$\|\varphi(\rho_{t_0}) - \psi_*\|_{\hil} <\varepsilon$ at time $t_0=1$ for all $\rho_0 \in \mathcal B_{\mathcal W_2}(\rho_{\psi_*},\delta)$. On the other hand, thanks to Corollary~\ref{Cor:FW2}, by taking $\delta$ small enough, we can also ensure that   $\mathcal F(\rho_{t_0}) - \mathcal F(\rho_{\psi_*}) < \zeta(\varepsilon) $ with such initial conditions.
\end{proof}
 
\begin{rem}
The argument that initial conditions that start with a free energy sufficiently close to $\mathcal F(\rho_*)$ will stay within a ball centered at $\rho_*$ has been used e.g. in \cite{tugaut2014self,Bashiri}. Here we combine it with the small-time regularization result of Corollary~\ref{Cor:FW2} to get a result only in terms of a $\mathcal W_2$ ball.
\end{rem}
 
 In fact, we can now reinterpret Proposition~\ref{prop:New_fixed_point} in light of Proposition~\ref{prop:g_minima}. Indeed, differentiating again \eqref{eq:nag} (assuming suitable regularity) at some $\psi_*\in\mathcal K'$ yields
 \begin{equation}
     \label{eq:nag2}
\na_{\hil}^2 g(\psi_*) = \na_{\hil}^2 R(\psi_*) \po \na_{\hil} f(\psi_*) - I\pf
\end{equation}
 (where we used that  $f(\psi_*)=\psi_*$). For brevity, denote $\na^2R_*=\na_{\hil}^2 R(\psi_*)$.  For any $u\in\hil$,
 \[\langle u,\na_{\hil}^2 g(\psi_*) u \rangle_{\hil} = \langle u , \na^2 R_* \na_{\hil} f(\psi_*)u  \rangle_{\hil} - \langle u , \na^2 R_* u\rangle_{\hil} \,.\]
 Assuming that $\na^2 R_*$ is negative definite, we can consider the norm $\|u\|_{*} = \sqrt{-\langle u, \na^2 R_*u\rangle_{\hil}}$ and the associated operator norm $\|\na_{\hil} f(\psi_*)\|_*$. Then, 
  \[\langle u, \na_{\hil}^2 g(\psi_*) u \rangle_{\hil} \geqslant \po 1 - \|\na_{\hil} f(\psi_*)\|_*\pf \|u\|_*^2 \,.\]
 We get that $\psi_*$ is a non-degenerate local minimizer of $g$ if $\|\na_{\hil} f\|_* <1$. On the other hand, if  $\|\na_{\hil} f\|_* <1$  and $(\na^2 R_*)^{-1}$ is bounded, then we can assume   that $\|\cdot\|_* = \|\cdot\|_{\hil}$ (which amounts to the linear change of coordinates given by $(\na^2 R_*)^{-1/2}$) and apply Proposition~\ref{prop:New_fixed_point} or, alternatively, we can use the contraction \eqref{eq:contract_fixedpoint} expressed with norm $\|\cdot\|_*$ to get  \eqref{eq:psipsi'} with the same norm and then with the initial $\|\cdot\|_{\hil}$ by the equivalence of the norms  and then proceed with Proposition~\ref{prop:New_fixed_point}). At least, this suggests that in some cases, in order to apply Proposition~\ref{prop:New_fixed_point} at some point $\psi_*$, it may be more natural to work with the norm $\|\cdot\|_*$ associated to $\na^2 R_*$ (when the latter is negative definite).

 \rev{\begin{rem}\label{rem:minimizer}
     To sum up, when $\psi_* \in \mathcal K'$ is such that $\na^2_{\mathbb H} R(\psi_*)$ is negative definite and $\|\na_{\mathbb H} f\|_*<1$, then $\rho_{\psi_*}$ is a local minimizer of $\mathcal F$ and thus Corollary~\ref{cor:|nabla_f|<1} applies. 
 \end{rem}}
 
 Besides, we compute that
 \[ \na_{\hil} f(\psi) = -\frac1{\sigma^2} \co \int_{\R^d} \varphi \po \na_{\hil}^2 R(\psi) \varphi\pf^T \rho_{\psi} - f(\psi) \po \na_{\hil}^2 R(\psi) f(\psi) \pf^T\cf \,, \] 
 from which, for $u_1,u_2\in\hil$, writing $w_i=\na^2 R_* u_i$ for $i=1,2$,
 \begin{eqnarray}
     \langle u_1, \na^2 R_* \na_\hil f(\psi_*) u_2\rangle_{\hil} &=& - \frac1{\sigma^2}\left  \langle w_1 ,  \co\int_{\R^d} \varphi \varphi^T \rho_{\psi_*} - f(\psi_*) \po f(\psi_*)\pf ^T\cf w_2 \right\rangle_{\hil} \nonumber\\
     &=& - \frac1{\sigma^2} \langle w_1 , \mathrm{covar}_{\rho_{\psi_*}}\po \varphi(X) \pf  w_2\rangle_{\hil} \,.     \label{eq:u1naR2fu2}
 \end{eqnarray}
  In particular, we see that $\na_{\hil} f(\psi_*)$ is symmetric and nonnegative  for the 
scalar product associated to $\|\cdot\|_*$, from which
\begin{equation}
         \label{eq:u1naR2fu2bis}
\|\na_{\hil} f(\psi_*)\|_* = \sup_{u\neq 0} \frac{\langle u, \na^2 R_* \na_{\hil} f(\psi_*) u\rangle_{\hil}}{\langle u, \na^2 R_*   u\rangle_{\hil}}\,.
\end{equation}
 
Furthermore, if we work with suitable coordinates in order to enforce that $\na^2 R_*=-I$, we get that, for $\psi_*\in\mathcal K'$,
 \[\na_{\hil}^2 g(\psi_*) = I - \frac{1}{\sigma^2} \mathrm{covar}_{\rho_{\psi_*}}(\varphi(X)) \,. \]

\subsubsection{Degenerate minima}\label{sec:degenerate}

There is no convergence rate in Proposition~\ref{prop:g_minima}, and we cannot expect one in general if we do not assume that $\psi_*$ is a non-degenerate local minimizer of $g$. Assuming again that $\na_{\hil}^2 R(\psi_*)$ is non-singular, in view of \eqref{eq:nag2}, if $\na_{\hil}^2 g(\psi_*)$ is singular then $|\na_{\hil} f(\psi_*)| \geqslant  1$ (since $\na_{\hil} f(\psi_*) u = u$ for   $u\in\mathrm{ker}\na_{\hil}^2 g(\psi_*)$, and thus this does not depend on the norm we use). In other words, degenerate minima of $g$ fall in the limit case where we cannot apply Proposition~\ref{prop:New_fixed_point} anymore. Then we can weaken the contraction condition~\eqref{eq:contract_fixedpoint} to allow for high-order stable fixed points. The proof of the following is the same as the one of Proposition~\ref{prop:New_fixed_point} (hence omitted).

\begin{prop}\label{prop:New_fixed_point_degenere}
Let $\psi_*\in\mathcal K'$, $\mathcal A' \subset \hil$ and $\beta>0,\nu\in(0,1)$ be such that for all $\psi\in\mathcal A'$,
\begin{equation}\label{eq:degenerate_condition}
\|\psi-\psi_*\|_{\hil} \leqslant \beta \|f(\psi)-\psi\|_{\hil}^{\nu}\,.
\end{equation}
Assuming   furthermore \eqref{eq:LSIlineaire_locale}, that $R+\theta \|\cdot\|_{\hil}^2$ is convex for some $\theta>0$ and that  $\varphi$ is $\ell $-Lipschitz continuous, then, for all $\rho\in\mathcal P_2(\R^d)$ such that $\varphi(\rho)\in\mathcal A'$,
\begin{equation}
    \label{eq:degenerateLSI}
    \mathcal F(\rho) - \mathcal F(\rho_{\psi_*}) \leqslant \eta   \sigma^2   \mathcal I\po \rho|\Gamma(\rho)\pf +  \theta  \beta^2 (4\eta^2\ell^2 )^{\nu}  \mathcal I^{\nu} \po \rho|\Gamma(\rho)\pf\,.
\end{equation}
\end{prop}

Under the settings of Proposition~\ref{prop:g_minima}, assume that \eqref{eq:degenerate_condition} holds in a neighborhood of $\psi_*$. Thanks to Proposition~\ref{prop:g_minima}, there exists $\delta>0$ such that for all initial conditions $\rho_0 \in \mathcal B_{\mathcal W_2}(\rho_{\psi_*},\delta)$, the gradient flow converges to $\rho_*$ and the inequality~\eqref{eq:degenerateLSI} holds for $\rho_t$ for all time $t\geqslant 0$. Combined with Corollary~\ref{Cor:FW2}, this yields algebraic convergence rates of the form
\[ \mathcal F(\rho_t) - \mathcal F(\rho_{\psi_*}) \leqslant  \frac{C}{\po \mathcal W_2^{2r}(\rho_0,\rho_{\psi_*}) + t  \pf^{1/r}}\,,\]
with $r=1/\nu-1>0$ and some constant $C>0$, for all $t\geqslant 1$. From this algebraic decay of the free energy, we can get convergence of $\rho_t$ to $\rho_{\psi_*}$ since \eqref{eq:Fmu0Fmu1>} in Lemma~\ref{lem:sandwich} gives
\[\mathcal H\po \rho_t|\rho_{\psi_*}\pf \leqslant  \mathcal F(\rho_t) - \mathcal F(\rho_{\psi_*}) + \theta \|\varphi(\rho_t)-\psi_*\|_{\hil}^2\,.\]
When following the proof of Proposition~\ref{prop:New_fixed_point} to get Proposition~\ref{prop:New_fixed_point_degenere}, we see that we  obtain as an intermediary inequality that 
\[\|\varphi(\rho_t)-\psi_*\|_{\hil}^2 \leqslant \beta^2 (4\eta \ell^2)^{\nu} \mathcal H^{\nu}\po \rho_t|\Gamma(\rho_t)\pf  \] 
as long as $\varphi(\rho_t) \in \mathcal A'$ (hence for all times here since $\rho_0 \in \mathcal B_{\mathcal W_2}(\rho_{\psi_*},\delta)$). Moreover, since $\psi_*$ is a local minimizer of $g$, by taking $\delta$ sufficiently small we get that $g(\varphi(\rho_t)) \geqslant g(\psi_*)$ for all $t\geqslant 0$, in other words
\[\mathcal H(\rho_t|\Gamma(\rho_t)) \leqslant \mathcal F(\rho_t) - \mathcal F(\rho_{\psi_*})\,.\]
Gathering all these bounds give 
\begin{equation}
    \label{eq:EntropyDecayDegenerate}
    \mathcal H\po \rho_t|\rho_{\psi_*}\pf \leqslant  \frac{C_0}{t^{\nu^2/(1-\nu)}}
\end{equation}
for some $C_0>0$ for all $t\geqslant 0$.

\subsection{Application to granular media equation}\label{subsec:LSIgranular}

In this section we focus on the granular media equation~\eqref{eq:granular_media} on $\R^d$, which corresponds to  the free energy \eqref{eq:def_F} (in fact we will consider a slightly more general case in Section~\ref{sec:localization}). For clarity, for now, let us focus on the quadratic interaction case where $W(x,y)=\theta|x-y|^2$ for some $\theta> 0$ (recall from \eqref{eq:courbureWa_kr_k} that, for $\theta \leqslant0$, $\mathcal E$ is convex and thus \eqref{eq:LSInonlineaire} follows from \eqref{eq:LSIlineaire_locale} thanks to Lemma~\ref{lem:sandwich}). This is a particular case of \eqref{eq:EHilbert} with $\mathcal H=\R^d$,
\[V_0(x) = V(x) + \theta |x|^2\,,\qquad  \varphi(x)=x\,,\qquad R(\psi) = -\theta |\psi|^2\,.\]
Assumption~\ref{assu:fixedpoint} is satisfied  since $R+\theta|\cdot|^2$ is convex and $\varphi$ is Lipschitz continuous (which, as discussed in Remark~\ref{rem:LellL'}, also gives \eqref{eq:convexity-c_W2} and \eqref{eq:Lipshitz}). We assume that $V_0$ is strongly convex outside a compact set, so that \eqref{eq:LSIlineaire_locale} holds (as explained in Section~\ref{sec:settings}) and also \eqref{eq:onesided}. Assumptions~\ref{assu:F}, \ref{assu:lfdE} and \ref{assu:loc-eq} are readily checked, and Assumption~\ref{assu:pde} follows from Propositions~\ref{prop:Ambrosio} and \ref{prop:Wang}.

 Using the notations introduced in Section~\ref{subsec:Hilbert} (except that we use $m$, as in \emph{mean}, to denote the parameter, instead of $\psi$ as in the general case), for $m\in\R^d$,  
\[\rho_m(x) = \frac{e^{- \frac{1}{\sigma^2}\co  V(x)+ \theta|x|^2 -2\theta x \cdot m\cf }}{\int_{\R^d} e^{- \frac{1}{\sigma^2}\co  V(y)+\theta|y|^2 -2\theta y \cdot m\cf}\dd y} \,,\qquad f(m) = \int_{\R^d} x\rho_m(x)\dd x\,, \]
and
\[g(m) =  \theta |m|^2 - \sigma^2 \ln \int_{\R^d} e^{ - \frac{1}{\sigma^2}\co V(x) + \theta|x|^2 - 2 \theta x\cdot m \cf  }\dd x  =  -\sigma^2 \ln \int_{\R^d} e^{- \frac{1}{\sigma^2}\co V(x) + \theta|x-m|^2 \cf }\dd x  \,. \]

The graph of $g$ and $f$ are represented in Figures~\ref{fig:double-well},  \ref{fig:convex} and \ref{fig:repulsif}, in dimension $1$, either in the double-well case $V(x)= x^4/4-x^2/2 $ (Figures \ref{fig:double-well} and \ref{fig:repulsif}), or the convex case $V(x)=x^2/2$ (Figure \ref{fig:convex}), in each case with $W(x,y)=\theta|x-y|^2$ (with $\theta=1$ in Figures \ref{fig:double-well} and \ref{fig:convex} and $\theta=-1$ in Figure~\ref{fig:repulsif}), at various temperatures.
 
\begin{figure}
\centering
\includegraphics[scale=0.3]{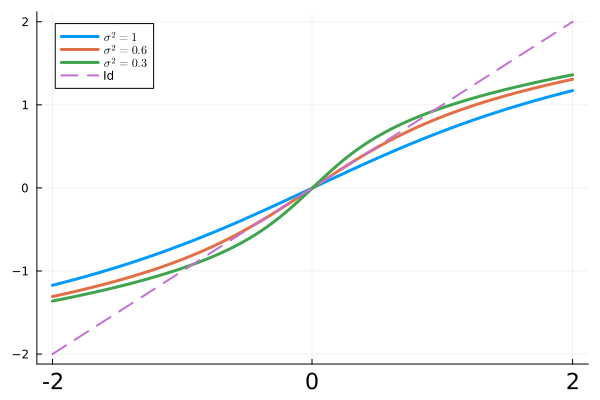}
\includegraphics[scale=0.3]{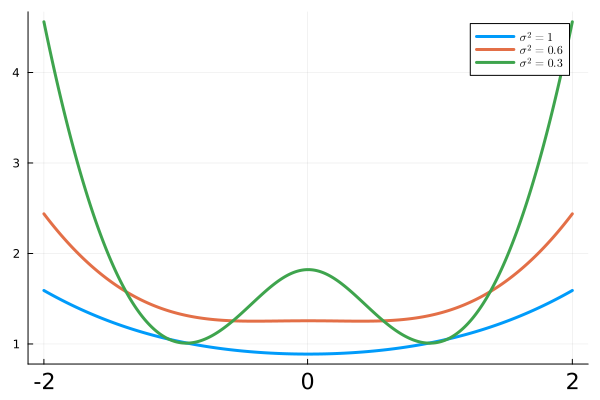}
\caption{Graph of $f$ (left) and $g$ (right) in the double-well attractive case for $\sigma^2\in\{1,0.6,0.3\}$} \label{fig:double-well}
\end{figure}

\begin{figure}
\centering
\includegraphics[scale=0.3]{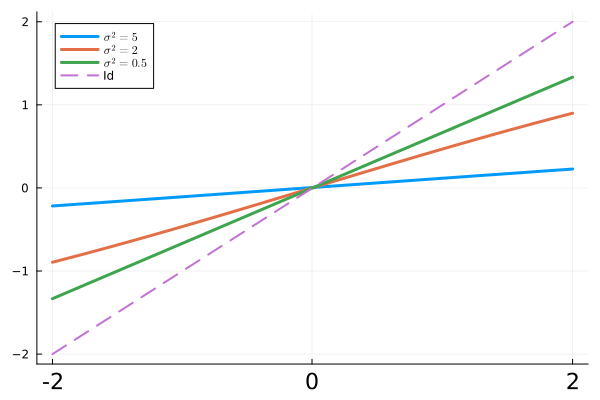}
\includegraphics[scale=0.3]{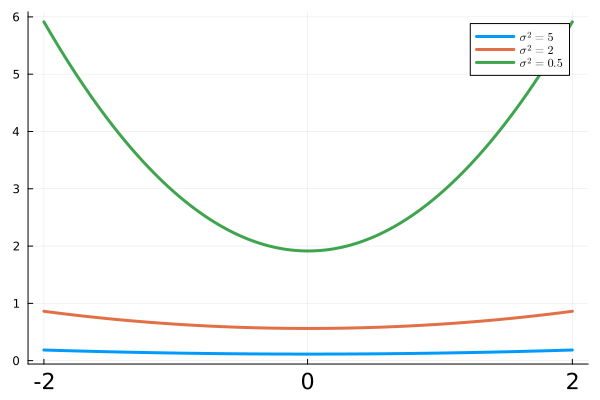}
\caption{Graph of $f$ (left) and $g$ (right) in the single-well attractive case for $\sigma^2\in\{5,2,0.5\}$} 
\label{fig:convex}
\end{figure}

\begin{figure}
\centering
\includegraphics[scale=0.3]{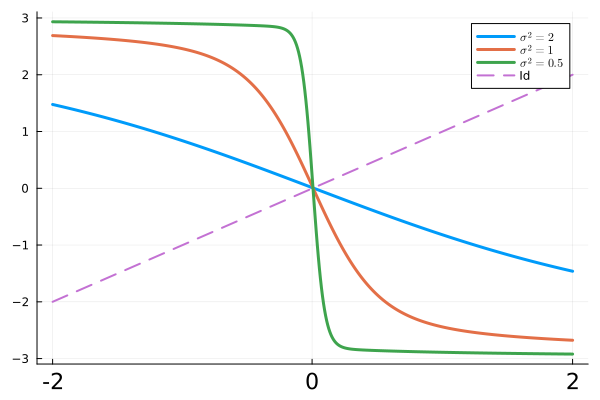}
\includegraphics[scale=0.3]{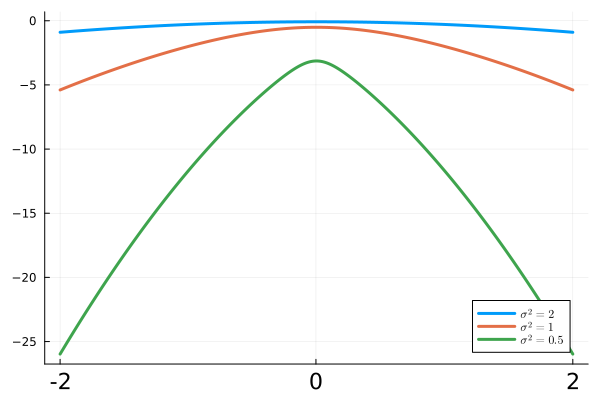}
\caption{Graph of $f$ (left) and $g$ (right) in the double-well repulsive case for $\sigma^2\in\{2,1,0.5\}$. At low temperature, $f'(0)<-1$, so that $0$ is an unstable fixed point of $x\mapsto f(x)$.} 
\label{fig:repulsif}
\end{figure}

\subsubsection{The one-dimensional double-well case}\label{sec:doublepuit1D}

In this section, $d=1$, $\theta>0$ and $V(x)=\frac{x^4}{4}-\frac{x^2}{2}$. In this one-dimensional attractive double-well case, the function $f$ has been precisely studied in e.g. \cite{Tugautdoublewell} (see also \cite{Minibatch}), and the following holds (see Figure~\ref{fig:double-well}).
\begin{itemize}
\item There exists $\sigma_c^2>0$ such that, for all $\sigma^2 \geqslant \sigma_c^2$, the unique fixed point of $f$ is $m=0$ and, for all $\sigma^2 <\sigma_c^2$, there are three such fixed points, $0,m_+,m_-$, with $m_+>0$ and  $m_-=-m_+$.
\item When $\sigma^2 < \sigma_c^2$, $f$ is an increasing function with  $0< f'(m_+)=f'(m_-) <1 < f'(0) $, $f'(m) \rightarrow 0$ as $|m|\rightarrow \infty$, $f(m) > m $ for $m\in(-\infty,m_-)\cup (0,m_+)$ and $f(m)<m$ for $m\in (m_-,0)\cup(m_+,\infty)$.
\item For $\sigma^2 = \sigma_c^2$, $f$ is an increasing function with $f'(m)<1$ for all $m\neq 0$,  $f'(m) \rightarrow 0$ as $|m|\rightarrow \infty$, $f'(0)=1$, $f''(0)=0$ and $f^{(3)}(0) <0$.
\item For $\sigma^2 > \sigma_c^2$, $f$ is an increasing function with $f'(m)\in(0,1)$ for all $m\in\R$ and  $f'(m) \rightarrow 0$ as $|m|\rightarrow \infty$.
\end{itemize}
Moreover, since $g'(m)= 2\theta (m-f(m))$,  the critical points of $g$ are the  fixed points of $f$ and, for $\sigma^2 < \sigma_c^2$, $g$ is  decreasing  on $(-\infty,m_-]$ and $ [0,m_+]$ and  increasing on $[m_-,0]$ and $[m_+,\infty)$. As a consequence of all this, below the critical temperature, we get a non-linear log-Sobolev inequality for any  non-centered $\rho\in\mathcal P_2(\R^d)$.  We introduce the notation
\begin{equation*}
    m(\rho) = \int_\R x \rho(x)\dd x.
\end{equation*}

\begin{prop}\label{prop:doublepuit}
In the double-well case  with quadratic interaction \eqref{eq:double_well} (with $\theta>0$, $d=1$) with $\sigma^2 <\sigma_c^2$, for any $\varepsilon>0$, there exists $\alpha_\varepsilon\in[0,1)$ and $\oeta_\varepsilon>0$ such that for any $m\geqslant \varepsilon$,
\begin{equation}
\label{eq:Contractdoublewell}
|f(m)-m_+| \leqslant \alpha_\varepsilon |m-m_+|
\end{equation}
and for any $\rho\in\mathcal P_2(\R)$ with $|m(\rho)|\geqslant \varepsilon$,
\begin{equation}
\label{eq:LSIdoublewell}
\mathcal F(\rho) - \inf \mathcal F \leqslant \oeta_\varepsilon \sigma^4 \mathcal I\po \rho|\Gamma(\rho)\pf\,.
\end{equation}
In other words,  for any $\varepsilon>0$,  \eqref{eq:loc_LSInonlineaire} holds over $\mathcal A_{\varepsilon} = \{\rho\in\mathcal P_2(\R), |m(\rho)|\geqslant\varepsilon\}$.
\end{prop}

\begin{rem}\label{rem:decomposition}
In the present setting, Assumption~\ref{assu:total} is satisfied and therefore Theorem~\ref{thm:main} applies so that we get an exponential convergence starting from a $\mathcal W_2$-ball centered on $\rho_{m_+}$ or $\rho_{m_-}$. This gives the first item of  Proposition~\ref{prop:double-well}. 
\end{rem} 

\begin{proof}
Since  $f(m) > m $ for $m\in  (0,m_+)$ and $f(m)<m$ for $m\in (m_+,\infty)$, the function $\alpha$ given by
\[\alpha(m) = \frac{f(m)-m_+}{m-m_+}\]
for $m\neq m_+$ and $\alpha(m_+)=f'(m_+)$ is continuous on $\R_+$ with values in $(0,1)$ on $(0,\infty)$. Moreover, since $f'(m) \rightarrow 0$ as $m\rightarrow \infty$, so does $\alpha$. As a consequence, for any $\varepsilon>0$, $\alpha_\varepsilon:= \sup_{m>\varepsilon} \alpha(m) <1$, which proves \eqref{eq:Contractdoublewell}. \rev{Using that $\mathcal F(\rho_{m_+})=\inf\mathcal F$,}   Proposition~\ref{prop:New_fixed_point} then gives \eqref{eq:LSIdoublewell} for $\rho\in\mathcal P_2(\R)$ with $m(\rho)\geqslant \varepsilon$. The case $m(\rho)\leqslant -\varepsilon$ is obtained by symmetry.
\end{proof}

\begin{rem}
\label{rem:supercritical}
The same arguments work in the super-critical case where $\sigma^2 >\sigma_c^2$ because in that case $f'(0)<1$. In that case the contraction $|f(m)|\leqslant \alpha |m|$ holds globally for $m\in\R$ for some $\alpha<1$. In other words, \eqref{eq:LSInonlineaire} holds for all $\sigma^2 >\sigma_c^2$.
\end{rem}

\begin{rem}\label{rem:explicit}
 To get explicit values for $\delta,\lambda,C$ in Proposition~\ref{prop:double-well}, it is sufficient to compute $f'(m_+)$ and to bound $\|f''\|_\infty$ (as in e.g. \cite{Minibatch}) to get an explicit $\delta>0$ such that  $f'(m) \leqslant (1+f'(m_+))/2=:\alpha <1$ for all $m$ with $|m-m_+|\leqslant \delta$. From this, Theorem~\ref{thm:main} gives explicit estimates in terms of $\sigma^2,\theta,\delta,\alpha$.
\end{rem}

\begin{rem}\label{rem:decomposition2}
Denoting $\mathcal A_{\varepsilon} = \{\rho\in\mathcal P_2(\R), |m(\rho)|\geqslant\varepsilon\}$, the previous result does not mean that $\mathcal F(\rho_t) - \inf\mathcal F \leqslant e^{-t/\oeta_\varepsilon} \po \mathcal F(\rho_t) - \inf\mathcal F \pf$ for all $t\geqslant 0$ as soon as $\rho_0 \in\mathcal A_{\varepsilon}$, because the  non-linear LSI \eqref{eq:LSIdoublewell} does not prevent  $T_{\mathcal A_\varepsilon}(\rho_0) < \infty$. In fact, we provide a counter-example in the next statement. 
\end{rem} 

In the setting of Proposition~\ref{prop:doublepuit}, the precise determination of the basins of attraction of $\rho_{m_+}$ and $\rho_{m_-}$ for the McKean-Vlasov dynamics~\eqref{eq:granular_media_1} remains an open question. The first item of Proposition~\ref{prop:double-well} provides a sufficient condition, together with an exponential rate of convergence. On the other hand, in~\cite{tugaut2013,Bashiri}, it is shown that if $\mathcal{F}(\rho_0)$ is sufficiently small, then the sign of $m(\rho_0)$ suffices to determine the limit of $\rho_t$. To complement these statements, we now show that without this condition on $\mathcal{F}(\rho_0)$, the sign of $m(\rho_t)$ may vary with $t$, so it is not enough to determine the limit of $\rho_t$. To our knowledge, this fact is known empirically but has never been explicitly evidenced.

\begin{prop}\label{prop:contre-Example}
Consider the granular media equation with $d=1$, $V(x)=\frac{x^4}{4}-\frac{x^2}{2}$ and $W(x,y)=\theta(x-y)^2$ for some $\theta>0$. Then, for any $\sigma^2>0$, there exist solutions of~\eqref{eq:granular_media} with $m(\rho_0)>0$ and $m(\rho_{t_0}) < 0$ for some $t_0>0$.
\end{prop}

\begin{proof}

  For $\varepsilon\in[0,1]$, we consider an initial condition $\rho_0 = (1-\varepsilon) \delta_{-1} + \varepsilon \delta_{2/\varepsilon-1}$ and write $m_t = \int_{\R} x\rho_t(x)\dd x$. In particular, $m_0=1$ independently from $\varepsilon$. We consider the time-inhomogeneous SDE
\begin{equation}
\label{eq:eq_contreExample}
\dd X_t = -V'(X_t) \dd t - (X_t - m_t) \dd t  + \sqrt{2}\sigma \dd B_t = \po -X_t^3 + m_t\pf \dd t + \sqrt{2}\sigma \dd B_t \,.
\end{equation}
Let $\rho_t^+$ and $\rho_t^-$ be the law of the solution of \eqref{eq:eq_contreExample} with respective initial conditions $\delta_{2/\varepsilon-1}$ and $\delta_{-1}$. Then, $m_t = (1-\varepsilon) \rho_t^- + \varepsilon\rho_t^+$ for all $t\geqslant 0$. Assume by contradiction that $m_t \geqslant 0$ for all $t\in[0,1]$. Consider the solution of
\[\dd Y_t =   -Y_t^3  \dd t + \sqrt{2}\sigma \dd B_t \]
(with the same Brownian motion as \eqref{eq:eq_contreExample}) with $Y_0 = -1$. Taking $X_0$ distributed according to $m_0$ in \eqref{eq:eq_contreExample} (in particular, $Y_0 \leqslant X_0$ and $X_t\sim m_t$ for all $t\geqslant 0$), we get by monotonicity that $Y_t \leqslant X_t$ for all $t\in[0,1]$, and thus
\[m_t \geqslant \mathbb E(Y_t)\,,\qquad \int_{\R} x^3 \rho_t(\dd x) \geqslant \mathbb E(Y_t^3)\,.\]
 Notice that the law of $Y$ is independent from $\varepsilon$. Now,
 \[\frac{\dd}{\dd t} m_t = - \int_{\R} x^3 \rho_t(\dd x) + m_t \leqslant -\mathbb E(Y_t^3) + m_t\,, \]
 which together with $m_0=1$ shows that $M:=\sup_{\varepsilon\in[0,1]}\sup_{t\in[0,1]} m_t <\infty$. By comparing the solution of \eqref{eq:eq_contreExample} initialized with $X_0=-1$ (so that $X_t\sim \rho_t^-$) to the solution of 
 \[\dd Z_t =   (-Z_t^3+M)  \dd t + \sqrt{2}\sigma \dd B_t \]
 with $Z_0 = -1$ (whose law is again independent from $\varepsilon$), we get that $X_t \leqslant Z_t$ for all $t\in[0,1]$ and since $t\mapsto \mathbb E(Z_t)$ is continuous we get that there exists $t_0\in(0,1]$, independent from $\varepsilon\in[0,1]$, such that
 \[\int_{\R} x \rho_{t_0}^-(\dd x) \leqslant -\frac12\,.\]
Now, writing $h(t) = \int_{\R} x^2 \rho_t^+(\dd x)$, we see that
\[h'(t)  = - 2 \int_{\R} x^4 \rho_t^+(\dd x)+ 2 m_t^2 + 2 \sigma^2 \leqslant -2 h^2(t) + 2(M+\sigma^2)\,.  \]
As a consequence, if $h^2$ goes below $2(M+\sigma^2)$ at some time, it stays below afterwards, and while $h^2$ is above $2(M+\sigma^2)$ it holds
\[h'(t) \leqslant  -h^2(t)\,, \]
hence $h(t) \leqslant (t + 1/h(0))^{-1} \leqslant 1/t$. In any case, for all $t\in[0,1]$,
\[h(t) \leqslant \frac{1}{t} + 2(M+\sigma^2)\,.\]
By Cauchy-Schwarz, 
\[\sup_{\varepsilon\in[0,1]} \int_{\R} x \rho_{t_0}^+ (\dd x) \leqslant \sup_{\varepsilon\in[0,1]}   \sqrt{h(t_0)} \leqslant \sqrt{\frac{1}{t_0} + 2(M+\sigma^2)}\,, \]
and thus
\[m_{t_0} \leqslant -\frac{1-\varepsilon}{2} + \varepsilon \sqrt{\frac{1}{t_0} + 2(M+\sigma^2)}\,,\]
which is negative for $\varepsilon$ small enough, leading to a contradiction with the assumption that $m_t\geqslant 0$ for all $t\in[0,1]$.
\end{proof}

To conclude the discussion in the one-dimensional double-well case, let us notice that the critical temperature $\sigma^2= \sigma_c^2$ provides an example of the degenerate situation addressed in Section~\ref{sec:degenerate}.

\begin{prop}\label{prop:doublewelldegenerate}
In the double-well case  with quadratic interaction \eqref{eq:double_well} (with $\theta>0$, $d=1$) with $\sigma^2 =\sigma_c^2$,  there exist $\beta>0$ and $\oeta>0$ such that for any $m\in\R$,
\begin{equation}
\label{eq:Contractdoublewell-degenerate}
|m| \leqslant  \beta \po  |m-f(m)| + |m-f(m)|^{1/3}\pf\,,
\end{equation}
and for any $\rho\in\mathcal P_2(\R)$
\begin{equation}
\label{eq:LSIdoublewell-degenerate}
\mathcal F(\rho) - \inf \mathcal F \leqslant \oeta   \co \mathcal I\po \rho|\Gamma(\rho)\pf + \po  \mathcal I\po \rho|\Gamma(\rho)\pf \pf^{1/3}\cf  \,.
\end{equation}
\end{prop}

\begin{rem}
By contrast to \eqref{eq:degenerate_condition}, which may only  hold locally, in \eqref{eq:Contractdoublewell-degenerate} we add a linear term, because in this very simple case this inequality in fact holds globally for all $m\in\R $ (and the linear term becomes dominant at long range). This does not change the conclusion on the free energy, in the sense that \eqref{eq:LSIdoublewell-degenerate} is similar to \eqref{eq:degenerateLSI}. 
\end{rem}

\begin{proof}
As in the proof of Proposition~\ref{prop:doublepuit}, we see that, for any $\varepsilon>0$, $\alpha(m) = f(m)/m$ is continuous on $\R_+\setminus[-\varepsilon,\varepsilon]$ with values in $(0,1)$, from which $\alpha_\varepsilon:= \sup_{|m|\geqslant \varepsilon} \alpha(m) <1$. Then, for $m \notin[-\varepsilon,\varepsilon]$, $|m|\leqslant |m-f(m)|+|f(m)| \leqslant |m-f(m)| + \alpha_{\varepsilon} |m|$, which proves that \eqref{eq:Contractdoublewell-degenerate} holds for such $m$ with $\beta = 1/(1-\alpha_\varepsilon)$.

Denoting $s = -f^{(3)}(0)>0$, we get that $f(m) = m - s m^3 + o(m^3)$ as $m\rightarrow 0$. Let $\varepsilon>0$ be such that $|f(m)-m+s m^3| \leqslant s|m|^3 /2$ for all $m \in[-\varepsilon,\varepsilon]$. Then, for such $m$,
\[|m|^3  \leqslant 2\frac{|f(m)-m|}{s}\,,\]
which concludes the proof of \eqref{eq:Contractdoublewell-degenerate}.

The proof of \eqref{eq:LSIdoublewell-degenerate} is then exactly the same as the proof of   \eqref{eq:degenerateLSI} from \eqref{eq:degenerate_condition} \rev{(since $\mathcal F(\rho_{m_*}) = \inf \mathcal F$ with $m_*=0$)}.
\end{proof}

\begin{rem}\label{rem:double-puit-degenere}
The second part of Proposition~\ref{prop:double-well} is a consequence of  Proposition~\ref{prop:doublewelldegenerate}, following the arguments that led to \eqref{eq:EntropyDecayDegenerate} (the only difference being that it is not necessary to check that the solution starts and remains in some neighborhood of the stationary solution since \eqref{eq:LSIdoublewell-degenerate} holds uniformly over $\mathcal P_2(\R^d)$).
\end{rem}

\subsubsection{Localization in the low-temperature regime}\label{sec:localization}

Consider now the granular  media  equation \eqref{eq:granular_media} on $\R^d$ with $W(x,y) = |x-y|^2$ (i.e. we take $\theta=1$ for simplicity, which can always be enforced up to rescaling $V$ and $\sigma^2$). Let $a\in\R^d$ be a non-degenerate local minimizer of $V$ such that $a$ is a global minimizer of $x\mapsto h_0(x)=V(x)+|x-a|^2$. It is proven in \cite{tugaut2014self} that, under additional technical conditions, there exists a family $(\rho_{*,\sigma})_{\sigma>0}$ such that $\mathcal W_2(\rho_{*,\sigma},\delta_a) \rightarrow 0 $ as $\sigma\rightarrow 0$ and, for all $\sigma>0$, $\rho_{*,\sigma}$ is a stationary solution of  \eqref{eq:granular_media}  (at temperature $\sigma$). We want to apply Corollary~\ref{cor:|nabla_f|<1} in this context to get that $\rho_{*,\sigma}$ is stable and get local convergence rates.

In fact we will work in a slightly more general setting. Consider the general case \eqref{eq:EHilbert}. For simplicity, we restrict the study to the finite-dimensional case where $\hil = \R^p$ for some $p\geqslant 1$ and  $R(\psi)=-|\psi|^2$. Moreover, we write $V(x)= V_0(x) + R\po \varphi(x)\pf$. In other words, in all this section,
\begin{eqnarray*}
\mathcal E(\mu) &= & \int_{\R^d} V(x) \mu(\dd x) + R\po \varphi(\mu)\pf - \int_{\R^d} R\po \varphi(x)\pf\mu(\dd x) \\
&=& \int_{\R^d} V(x)\mu(\dd x) + \frac12 \int_{\R^d} |\varphi(x) - \varphi(y)|^2 \mu(\dd x)\mu(\dd y)\,.
\end{eqnarray*}
Denoting by $g_{\sigma}$ the function defined in \eqref{eq:gsigma} to emphasize the dependency in the temperature, we see that, for any $\psi\in\hil$,
\begin{eqnarray*}
g_{\sigma}(\psi) &=& -|\psi|^2 + 2 |\psi|^2  - \sigma^2 \ln \int_{\R^d} e^{- \frac{1}{\sigma^2}\co  V(x) + |\varphi(x)|^2 -2 \psi  \cdot \varphi(x)\cf }\dd x\\
& = &  - \sigma^2 \ln \int_{\R^d} e^{- \frac{1}{\sigma^2}\co  V(x) + |\varphi(x)-\psi|^2 \cf }\dd x\,.
\end{eqnarray*}
As $\sigma\rightarrow 0$,
\[g_{\sigma}(\psi) \rightarrow g_0(\psi) := \inf_{x\in\R^d}\{ V(x) + |\varphi\po x\pf -\psi|^2\} \,, \]
and this convergence holds uniformly for $\psi$ in compact sets of $\R^p$.
We are interested in critical points $\rho_{*,\sigma}$ which are localized at low temperature at some point $a\in\R^d$, in the sense that $\rho_{*,\sigma}\rightarrow \delta_a$ as $\sigma$ vanishes. This implies that $\varphi(\rho_{*,\sigma}) \rightarrow \varphi(a)$. On the other hand, for a fixed $\psi$, if $x\mapsto V(x) + |\varphi(x)-\psi|^2$ has a unique global minimum $x_0$, then $\rho_{\psi}$ converges as $\sigma$ vanishes to a Dirac mass at $x_0$. Hence, we expect $\rho_{*,\sigma}$ to converge to a Dirac mass at the minimizer of $x\mapsto V(x)+|\varphi(x)-\varphi(a)|^2$, and thus we need this minimizer to be $a$. We retrieve the condition of \cite{tugaut2014self}.

\begin{rem}
If we do not assume that $R(\psi)=-|\psi|^2$ but still write $V(x) = V_0(x) +R(\varphi(x))$, the previous computation gives 
\[
g_{\sigma}(\psi) =   - \sigma^2 \ln \int_{\R^d} e^{- \frac{1}{\sigma^2} h(x,\psi) }\dd x\,,
\]
with a modulated energy 
\[h(x,\psi) =   R(\psi) - R(\varphi(x)) + \na R(\psi) \cdot \po\varphi(x)  - \psi\pf  +   V(x) = M\po\psi,\varphi(x)\pf + V(x)  \,, \]
with the notation of Remark~\ref{rem:Bregman}. As $\sigma$ vanishes, $g_\sigma(\psi)$ converges toward $\inf_{x\in \R^d} h(x,\psi)$.
\end{rem}

In view of the previous discussion, we work under the following setting:

\begin{assu}\label{assu:low_temp}
In the case \eqref{eq:EHilbert}, $\hil=\R^p$, $R(\psi)=-|\psi|^2$, $V_0(x) = V(x) - R\po \varphi(x)\pf$, $V\in\mathcal C^3(\R^d,\R)$, $\varphi\in\mathcal C^3(\R^d,\R^p)$, the following conditions hold:
\begin{itemize}
    \item  $a\in\R^d$ is a non-degenerate local minimizer of $V$ (i.e. $\na^2 V(a)>0$) and it is the unique global minimizer of $x\mapsto V(x) + |\varphi(x)-\varphi(a)|^2$.
    \item  $V$ goes to infinity at infinity and $\int_{\R^d} |x|^2 e^{- \beta V(x)} \dd x < \infty$ for some $\beta>0$. 
\end{itemize}

\end{assu}

\begin{prop}\label{prop:multiwell_new}
Under Assumptions~\ref{assu:fixedpoint} and \ref{assu:low_temp}, assuming moreover that \eqref{eq:onesided} holds for some $\kappa_1>0$, there exist $r_0,\sigma_0^2>0$ and a family $(\rho_{*,\sigma})_{\sigma\in(0,\sigma_0]}$ with $\mathcal W_2(\rho_{*,\sigma},\delta_a) \rightarrow 0$ as $\sigma\rightarrow 0$ such that the following holds: for each $\sigma\in(0,\sigma_0]$, $\rho_{*,\sigma}$ is a stationary solution of the McKean-Vlasov equation \eqref{eq:granular_media_2} at temperature $\sigma^2$ and there exist $C,\oeta>0$ such that for all initial condition $\rho_0 \in  \mathcal B_{\mathcal W_2}(\delta_a,r_0)$, along the flow \eqref{eq:granular_media_2}  at temperature $\sigma^2$, for all $t\geqslant 1$,
\begin{equation}
    \label{eq:WHFlocalization}
    \mathcal W_2^2(\rho_t,\rho_{*,\sigma}) + \sigma^2 \mathcal H(\rho_t|\rho_{*,\sigma}) + \mathcal F(\rho_t) - \mathcal F(\rho_{*,\sigma}) \leqslant C e^{-t/\oeta} \mathcal W_2^2(\rho_0,\rho_{*,\sigma})\,.
\end{equation}
    
    Moreover,  if we assume furthermore that $x\mapsto V(x) + |\varphi(x)-\psi|^2$ is strongly convex uniformly over $\psi$ in a neighborhood of $\varphi(a)$, then the previous statement holds with $\oeta,C$ which are independent from  $\sigma\in(0,\sigma_0]$.
\end{prop}

\begin{rem}
In particular, if $V$ has several local minimizers $a_1,\dots,a_n$ and the interaction is strong enough so that $a_j$ is the unique global minimizer of $V+|\varphi-\varphi(a_j)|^2$ for several $j$, then, at a sufficiently low temperature, there are several stable stationary solutions to the granular media equation. By contrast, as discussed in Section~\ref{sec:convex_interaction_faible}, in the same setting, at sufficiently high temperature or sufficiently weak interaction (i.e. multiplying $\varphi$ by a sufficiently small factor $\varepsilon>0$), uniqueness holds. This shows that phase transitions occur as the temperature varies. See \cite{Pavliotis,carrillo2020long} on this topic. 
\end{rem}

\begin{proof}
First, to get the existence of $\rho_{*,\sigma}$, we reason in a first step as in  Proposition~\ref{prop:g_minima}, namely we identify that $g_{\sigma}$ should have a local minimizer close to $\varphi(a)$ for $\sigma$ small enough. Second, in the rest of the proof, we show that  Corollary~\ref{cor:|nabla_f|<1} applies. This  gives a local non-linear LSI and thus a convergence rate thanks to Theorem~\ref{thm:main} (since, as discussed in Remark~\ref{rem:LellL'}, in the present case where $R$ is quadratic, the conditions \eqref{eq:convexity-c_W2} and \eqref{eq:Lipshitz} of Assumption~\ref{assu:total} are implied by Assumption~\ref{assu:fixedpoint}).

\medskip

\noindent\textbf{Step 1.} Fix $r\in(0,1]$ (which later on will be assumed sufficiently small, but always independently from $\sigma$). For $\psi \in\R^p$, set
\[\zeta(\psi,r):= \inf\{V(x)+|\varphi(x)-\psi|^2,\ |x-a|\geqslant r\} \,.\]
Let $M\geqslant 1$ be such that $V(x)\geqslant \sup\{V(y),\ y\in\mathcal B(a,1)\} + 
(\ell+1)^2$ for all $x\in\R^d$ with $|x|\geqslant M$. Then, for $\psi \in \mathcal B(\varphi(a),1)$ and $x,y\in\R^d$ with $|x|\geqslant M$ and $|y-a|\leqslant 1$,
\begin{align*}
    V(y) + |\varphi(y)-\psi|^2 &\leqslant V(y) + \left(|\varphi(y)-\varphi(a)| + |\varphi(a)-\psi|\right)^2\\
    &\leqslant V(y) + \left(\ell |y-a| + |\varphi(a)-\psi|\right)^2 \\
    & \leqslant V(y) + (\ell+1)^2 \\
    & \leqslant  V(x)  \\
    & \leqslant V(x) +  |\varphi(x)-\psi|^2\,. 
\end{align*} 
As a consequence, since $r\in(0,1]$, for  $\psi \in \mathcal B(\varphi(a),1)$,
\begin{eqnarray*}
\zeta(\psi,r) &=& \inf\{V(x)+|\varphi(x)-\psi|^2,\ |x-a|\geqslant r,\ |x|\leqslant M\} \\
& \geqslant & \zeta\po\varphi(a),r\pf - 2 |\psi - \varphi(a)| \sup_{|x|\leqslant M}|\varphi(x)-\varphi(a)|\,,
\end{eqnarray*}
where, for the second line, we have used that for $x\in\mathcal B(0,M)$,
\begin{align*}
    V(x)+|\varphi(x)-\psi|^2 & = V(x)+|\varphi(x)-\varphi(a)|^2 - 2 \po \varphi(x)-\varphi(a)\pf \cdot \po \psi-\varphi(a)\pf   + |\psi- \varphi(a)|^2 \\
    & \geqslant V(x)+|\varphi(x)-\varphi(a)|^2- 2 |\psi - \varphi(a)| \sup_{|x|\leqslant M}|\varphi(x)-\varphi(a)|\,.
\end{align*}
On the other hand, since $a$ is the unique global minimizer of $x\mapsto V(x)+|\varphi(x)-\varphi(a)|^2$, $\zeta(\varphi(a), r) > V(a)$. Thus, we can find $\varepsilon_0>0$ (which depends on $r$) such that for all $\varepsilon\in(0,\varepsilon_0]$ and  $\psi \in \mathcal B(\varphi(a),\varepsilon)$,
\[ \zeta\po \psi,r\pf > V(a) + \varepsilon^2\]
(notice that the condition becomes weaker as $\varepsilon$ decreases). As a consequence, for all  $\psi \in \mathcal B(\varphi(a), \varepsilon )$,
\begin{equation}
\label{eq:g0psi}
    g_0(\psi) = \inf_{x\in\mathcal B(a,r)}\{ V(x) + |\varphi\po x\pf -\psi|^2\}
\end{equation}
(since, for $x\notin \mathcal B(a,r)$, $V(x)+|\varphi(x)-\psi|^2 \geqslant \zeta(\psi,r) > V(a) + |\varphi(x)-\psi|^2$). 
Assuming that $r$ is small enough so that $a$ is the unique global minimizer of $V$ over $\mathcal B(a,r)$, we get that $g_0(\psi)>V(a)=g_0\po \varphi(a)\pf$ for all $\psi \in \mathcal B(\varphi(a), \varepsilon )\setminus \{\varphi(a)\}$, and in particular by continuity
\[\inf\{g_0(\psi),\ |\psi - \varphi(a)|=\varepsilon\} - V(a)>0\,.\]
By uniform convergence on compact sets of $g_{\sigma}$ toward $g_0$ as $\sigma$ vanishes, there exists $\sigma_0>0$ (depending on $\varepsilon$) such that 
\[\kappa := \inf\{g_{\sigma}(\psi),\ |\psi - \varphi(a)|=\varepsilon,\sigma\in(0,\sigma_0]\} - V(a)>0\,,\]
while $g_{\sigma}\po \varphi(a)\pf \rightarrow V(a)$ as $\sigma\rightarrow 0$, and in particular $g_{\sigma}\po \varphi(a)\pf \leqslant V(a)+\kappa/2$ for $\sigma$ small enough. This means that, then, the minimum of $g_{\sigma}$ over $\mathcal B\po \varphi(a),\varepsilon\pf$ is not attained at its boundary, which implies that $g_{\sigma}$ admits a critical point within this ball. Since we can take $\varepsilon$ arbitrarily small, we can take at any $\sigma\in(0,\sigma_0]$ such a critical point $\psi_{*,\sigma}$ in such a way that  $\psi_{*,\sigma} \rightarrow \varphi(a)$ as $\sigma\rightarrow 0$. Thanks to \eqref{eq:nag}, $\rho_{*,\sigma}:= \rho_{\psi_{*,\sigma}}$ is a stationary solution to \eqref{eq:granular_media_2}. This concludes the first step of the proof. 

Before proceeding with Step 2 of the proof, let us discuss a few points. First,
\[\mathcal W_2^2(\rho_{*,\sigma},\delta_a) = \frac{\int_{\R^d} |x-a|^2  e^{- \frac{1}{\sigma^2}\co  V(x) + |\varphi(x)-\psi_{*,\sigma}|^2 \cf }\dd x }{\int_{\R^d}  e^{- \frac{1}{\sigma^2}\co  V(x) + |\varphi(x)-\psi_{*,\sigma}|^2 \cf }\dd x}\,.\]
Using that $\psi_{*,\sigma}\rightarrow \varphi(a)$ and that $a$ is the global minimizer of $V+|\varphi-\varphi(a)|^2$, it is then not difficult to see that $\mathcal W_2^2(\rho_{*,\sigma},\delta_a)$ vanishes with $\sigma$ (see also the end of the proof for details of the argument in a more complicated case).

It remains to show that Corollary~\ref{cor:|nabla_f|<1} applies. To emphasize the dependency on the temperature, we write $f_{\sigma}$ the function \eqref{eq:f}. We have to prove that $|\na f_{\sigma}(\psi_{*,\sigma})|<1$ for $\sigma$ small enough \rev{(thanks to Remark~\ref{rem:minimizer}, this will imply that $\rho_{*,\sigma}$ is a local minimizer of $\mathcal F$)}. From \eqref{eq:u1naR2fu2} and \eqref{eq:u1naR2fu2bis}, 
\[|\na f_{\sigma}(\psi_{*,\sigma})| = \frac{2}{\sigma^2} \sup_{u\in\mathbb S^{p-1}} \mathrm{var}_{\rho_{*,\sigma}}\po u\cdot \varphi(X)\pf \,.\]
Note that by the variational characterization of the variance, for $u\in\mathbb S^{p-1}$,
 \[\mathrm{var}_{\rho_{*,\sigma}}\po u\cdot \varphi(X)\pf \leqslant    \int_{\R^d} |u\cdot \po \varphi(x)-\varphi(z)\pf |^2 \rho_{*,\sigma}(\dd x) \]
for any $z\in\R^d$. Having in mind the Laplace approximation of $\rho_{*,\sigma}$ as $\sigma$ vanishes, it is natural to take $z$ as the minimizer of $x\mapsto h_{\sigma}(x):= V(x) + |\varphi(x)-\psi_{*,\sigma}|^2$. In the rest of the proof, first, we justify that, for $\sigma$ small enough, $h_{\sigma}$ admits a unique minimizer $z_\sigma$, which converges to $a$ as $\sigma$ vanishes (this is Step 2) and then (in Step 3), we apply the Laplace approximation (which still works although $h_\sigma$ and $z_\sigma$ depends on $\sigma$) to get that,
\begin{equation}\label{eq:lim2/sigma}
\lim_{\sigma\rightarrow 0}\frac{2}{\sigma^2} \sup_{u\in\mathbb S^{p-1}}   \int_{\R^d} |u\cdot \po \varphi(x)-\varphi(z_\sigma)\pf |^2 \rho_{*,\sigma}(\dd x) < 1\,.
\end{equation}
This implies that the same holds for $\sigma$ small enough, which concludes the proof. The technical justification of the Laplace approximation is postponed to Step 4.

\medskip

\noindent\textbf{Step 2.} We write $\na \varphi=(\partial_{x_i} \varphi_j)_{(i,j)\in \cco 1,d\ccf\times\cco 1,p\ccf}$, where $i$ (resp. $j$) is the index of the line (resp. the column).  Denoting $h_0(x)=V(x)+|\varphi(x)-\varphi(a)|^2$, we compute that 
\begin{eqnarray*}
\na^2 h_\sigma(x) &=& \na^2 V(x) + 2 \na \varphi(x) \po \na \varphi(x)\pf^T + \sum_{j=1}^p  \na^2 \varphi_j(x) \po \varphi_j(x)-\psi_{j,*,\sigma}\pf \\
&= & \na^2 h_0(x) + \epsilon(x,\sigma)
\end{eqnarray*}
with 
\[\epsilon(x,\sigma)  = \sum_{j=1}^p  \na^2 \varphi_j(x) \po \varphi_j(a)-\psi_{j,*,\sigma}\pf \,, \]
so that $\epsilon(x,\sigma)$ vanishes with $\sigma$ uniformly over $x\in \mathcal B(a,r)$. Besides,
\[\na^2 h_0(a) = \na^2 V(a) + 2  \na \varphi(a) \po \na \varphi(a)\pf^T \]
is positive definite, and thus we can assume that $r$ is small enough so that $\na^2 h_0(x)$ is uniformly bounded below by a positive constant over $x\in\mathcal B(a,r)$. Thanks to the uniform convergence of $\na^2 h_{\sigma}$ to $\na^2 h_0$ over $\mathcal B(a,r)$, there is $\sigma_0',\kappa>0$ such that for all $\sigma\in(0,\sigma_0']$ and $x\in\mathcal B(a,r)$, $\na^2 h_\sigma(x) \geqslant 2\kappa$. On the other hand, we have seen when establishing \eqref{eq:g0psi} that there exists $\kappa'>0$ such that for $\sigma$ small enough, 
\begin{equation}
    \label{eq:kappa'}
    \inf\{h_{\sigma}(x),\ |x-a|\geqslant r\} - \inf\{h_{\sigma}(x),\ |x-a|\leqslant r\} \geqslant \kappa'
\end{equation}
(here we use that for any $\varepsilon>0$, $\psi_{*,\sigma} \in\mathcal B(\varphi(a),\varepsilon)$ for $\sigma$ small enough). The infimum of $h_\sigma$ is thus attained (for $\sigma$ small enough) in $\mathcal B(a,r)$, where it is strongly convex, and thus it admits a unique global minimizer $z_\sigma$. Since $r$ can be taken arbitrarily small (which changes the threshold $\sigma_0'$), we get that $z_{\sigma} \rightarrow a$ as $\sigma $ vanishes. Moreover, for $\sigma$ small enough,  for all $x\in\mathcal B(a,r)$,
\begin{equation}
    \label{eq:hkappa}
    h_\sigma(x) \geqslant h_{\sigma}(z_\sigma) + \kappa |x-z_{\sigma}|^2\,.
\end{equation}
  Since $h_\sigma$ admits a unique global maximum which is non-degenerate, the Laplace approximation heuristics suggest that expectations with respect to $\rho_{*,\sigma} \propto \exp(-\frac{1}{\sigma^2} h_\sigma)$ are equivalent to expectations with respect to the Gaussian measure $\hat \rho_\sigma\propto \exp(-\frac{1}{\sigma^2} \hat h_\sigma)$ with
\[\hat h_{\sigma}(x) = h_{\sigma}(z_\sigma) + \frac12(x-z_{\sigma})^T \na^2 h_{\sigma}(z_\sigma) (x-z_{\sigma})\,.\]
As mentioned above, we postponed this technical justification to the end of the proof and, for now, take for granted that
\begin{equation}
    \label{eq:Laplaceapprox}
    \int_{\R^d}   \po \varphi-\varphi(z_\sigma)\pf \po \varphi-\varphi(z_\sigma)\pf ^T  \rho_{*,\sigma} \underset{\sigma\rightarrow 0}\sim \int_{\R^d}   \po \varphi-\varphi(z_\sigma)\pf \po \varphi-\varphi(z_\sigma)\pf ^T  \hat \rho_{\sigma} \,.
\end{equation}

\noindent\textbf{Step 3.} From this approximation, we have now to study the limit in \eqref{eq:lim2/sigma}. First, writing
\begin{align*}
    &\frac{1}{\sigma^2}\int_{\R^d} (\varphi(x)-\varphi(z_\sigma))(\varphi(x)-\varphi(z_\sigma))^T \hat{\rho}_\sigma(\dd x)\\
    &= \mathbb{E}\left[\left(\frac{\varphi(z_\sigma + \sigma \sqrt{D_\sigma^{-1}} G)-\varphi(z_\sigma)}{\sigma}\right)\left(\frac{\varphi(z_\sigma + \sigma \sqrt{D_\sigma^{-1}} G)-\varphi(z_\sigma)}{\sigma}\right)^T\right],
\end{align*}
with $G$ a standard $d$-dimensional Gaussian variable and  $D_\sigma := \nabla^2 h_\sigma(z_\sigma)$, the almost sure convergence
\begin{equation*}
    \lim_{\sigma \to 0} \frac{\varphi(z_\sigma + \sigma \sqrt{D_\sigma^{-1}} G)-\varphi(z_\sigma)}{\sigma} = (\nabla \varphi(a))^T \sqrt{D_0^{-1}} G,
\end{equation*}
together with the bound
\begin{equation*}
    \left|\frac{\varphi(z_\sigma + \sigma \sqrt{D_\sigma^{-1}} G)-\varphi(z_\sigma)}{\sigma}\right|^2 \leq \ell^2 \left|\sqrt{D_\sigma^{-1}} G\right|^2 \underset{\sigma\rightarrow 0}{\longrightarrow}  \ell^2 \left|\sqrt{D_0^{-1}} G\right|^2 
\end{equation*}
give by dominated convergence that
\begin{align}
    &\lim_{\sigma \to 0} \mathbb{E}\left[\left(\frac{\varphi(z_\sigma + \sigma \sqrt{D_\sigma^{-1}} G)-\varphi(z_\sigma)}{\sigma}\right)\left(\frac{\varphi(z_\sigma + \sigma \sqrt{D_\sigma^{-1}} G)-\varphi(z_\sigma)}{\sigma}\right)^T\right]\nonumber \\
    &= \mathbb{E}\left[\left((\nabla \varphi(a))^T \sqrt{D_0^{-1}} G\right)\left((\nabla \varphi(a))^T \sqrt{D_0^{-1}} G\right)^T\right] \nonumber\\
    &= (\nabla \varphi(a))^T D_0^{-1} \nabla \varphi(a)\\
    &= A^T \po 2 A A^T + \na^2 V(a) \pf^{-1} A\,,\label{eq:A2AAVA}
\end{align}
where $A:=\na \varphi(a)$. 
Thanks to \eqref{eq:Laplaceapprox},
\[\frac{2}{\sigma^2} \int_{\R^d}   \po \varphi(x)-\varphi(z_\sigma)\pf \po \varphi(x)-\varphi(z_\sigma)\pf ^T  \rho_{*,\sigma}(\dd x) \underset{\sigma\rightarrow0}{\longrightarrow} 2 A^T \po 2 A A^T + \na^2 V(a) \pf^{-1} A\,. \] 

It remains to show that this limit is strictly smaller than $I_p$ (in the sense of quadratic forms), from which the left-hand side will be uniformly strictly less than $I_p$ for $\sigma$ small enough, which will conclude.

Let $u\in\mathbb S^{p-1}$ and $h=(2AA^T + \na^2 V(a))^{-1} Au$.  Then
\[h^TAu = h^T  (2AA^T + \na^2 V(a))  h \geqslant (2+c) |A^T h|^2\]
for some $c>0$, using that $\na^2 V(a)$ is definite positive. Then, 
\[h^TAu   \leqslant |u| |A^T h | \leqslant \frac{1}{\sqrt{2+c}}|u| \sqrt{h^TAu }\,,\]
and thus
\[
2 u^T A^T \po 2 A A^T + \na^2 V(a) \pf^{-1} A u
=2 u^T A^T h   \leqslant \frac{2}{2+c}\,.\]
As discussed above, this concludes the proof.

Notice that $c$ is independent from $\sigma$, and thus, in Proposition~\ref{prop:New_fixed_point}, 
\eqref{eq:contract_fixedpoint} holds on $\mathcal A'=\mathcal B\po \varphi(a),r\pf$ with $r>0$ and $\alpha\in[0,1)$ independent from $\sigma\in(0,\sigma_0'']$ for some $\sigma_0''>0$. If, moreover, $x\mapsto V(x) + |\varphi(x)-\psi|^2$ is $k$-strongly convex for some $k>0$ for all $\psi \in\mathcal A'$, the Bakry-Emery criterion shows that $\Gamma(\rho)$ satisfies an LSI with constant $\eta = k \sigma^2$ for all $\rho$ such that $\varphi(\rho) \in \mathcal A'$, and thus for all $\rho \in \mathcal A:= \mathcal B_{\mathcal W2}(\delta,r_0)$ provided $r_0$ is small enough (independently from $\sigma$). Following the proof of Proposition~\ref{prop:New_fixed_point} (which in fact does not require \eqref{eq:LSIlineaire_locale} to hold  globally on $\mathcal P(\R^d)$, but only on $\mathcal A$), we obtain for $\rho\in \mathcal A$ the local non-linear LSI
\[\mathcal F(\rho) - \mathcal F(\rho_{*,\sigma}) \leqslant k \sigma^4 \po 1   +  \frac{4 k  \theta \ell^2}{(1-\alpha)^2}\pf\mathcal I\po \rho|\Gamma(\rho)\pf\,,\]
where $\theta,\ell,\alpha,k$ are independent from $\sigma\in(0,\sigma_0'']$, i.e. \eqref{eq:loc_LSInonlineaire} with $\oeta$ independent from $\sigma$. Similarly, using that $\eta=k\sigma^2$, we get that $q_1$ defined in \eqref{eq:qt} with $t=1$ is bounded independently from $\sigma$ small enough, for initial conditions $\rho_0 \in \mathcal A$. Applying Theorem~\ref{thm:main} shows that \eqref{eq:WHFlocalization} holds with $C,\oeta$ independent from $\sigma\in(0,\sigma_0'']$ in this case (again, the proof only uses an LSI for $\Gamma(\rho)$ uniformly over $\mathcal A$, not over $\mathcal P(\R^d)$).

\medskip

\noindent\textbf{Step 4.} We now turn to the justification of \eqref{eq:Laplaceapprox}. Write
$\chi_0(x)=1$ and, omitting the dependency in $\sigma$, $\chi_1(x)= \po \varphi(x)-\varphi(z_\sigma)\pf \po \varphi(x)-\varphi(z_\sigma)\pf ^T $ (notice that $|\chi_1(x)|\leqslant C|x|^2+C$ for some $C$ independent from $\sigma\in(0,\sigma_0]$). For $i\in\{0,1\}$ and $k\in\{1,2,3,4\}$ we write
\[I_{i,k} = \int_{A_k} \chi_i   \exp\po -\frac{1}{\sigma^2}\co h_\sigma -  h_\sigma(z_\sigma)\cf \pf \,,\qquad \hat I_{i,k} = \int_{A_k} \chi_i  \exp\po -\frac{1}{\sigma^2}\co \hat h_\sigma -  h_\sigma(z_\sigma)\cf \pf  \,,\]
with
\[
\begin{array}{ll}
    A_1 = \{|x-z_{\sigma}|\leqslant \sqrt{\sigma}\}\qquad & A_2 = \{\sqrt{\sigma} < |x-z_{\sigma}|\leqslant 2r \}\\
A_3 = \{2r < |x-z_{\sigma}|\leqslant M\}\qquad 
&A_4 = \{M < |x-z_{\sigma}| \}\,,
\end{array}\]
where $r$ is small enough for  \eqref{eq:kappa'} to hold for some $\kappa'$  and for  \eqref{eq:hkappa} to hold for all $x\in\mathcal B(a,3r)$ for some $\kappa$ (with $r,\kappa,\kappa'>0$ independent from $\sigma$) for all $\sigma$ small enough, and $M>1$ is large enough so that 
\begin{equation}
    \label{eq:M2}
    \forall x\notin \mathcal B(a,M-1)\,,\qquad \frac12 V(x) \geqslant V(a)+2
\end{equation}
For $\sigma$ small enough, $|z_\sigma-a|<r$, and thus $\mathcal B(a,r) \subset \mathcal B(z_\sigma,2r)  $, which together with 
\eqref{eq:kappa'} gives that 
\[    \inf\{h_{\sigma}(x),\ |x-z_\sigma |\geqslant 2 r\}  \geqslant  h_\sigma(z_{\sigma}) + \kappa'\,. \] 
As a consequence, for $i\in\{0,1\}$, using that $\varphi$ is bounded uniformly over $\mathcal B(a,M+1)$ (hence over $\mathcal B(z_{\sigma},M)$ for $\sigma$ small enough)
\[I_{i,3}  = \underset{\sigma\rightarrow 0}{\mathcal O}\po e^{-\frac{\kappa'}{\sigma^2}} \pf\,.  \]
Similarly, for $\sigma$ small enough, $\mathcal B(z_\sigma,2r)\subset \mathcal B(a,3r)$ and thus, thanks to \eqref{eq:hkappa}  for $x\in\mathcal B(a,2r)$, using that $\varphi$ is bounded uniformly over $\mathcal B(a,3r)$, for $i\in\{0,1\}$,
\[I_{i,2}   = \underset{\sigma\rightarrow 0}{\mathcal O}\po e^{-\frac{\kappa}{\sigma}} \pf\,.  \]
Next, using that $\mathcal B(a,M-1)\subset \mathcal B(z_\sigma,M)$ and $V(a)+1 \geqslant h_\sigma(z_\sigma) $ for $\sigma$ small enough, \eqref{eq:M2} implies that
\[    \forall x\notin \mathcal B(z_\sigma,M)\,,\qquad   h_\sigma(x) \geqslant V(x) \geqslant \frac12 V(x)+1+ h_\sigma(z_\sigma)\,.\] 
This gives, for $\sigma\leqslant 1/(2\beta)$,
\[I_{i,4} \leqslant e^{-\frac{1}{\sigma^2}} \int_{|x-z_\sigma|> M}  |\chi_i(x)| e^{-\frac{V(x)}{2\sigma^2} }\dd x  \leqslant e^{-\frac{1}{\sigma^2}} \int_{\R^d}  |\chi_i(x)| e^{-\beta V(x) }\dd x  = \underset{\sigma\rightarrow 0}{\mathcal O}\po e^{-\frac{1}{\sigma^2}} \pf\,, \]
since $e^{-\beta V}$ admits a second moment (here for simplicity we have assumed without loss of generality that $V\geqslant 0$).

Finally, using that $\na^{(3)} h_\sigma$  converges $\na^{(3)} h_0$ uniformly over compact sets, we get that there exists $C>0$ such that
\[|h_\sigma(x) - \hat h_\sigma(x)| \leqslant C |x-z_\sigma|^3\]
for all $x\in\mathcal B(a,1)$ (hence all $x\in\mathcal B(z_\sigma,\sqrt{\sigma})$) for $\sigma$ small enough. From this, for $i\in\{0,1\}$,
\[I_{i,1} = \hat I_{i,1} \po 1+ \underset{\sigma\rightarrow 0}{\mathcal O}\po\sigma  \pf\pf\,.   \] 
On the other hand, denoting $c_\sigma = (2\pi)^{d/2}\sqrt{\mathrm{det}[\po \na^2 h_\sigma(z_\sigma)\pf^{-1}]}$ (which converges to some $c_0>0$ as $\sigma\rightarrow 0$),
 by usual computations for Gaussian distributions,
 \[
 I_0 := \sum_{k=1}^4 I_{0,k}  =    \hat I_{0,1} \po 1+ \underset{\sigma\rightarrow 0}{\mathcal O}\po\sigma  \pf\pf +  \underset{\sigma\rightarrow 0}{\mathcal O}\po e^{-\frac{\kappa}{\sigma}} \pf = \sigma^d c_\sigma +  \underset{\sigma\rightarrow 0}{\mathcal O}\po \sigma^{d+1} \pf  \,,\]
and then
\begin{eqnarray*}
\frac{\sum_{k=1}^4 I_{1,k}}{I_0} & =& \po \sigma^d c_\sigma +  \underset{\sigma\rightarrow 0}{\mathcal O}\po \sigma^{d+1} \pf \pf^{-1} \po  \hat I_{1,1} \po 1+ \underset{\sigma\rightarrow 0}{\mathcal O}\po\sigma  \pf\pf +  \underset{\sigma\rightarrow 0}{\mathcal O}\po e^{-\frac{\kappa}{\sigma}} \pf\pf  \\
 & = &  \frac{\hat I_{1,1}}{\sigma^d c_\sigma } \po 1+ \underset{\sigma\rightarrow 0}{\mathcal O}\po\sigma  \pf\pf +  \underset{\sigma\rightarrow 0}{\mathcal O}\po e^{-\frac{\kappa}{2\sigma}} \pf \\
 & = &  \frac{\sum_{k=1}^4 \hat I_{1,k}}{\sigma^d c_\sigma } \po 1+ \underset{\sigma\rightarrow 0}{\mathcal O}\po\sigma  \pf\pf +  \underset{\sigma\rightarrow 0}{\mathcal O}\po e^{-\frac{\kappa}{2\sigma}} \pf \\
 & = &  \sigma^2  \po\na \varphi(a)\pf^T  \po \na^2 h_{0}(a)\pf^{-1}  \na \varphi(a)  + \underset{\sigma\rightarrow 0}o\po \sigma^2\pf\,,
\end{eqnarray*}
as we computed in \eqref{eq:A2AAVA}. This concludes the proof of \eqref{eq:Laplaceapprox}, hence of Proposition~\ref{prop:multiwell_new}. 
\end{proof}

\begin{rem}\label{rem:theta>-V''}
In the case \eqref{eq:granular_media} with $W(x,y)=\theta|x-y|^2$, taking $\theta > - \inf \na^2 V/2$ (as in \cite[Theorem 2.3]{tugaut2023steady}), we get that  \eqref{eq:WHFlocalization} holds with $\oeta,C$ that are uniform over $\sigma$ small enough in a neighborhood of $\delta_a$. At first this seems to be a desirable property for optimization (and in fact this is the idea underlying consensus-based optimization). This should be mitigated by the following observations.

First, this condition implies that $x\mapsto b_a(x)= V(x) + \theta|x-a|^2$ is strongly convex for all $a$ and thus, any local minimizer  $a$ of $V$ being a critical point of $b_a$, it is then the unique global minimizer of $b_a$. Hence, the Wasserstein gradient descent can be trapped in any local well of $V$ (and thus it is not very different from a basic gradient descent for $V$ in $\R^d$). By contrast, if $\theta$ is taken smaller,  the condition that $a$ has to be a global minimizer of $b_a$ for a localization to occur  can be used as a way to  discard shallow local wells and select better solutions.

Besides, the decay  \eqref{eq:WHFlocalization} holds for solutions initialized in $\mathcal B_{\mathcal W_2}(\delta_a,r_0)$ for some small $r_0$. Now, assume for instance that we start close to $\delta_{a'}$ where $a'$ is a local minimizer of $V$ but not a global minimizer of $b_{a'}$, and $a\neq a'$ is the global minimizer of $V$ and also the global minimizer of $b_{x}$ for all $x\in\R^d$. For $\sigma$ small enough, we expect all solutions to converge to a stationary solution close to $\delta_a$. In particular, eventually, any trajectory will lie in  $\mathcal B_{\mathcal W_2}(\delta_a,r_0)$ and thus the convergence rate $\oeta$ in \eqref{eq:WHFlocalization} is independent from $\sigma$, as in Remark~\ref{rem:Ainfini}. However the constant $C$ highly depends on the initial condition, reflecting the time needed to reach  $\mathcal B_{\mathcal W_2}(\delta_a,r_0)$. Starting close to $\delta_{a'}$, the additional convexity due to the interaction will increase the energy barrier the process has to overcome to move from $a'$ to $a$. More quantitatively, the Arrhenius law indicates that the typical time for the solution to put most of its mass around $a$ will be of order $e^{D/\sigma^2}$ where $D$ is larger than in the case $\theta=0$. In other words, an attractive interaction worsen the metastability of the process. If the goal is not to improve the local convergence rate but to enhance the exploration of the space by lowering energy barriers (as in the Adaptive Potential algorithms and related methods~\cite{ABP}), repulsive interaction (i.e. $\theta<0$) seems more indicated, as studied in \cite{CdRandco}.
\end{rem}

\section{Vlasov-Fokker-Planck equation}\label{sec:VFP}

In this section, we consider the kinetic Vlasov-Fokker-Planck equation, which reads 
\begin{equation}
\label{eq:VFP}
\partial_t \rho_t + v\cdot \na_x \rho_t  = \na_v\cdot \po \sigma^2\nabla_v \rho_t + \po v+ \na E_{\rho_t^x}\pf \rho_t\pf\,,
\end{equation}
where $\rho_t(x,v)$ is the probability density of particles at position $x\in\R^d$ with velocity $v\in\R^d$, $\rho^x(x)=\int_{\R^d} \rho(x,v)\dd v$ is the marginal density of the position, $E_{\rho^x}$ is the linear functional derivative of an energy $\mathcal E$ as in Section~\ref{sec:settings} (it only depends on the position $x$ and the marginal density $\rho^x$). 

Entropic long-time convergence \rev{rates} for \eqref{eq:VFP} have been established in some cases (with the energy corresponding to the granular media case \eqref{eq:def_F}) in \cite{MONMARCHE20171721,guillin2021uniform} (using uniform LSI for the associated particle system), \cite{Songbo2} (with global non-linear LSI) or \cite{ren2021exponential} (with uniform conditional LSI for the associated particle system and a smallness assumption on the interaction). All these works establish global convergence toward a unique stationary solution and thus do not cover cases with several stationary solutions. \rev{In this latter situation, qualitative convergence to the set of stationary solutions is studied in \cite{DuongTugaut}.}

The Vlasov-Fokker-Planck is not the gradient flow of some free energy functional. However, we can still get a local convergence under the same assumptions as in the elliptic case \eqref{eq:granular_media}. This is the main result of this section, stated in Theorem~\ref{thm:mainVFP} below.

The free energy associated to \eqref{eq:VFP} is
\[\mathcal F_{k}(\rho ) = \sigma^2 \mathcal H(\rho) + \mathcal E(\rho^x) + \mathcal E_{k}(\rho)\]
with the kinetic energy 
\[\mathcal E_{k}(\rho) = \int_{\R^{2d}} \frac{|v|^2}{2}\rho (x,v)\dd x\dd v\,.\]
In this section, Assumptions~\ref{assu:F}, \ref{assu:lfdE} and \ref{assu:loc-eq}  remain in force (which implies that $\mathcal F_k$ is also lower bounded, see Remark~\ref{rem:elliptic_to_kinetic} below). The corresponding local equilibria are given by
\begin{equation}
\label{eq:GammaVFP}
\Gamma_k(\rho) = \Gamma(\rho^x) \otimes \mathcal N(0, \sigma^2 I_d)
\end{equation}
where $\Gamma(\rho^x)$ is the corresponding overdamped local equilibrium \eqref{eq:Gamma} and $\mathcal N(0,\sigma^2 I_d)$ stands for the centered  $d$-dimensional Gaussian distribution with variance $\sigma^2I_d$. In other words,
\[\Gamma_k(\rho) \propto \exp \po - \frac{1}{\sigma^2}\po  E_{\rho^x}(x) + \frac{1}{2}|v|^2\pf \pf \dd x \dd v \,. \]
 Along the flow \eqref{eq:VFP}, under suitable regularity conditions, 
\begin{equation}
\label{eq:dissipationVFP}
\frac{\dd}{\dd t}  \mathcal F_{k}(\rho_t) = - \sigma^4 \int_{\R^{2d}} \left| \na_v \ln \frac{\rho_t}{\Gamma_k(\rho_t)} \right|^2 \dd \rho_t  \,.
\end{equation}
Hence, $t\mapsto \mathcal F_k(\rho_t)$ is non-increasing, but the free energy dissipation can vanish even if $\rho\neq \Gamma_k(\rho)$. To avoid ambiguity we write $\mathcal K_k = \{\rho_*\in\mathcal P_2(\R^{2d}),\ \mathcal F_k(\rho_*)<\infty,\  \Gamma_k(\rho_*)=\rho_*\}$. There is a one-to-one correspondence between $\mathcal K_k$ (the set of critical points of $\mathcal F_k$ in $\mathcal P_2(\R^{2d})$) and $\mathcal K$ (the set of critical points of $\mathcal F=\sigma^2 \mathcal H + \mathcal E$ in $\mathcal P_2(\R^{d})$), since $\rho_*\in\mathcal K_k$ if and only if $\rho_*^x \in\mathcal K$ and $\rho_*=\rho_*^x \otimes \mathcal{N}(0,\sigma^2 I_d)$. 

As in the elliptic case, we do not address well-posedness issues and work under the following conditions.

 \begin{assu}\label{assu:general_kinetic}
 For all $\rho_0\in\mathcal P_2(\R^{2d})$, \eqref{eq:VFP} has a unique strong solution, continuous in time for $\mathcal W_2$, such that, for all $t>0$, $\rho_t$ has a continuous positive density, $\mathcal F_k(\rho_t)$, $\mathcal H(\rho_t|\Gamma_k(\rho_t))$ and  $\mathcal I(\rho_t|\Gamma_k(\rho_t))$ are finite and \eqref{eq:dissipationVFP} holds, with $t\mapsto \mathcal I(\rho_t|\Gamma(\rho_t))$ continuous over $\R_+^*$. 
 \end{assu}
 
 To recover an entropy dissipation with a  full Fisher information instead of~\eqref{eq:dissipationVFP}, we work with a modified free energy of the form 
 \begin{equation}
 \label{eq:defLa}
\mathcal L(\rho_t) = \mathcal F_{k}(\rho_t)  +  a \sigma^4  \int_{\R^{2d}} \left|(\na_x+\na_v) \ln \frac{\rho_t}{\Gamma_k(\rho_t)} \right|^2 \dd\rho_t\,,
 \end{equation}
 for some $a>0$ to be fixed.
 
 \begin{lem}\label{lem:dissipationVFP}
  Assume~\eqref{eq:Lipshitz} and that 
\begin{equation}
\label{eq:M}
M := \sup_{\rho\in\mathcal P_2(\R^d)} \|\na^2 E_{\rho}\|_\infty <\infty\,.\tag{$x-$\textbf{Lip}}
\end{equation}
Then, taking $a=\frac12 \po 1+ 2M^2 + 2\frac{L^2}{\sigma^4} \pf ^{-1}$ in the definition of $\mathcal L$, for all $t> 0$,
\[\frac{\dd}{\dd t} \mathcal L(\rho_t) \leqslant - \frac{\sigma^4 a}{2} \mathcal I\po \rho_t|\Gamma_k(\rho_t)\pf\,.\]
 \end{lem}
 \begin{proof}
This is a computation from the proof of \cite[Theorem 2.1]{Songbo2}, that we recall now  (we follow a presentation similar to the proof of \cite[Proposition 3]{NonSymVFP}). For conciseness, write $\hat \rho_t = \Gamma_k(\rho_t) $, $h_t = \rho_t/\hat \rho_t$ and let $A$ be  a $2d\times2d$ matrix of the form  
 \[A= \begin{pmatrix}
a_{11} I_d & a_{12} I_d \\ a_{21}I_d & a_{22} I_d
\end{pmatrix}\,.\]
For $t_1, t_2 > 0$, set
\[
	u(t_1,t_2) = \int_{\R^{2d}} \left|A \na \ln \frac{\rho_{t_1}}{\hat{\rho}_{t_2}} \right|^2 \dd\rho_{t_1},
\]
with $\nabla = \begin{pmatrix} \nabla_x\\ \nabla_v\end{pmatrix}$. For $t>0$, we decompose
\[
	\frac{\dd}{\dd t} \int_{\R^{2d}} \left|A \na \ln h_t \right|^2 \dd\rho_t = \partial_{t_1} u(t,t) + \partial_{t_2} u(t,t) = (\star)+(\star\star)\,.
\]
\rev{\emph{Step 1 : bounding ($\star$).}} Since $\hat \rho_t$ is the invariant measure of the linear kinetic Fokker-Planck equation with fixed energy equal to $E_{\rho_t^x}$ with generator 
\[L_{\rho_t}   =  v\cdot \na_x- (\na E_{\rho_t^x}+ v)\cdot\na_v +  \sigma^2 \Delta_v \,,\]
 the first part is treated with classical computations (e.g. \cite[Lemma 8]{Gamma}, and below) leading to 
\begin{equation}
\label{eq:hypocoercivite}
(\star) \leqslant  2 \int_{\R^{2d}}  \na \ln h_t \cdot A^T A J   \na  \ln h_t  \dd\rho_t \,,
\end{equation}
where
\[J(x) = \begin{pmatrix}
0 &   \na^2 E_{\rho_t^x}(x) \\- I_d & -I_d 
\end{pmatrix}  =: J_0 + \begin{pmatrix}
0 &   \na^2 E_{\rho_t^x}(x) \\0 & 0 
\end{pmatrix} \]
is the Jacobian matrix of the drift of 
\[L_{\rho_t}^* =- v\cdot \na_x + (\na E_{\rho_t^x}- v)\cdot\na_v + \sigma^2  \Delta_v\,,\]
the dual of $L_{\rho_t}$ in $L^2 (\hat \rho_t)$. For the reader's convenience, let us give a self-contained proof of \eqref{eq:hypocoercivite} (keeping computations formal and referring to \cite{Songbo2} for technical justifications). We start by some preliminary considerations. First, notice that, for a smooth function $\varphi$ with compact support,
\[\int_{\R^{2d}} \varphi \partial_t \rho_t = \int_{\R^{2d}} L_{\rho_t} \varphi \rho_t = \int_{\R^{2d}}  \varphi  L_{\rho_t}^* \po \frac{\rho_t}{\hat \rho_t}\pf \hat\rho_t\,.   \]
This being true for an arbitrary $\varphi$, this amounts to
\begin{equation}
\label{eq:hypoco1}
\frac{\partial_t \rho_t}{\hat \rho_t} = L_{\rho_t}^*h_t\,,\qquad \text{hence}\qquad \frac{\partial_t \rho_t}{ \rho_t} = \frac{L_{\rho_t}^*h_t}{h_t}\,. 
\end{equation}
Moreover, denoting for smooth functions $f,g$
\[\Gamma(f,g) = \sigma^2 \na_v f\cdot \na_v g\,,\]
which for all $t\geqslant 0$ is the carr\'e du champs associated to $L_{\rho_t}$ since
\[\Gamma(f,g) = \frac12 \co L_{\rho_t}(fg) - f L_{\rho_t}(g)- gL_{\rho_t}(f)\cf\]
(and similarly if $L_{\rho_t}$ is replaced by $L_{\rho_t}^*$), we see that, for a positive smooth $f$,
\begin{equation}
\label{eq:hypoco-log}
L_{\rho_t}^* \po \ln f\pf = \frac{L_{\rho_t}^* f }{f} - \frac{\Gamma(f,f)}{f^2} = \frac{L_{\rho_t}^* f }{f} -\Gamma(\ln f)\,.
\end{equation}  
Next, using that $\hat{\rho}_t$ is invariant by $L_{\rho_t}$, it holds $\int_{\R^{2d}} L_{\rho_t} \varphi \hat\rho_t =  0$ for all nice $\varphi$. In particular, for any $g$,
\begin{equation}\label{eq:intLrho}
    \begin{aligned}
        \int_{\R^{2d}} L_{\rho_t}(g) \rho_t & =  \int_{\R^{2d}} L_{\rho_t}(g) h_t \hat \rho_t  \\
        & = - \int_{\R^{2d}} \co g  L_{\rho_t}( h_t ) + 2\Gamma(g,h_t)\cf  \hat\rho_t  \ = \ - \int_{\R^{2d}} \co h_t  L_{\rho_t}^*( g) + 2\Gamma(g,h_t)\cf  \hat\rho_t.
    \end{aligned}
\end{equation}
Now, since for $t_1, t_2 > 0$,
\begin{align*}
	\partial_{t_1} u(t_1,t_2) &= \int_{\R^{2d}} \left(\frac{\dd}{\dd t_1} \left|A \nabla \ln \frac{\rho_{t_1}}{\hat{\rho}_{t_2}}\right|^2\right) \rho_{t_1} + \int_{\R^{2d}} \left|A \nabla \ln \frac{\rho_{t_1}}{\hat{\rho}_{t_2}}\right|^2 \partial_t \rho_{t_1}\\
	&= \int_{\R^{2d}} \left(2 A \nabla \ln \frac{\rho_{t_1}}{\hat{\rho}_{t_2}} \cdot A \nabla \frac{\partial_t \rho_{t_1}}{\hat{\rho}_{t_2}}\right) \rho_{t_1} + \int_{\R^{2d}} L_{\rho_{t_1}}\left(\left|A \nabla \ln \frac{\rho_{t_1}}{\hat{\rho}_{t_2}}\right|^2\right) \rho_{t_1},
\end{align*}
we get, for $t>0$,
\[
	(\star) = \int_{\R^{2d}} \left(2 A \nabla \ln h_t \cdot A \nabla \frac{\partial_t \rho_t}{\hat{\rho}_t}\right) \rho_t + \int_{\R^{2d}} L_{\rho_t}\left(\left|A \nabla \ln h_t\right|^2\right) \rho_t.
\]
Therefore, writing $\Phi=| A \na \ln h_t |^2$, using \eqref{eq:hypoco1} and the equality~\eqref{eq:intLrho} with $g=\Phi$, and then~\eqref{eq:hypoco-log}, 
\begin{align}
    (\star) 
& = \int_{\R^{2d}} 2 A \na \ln h_t \cdot A\na \po \frac{L_{\rho_t}^*h_t}{h_t} \pf   \rho_t - \int_{\R^{2d}} \co h_t  L_{\rho_t}^*( \Phi) + 2\Gamma(\Phi,h_t)\cf  \hat\rho_t \nonumber  \\
& = \int_{\R^{2d}} 2 A \na \ln h_t \cdot A\na \po L_{\rho_t}^*(\ln h_t) + \Gamma(\ln h_t) \pf   \rho_t - \int_{\R^{2d}} \co h_t  L_{\rho_t}^*( \Phi) + 2\Gamma(\Phi,h_t)\cf  \hat\rho_t. \label{eq:hypoco2}
\end{align}
Focusing on the first term, and noticing that
\[\na L_{\rho_t}^* f = L_{\rho_t}^* \na f + J \na f\,,\]
we see that, denoting by $(B)$ the right-hand side of \eqref{eq:hypocoercivite}
\begin{eqnarray*}
\int_{\R^{2d}} 2 A \na \ln h_t \cdot A\na  L_{\rho_t}^*(\ln h_t)  \rho_t & = & (B) +  \int_{\R^{2d}} 2 A \na \ln h_t \cdot L_{\rho_t}^* \po A \na  \ln h_t\pf    \rho_t  \\
& = &  (B) +  \int_{\R^{2d}} \po L_{\rho_t}^* (\Phi) - 2\Gamma(A\na \ln h_t) \pf    \rho_t  \,.
\end{eqnarray*}
(In the first line, $L_{\rho_t}^* \po A \na  \ln h_t\pf$ is understood as applying coordinate-wise $L_{\rho_t}^*$ to $ A \na  \ln h_t$, and similarly $\Gamma$ in the second line is the sum over the coordinates of $\Gamma$ applied coordinate-wise). Notice that the first term of this integral cancels out with the first term of the second integral of \eqref{eq:hypoco2}. Moreover, by chain rule, for $i\in\cco 1,d\ccf$,
\[
A \na \ln h_t \cdot A\na \po |\partial_{v_i}\ln h_t|^2 \pf  = 2 \partial_{v_i}\ln h_t  A \na \ln h_t \cdot A\na \partial_{v_i}\ln h_t  =   \partial_{v_i}\ln h_t \partial_{v_i}(\Phi)\,.
\]
Summing over $i$ and multiplying by $\sigma^2$ reads 
\[
A \na \ln h_t \cdot A\na \po \Gamma\po \ln h_t\pf \pf    =   \Gamma\po \ln h_t,\Phi\pf = \frac{\Gamma\po  h_t,\Phi\pf}{h_t} \,.
\]
This means that the second terms of the two integrals in~\eqref{eq:hypoco2} cancel out.  We have obtained that
\[(\star) = (B) -  2\int_{\R^{2d}} \Gamma(A\na \ln h_t)    \rho_t\,.  \]
The integral being non-positive, this concludes the proof of~\eqref{eq:hypocoercivite}.

Writing $a_1= \sqrt{a_{11}^2+a_{21}^2}$, we bound
\begin{eqnarray*}
(\star)  & \leqslant &   2\int_{\R^{2d}} \co   \na \ln h_t \cdot A^T A J_0   \na  \ln h_t   + a_1M\left| A  \na \ln h_t \right| \left|  \na_v \ln h_t\right|\cf \dd\rho_t \\
 & \leqslant &  \int_{\R^{2d}} \co   \na \ln h_t \cdot A^T A\po 2 J_0 + r I_{2d}\pf   \na  \ln h_t   + \frac{a_1^2M^2}{r}  \left|  \na_v \ln h_t\right|^2\cf \dd\rho_t \,,
\end{eqnarray*}
for any $r>0$.

\rev{\emph{Step 2 : bounding ($\star\star$).}} We now deal with $(\star\star)$, that we bound first as
\begin{eqnarray*}
(\star\star) & \leqslant & 2 \int_{\R^{2d}} \left|A \na \ln h_t \right|\left|\frac{\dd}{\dd t} A\na \ln  \hat\rho_t\right| \dd\rho_t \\
& = & 2 a_1\int_{\R^{2d}} \left|A \na \ln h_t \right|\left|\frac{\dd}{\dd t} \na_x \ln  \hat\rho_t\right| \dd\rho_t \\
 & \leqslant & r' \int_{\R^{2d}} \left|A \na \ln h_t \right|^2 \dd\rho_t + \frac{a_{1}^2}{r'}  \int_{\R^{2d}} \left|\frac{\dd}{\dd t} \na_x \ln  \hat\rho_t\right|^2 \dd\rho_t \,,
\end{eqnarray*}
for any $r'>0$. By the explicit expression of $\hat\rho_t$ and \eqref{eq:Lipshitz}, for $t,s\geqslant 0$,
\[
\left|\na_x \ln  \hat\rho_t - \na_x \ln  \hat\rho_{t+s}\right| = \frac{1}{\sigma^2} \left| \na_x E_{\rho_t^x} - \na_x E_{\rho_{t+s}^x} \right|   
 \leqslant  \frac{L}{\sigma^2} \mathcal W_2\po \rho_t^x,\rho_{t+s}^x\pf\,.
\]
Using the continuity equation on $\rho_t^x$ obtained by integrating \eqref{eq:VFP} over the variable $v$ and using the Benamou-Brenier formulation of the Wasserstein distance~\cite{BenamouBrenier},  as detailed in the proof of \cite[Theorem 2.1]{Songbo2}, we get that 
\[\limsup_{s\rightarrow 0}  s^{-1} \mathcal W_2^2\po \rho_t^x,\rho_{t+s} ^x\pf \leqslant \int_{\R^{2d}} \left| \na_v \ln h_t \right|^2 \dd\rho_t\,,  \]
from which 
\[\left|\frac{\dd}{\dd t} \na_x \ln  \hat\rho_t\right|^2 \leqslant  \frac{L^2}{\sigma^4} \int_{\R^{2d}} \left| \na_v \ln h_t \right|^2 \dd\rho_t \,.\]

\rev{\emph{Step 3 : Conclusion.}} Gathering all previous bounds on $(\star)$ and $(\star\star)$ yields
\[\partial_t \int_{\R^{2d}} \left|A \na \ln h_t \right|^2 \dd\rho_t \leqslant  \int_{\R^{2d}}\na \ln h_t  \cdot B \na \ln h_t \dd\rho_t    \]
with
\[ B = A^T A \begin{pmatrix}
(r+r')I_d & 0 \\ -2 I_d & (r+r'-2) I_d
\end{pmatrix} + a_1^2 \po \frac{M^2}{r} + \frac{ L^2}{r'\sigma^4} \pf \begin{pmatrix}
0 & 0 \\ 0 & I_d
\end{pmatrix}\,.\]
Taking  $a_{ij}=1/\sqrt{2}$ for all $i,j\in\{1,2\}$ we get that 
\begin{eqnarray*}
 B & =  &   \begin{pmatrix}
(r+r'-2) I_d & (r+r'-2) I_d \\ (r+r'-2) I_d  & \po r+r'-2 + \frac{M^2}{r}+ \frac{L^2}{r'\sigma^4}\pf  I_d
\end{pmatrix}\\
& \preceq & \begin{pmatrix}
- \frac12 I_d & 0 \\ 0  & \po 1+ 2M^2 + 2\frac{L^2}{\sigma^4} \pf  I_d
\end{pmatrix}
\end{eqnarray*}
taking $r=r'=1/2$, where $\preceq$ stands for the order over symmetric matrices and we used that $ -2 xv \leqslant \frac{1}{2}x^2 +  2 v^2 $ to bound the cross term. In other words, we have obtained that
\begin{multline*}
\frac{\dd}{\dd t} \int_{\R^d} \left|(\na_x+\na_v) \ln h_t \right|^2 \dd\rho_t \leqslant - \frac12 \int_{\R^d} \left|\na_x \ln h_t \right|^2 \dd\rho_t
 \\
 + \po 1+ 2M^2 + 2\frac{L^2}{\sigma^4} \pf  \int_{\R^d} \left|\na_v \ln h_t \right|^2 \dd\rho_t\,.    
\end{multline*}
Conclusion follows from \eqref{eq:dissipationVFP} (the definition of $a$ ensuring that the last term above is compensated by half the dissipation term of \eqref{eq:dissipationVFP}).
 \end{proof}
 As a corollary, assuming in the kinetic case the local non-linear LSI 
 \begin{equation}
 \label{eq:nonlinLSIkinetic}
\forall \rho\in\mathcal A\,,\qquad \mathcal F_{k}(\rho) - \mathcal F_{k}(\rho_*) \leqslant \oeta \sigma^4 \mathcal I(\rho|\Gamma_k(\rho))\tag{\textbf{k-NL-LSI}}
 \end{equation}
for some $\mathcal A\subset \mathcal P_2(\R^{2d})$,  $\oeta>0$ and $\rho_*\in\mathcal K_k$, which, using that $|(\na_x+\na_v) f|^2 \leqslant 2|\na f|^2$, in turn implies 
 \begin{equation*}
\forall \rho\in\mathcal A\,,\qquad \mathcal L(\rho) - \mathcal L(\rho_*) \leqslant \po \oeta +2a\pf \sigma^4 \mathcal I(\rho|\Gamma_k(\rho)) \,,
 \end{equation*}
we get an exponential decay of $\mathcal L(\rho_t)-\mathcal L(\rho_*)$ for times $t\leqslant T_{\mathcal A}(\rho_0)$. 

\begin{rem}\label{rem:elliptic_to_kinetic}
For $\rho \in \mathcal{P}_2(\R^{2d})$,
\begin{eqnarray*}
\mathcal F_k(\rho) & = & \sigma^2 \mathcal H \po \rho|\Gamma_k(\rho)\pf - \sigma^2 \ln \po Z_{\rho^x} (2\pi\sigma^2)^{d/2}\pf + \mathcal E(\rho^x) - \int_{\R^d} \rho^x E_{\rho^x} \\
& = & \sigma^2 \mathcal H \po \rho|\Gamma_k(\rho)\pf - \frac d2\sigma^2 \ln \po 2\pi\sigma^2\pf + \mathcal G\po \rho^x\pf 
\end{eqnarray*}
with
\[\mathcal G(\rho^x) = \mathcal E(\rho^x) - \int_{\R^d} \rho^x  E_{\rho^x} - \sigma^2 \ln Z_{\rho^x}\,,  \]
which is also such that, in the elliptic case, $\mathcal F(\rho^x) = \sigma^2 \mathcal H(\rho^x|\Gamma(\rho^x)) + \mathcal G(\rho^x)$. In particular, by the extensivity property of the relative entropy, $\mathcal F_k(\rho) \geqslant \mathcal F(\rho^x) - d\sigma^2/2 \ln (2\pi\sigma^2) $.

If $\Gamma(\rho^x)$ satisfies a LSI uniformly over $\rho$, so does $\Gamma_k(\rho)$ by tensorization of LSI, and then both \eqref{eq:nonlinLSIkinetic} and \eqref{eq:loc_LSInonlineaire} follow from a bound of the form $\mathcal G(\rho^x)-\mathcal G(\rho_*^x) \leqslant \oeta' \mathcal I(\rho^x|\Gamma(\rho^x)) $ for some $\oeta'>0$, as has been established in various situations in  Section~\ref{sec:LSIparametric}.
\end{rem}

In order to follow the argument that led in the elliptic case to Theorem~\ref{thm:main}, we need to provide an analogue of  Corollary~\ref{Cor:FW2} in the present hypoelliptic case. This is done by   using the Wasserstein-to-entropy short time regularization of \cite[Theorem 4.1]{QianRenWang} instead of Proposition~\ref{prop:Wang} and the entropy-to-Fisher short time regularization of \cite[Proposition 5.5]{Songbo2}, as we now detail.

The following is from \cite[Proposition 5.5]{Songbo2} (the proof is based on computations similar to Lemma~\ref{lem:dissipationVFP}).
\begin{prop}\label{prop:FisherToEntropy}
Assuming \eqref{eq:Lipshitz} and \eqref{eq:M}, there exists $\varepsilon\in(0,1]$ (which depends only on $M,L$ and $\sigma^2$) such that for any solution of \eqref{eq:VFP} with $\mathcal F_k(\rho_0) <\infty$, writing
\[ \mathcal E(t,\rho) := \mathcal F_{k}(\rho)  +  \varepsilon t   \int_{\R^{2d}} \co \left|\na_v  \ln \frac{\rho}{\Gamma_k(\rho)} \right|^2 +  \left|(\varepsilon t \na_x+  \na_v) \ln \frac{\rho}{\Gamma_k(\rho)} \right|^2  \cf \dd \rho\,,\]
then $t \mapsto \mathcal E(t,\rho_t)$ is non-increasing over $t\in[0,1]$.
\end{prop}

The following is a particular case  of \cite[Theorem 4.1]{QianRenWang}.

\begin{prop}\label{prop:W2ToEntropy_kin}
Assuming \eqref{eq:Lipshitz} and \eqref{eq:M}, there exists $K>0$ (which depends only on $M,L$ and $\sigma^2$) such that for any solution of \eqref{eq:VFP} with $\rho_0\in\mathcal P_2(\R^{2d})$ and any $\rho_*\in\mathcal K_k$ and $t\geqslant 0$,
\[ \mathcal H\po \rho_{t+1} |\rho_*\pf \leqslant K \mathcal W_2^2(\rho_t,\rho_*)\,.\]
\end{prop}

We can now state the analogue to Theorem~\ref{thm:main}.

\begin{assu}\label{assu:total_kinetic}
The conditions \eqref{eq:convexity-c_W2}, \eqref{eq:Lipshitz}, \eqref{eq:M} and \eqref{eq:LSIlineaire_locale}  hold for some $\lambda, L,M,\eta>0$, and \eqref{eq:nonlinLSIkinetic} holds on some non-empty set  $\mathcal A \subset \mathcal P_2(\R^{2d})$ for some $\oeta>0$. Furthermore  there exist $\rho_* \in \mathcal A\cap \mathcal K_k$ and $\delta>0$ such that $\mathcal B_{\mathcal W_2}(\rho_*,\delta) \subset \mathcal A$ and $\mathcal B_{\mathcal W_2}(\rho_*,\delta) \cap \mathcal K_k = \{\rho_*\}$. The function $\mathcal L$ is given by \eqref{eq:defLa} with $a=\frac12 \po 1+ 2M^2 + 2\frac{L^2}{\sigma^4} \pf ^{-1}$.
\end{assu}

\begin{thm}\label{thm:mainVFP}
Under Assumption~\ref{assu:total_kinetic}, there exist $\delta',c,C>0$ (which depend only on $\lambda$, $M$, $L$, $\delta$, $\eta$, $\oeta$ and $\sigma^2$)  such that, for all $\rho_0 \in \mathcal B_{\mathcal W_2}(\rho_*,\delta')$, $T_{\mathcal A}(\rho_0)=\infty$; for all $t\geqslant 0$,
\begin{eqnarray}
\mathcal W_2^2( \rho_t,\rho_*) & \leqslant & C e^{-ct} \mathcal W_2^2 (\rho_0,\rho_*) \label{eq:thmVFPW2}\\
\mathcal F_k(\rho_t)-\mathcal F_k(\rho_*) & \leqslant & C e^{-ct} \po \mathcal F_k(\rho_0)-\mathcal F_k(\rho_*)\pf  \label{eq:thmVFPF}\\
 \mathcal H(\rho_t|\rho_\infty) & \leqslant & C e^{-ct}  \mathcal H(\rho_0|\rho_\infty) \label{eq:thmVFPH}
\end{eqnarray}
and for all $t\geqslant 1$,
\begin{multline}\label{eq:thmVFPminmax}
 \max \po \mathcal W_2^2 (\rho_{\rev{t}},\rho_*) \, ,\,  \mathcal H(\rho_{\rev{t}}|\rho_\infty) \,,\, \mathcal F_k(\rho_{\rev{t}})-\mathcal F_k(\rho_*)\pf \\
  \leqslant    C e^{-ct} \min \po \mathcal W_2^2 (\rho_0,\rho_*) \, ,\,  \mathcal H(\rho_0|\rho_\infty) \,,\, \mathcal F_k(\rho_0)-\mathcal F_k(\rho_*)\pf\,.
\end{multline}
\end{thm}

\begin{proof}
First, without loss of generality we can assume that \eqref{eq:LSIlineaire_locale} holds with $\eta\geqslant \sigma^2$, in which case the same inequality holds for all $\nu,\rho \in\mathcal P_2(\R^{2d})$ when $\Gamma$ is replaced by $\Gamma_k$, due to the LSI satisfied by $\mathcal N(0,\sigma^2 I_d)$ and the tensorization property of LSI. In particular, $\rho_*$ satisfies a LSI and thus a Talagrand inequality with constant $\eta$. Moreover, notice that \eqref{eq:convexity-c_W2} implies the same condition with $\mathcal E$ replaced by $\mathcal E + \mathcal E_k$, since $\mu\mapsto \mathcal E_k(\mu)$ is linear. In particular, in view of these two points, we can apply Lemmas~\ref{lem:sandwich} and \ref{lem:Songbo} in the kinetic case (namely with $\mathcal F$ replaced by $\mathcal F_k$,  $\Gamma$   by $\Gamma_k$,  $\rho_*\in\mathcal K_k$ and $\mu_0,\mu_1,\rho\in\mathcal P_2(\R^{2d})$). Using all these points, for any $\rho\in\mathcal P_2(\R^{2d})$,
\begin{eqnarray}
 \mathcal F_k(\rho) -  \mathcal F_k(\rho_*) &  \leqslant & \sigma^2 \mathcal H\po \rho|\Gamma_k(\rho)\pf+   \lambda \mathcal W_2^2(\rho,\rho_*) \nonumber \\
 & \leqslant & \sigma^2  \po 1+ L\eta +  \frac{4\lambda\eta}{\sigma^2} + \frac{L^2\eta^2 }{2}\pf \mathcal H(\rho|\rho_*)\,.\label{eq:demoTHM1}
\end{eqnarray} 
Similarly, we can apply Theorem~\ref{thm:main} to the elliptic equation on $\R^{2d}$
\[\partial_t \tilde \rho_t = \na\cdot \po \tilde \rho_t \na  \frac{\delta\mathcal{F}_k(\tilde \rho_t) }{\delta\mu} \pf\,, \]
from which \eqref{eq:thmTalagrand} reads 
\begin{equation}
\label{eq:demoTHMW2}
 \mathcal W_2^2(\rho,\rho_*) \leqslant 4\oeta \po \mathcal F_k(\rho) - \mathcal F_k(\rho_*)\pf
\end{equation}
for all $\rho \in \mathcal B_{\mathcal W_2}(\rho_*,\delta_1')$ for $\delta_1'\leqslant \delta $ small enough. Using that $t\mapsto \mathcal F_k(\rho_t)$ is non-increasing and Proposition~\ref{prop:W2ToEntropy_kin}, gathering the previous bounds gives that 
\[ \mathcal W_2^2(\rho_t,\rho_*)  \leqslant  4\oeta  \sigma^2  \po 1+ L\eta +  \frac{4\lambda\eta}{\sigma^2} + \frac{L^2\eta^2 }{2}\pf K  \mathcal W_2^2(\rho_0,\rho_*)\,. \]
for all $t\geqslant 1$ such that $\rho_t\in \mathcal B_{\mathcal W_2}(\rho_*,\delta_1')$.  Using a synchronous coupling as in the proof of Theorem~\ref{thm:main}, it is easily seen that
\begin{equation}
\label{eq:demoTHMW2synchro}
\mathcal W_2(\rho_t,\rho_*) \leqslant e^{Ct}\mathcal W_2(\rho_0,\rho_*) 
\end{equation}
for some $C$ depending only on $M$ and $L$. As a consequence, we can take $\delta'_2 \in(0, \delta_1']$ such that for all $\rho_0 \in \mathcal B_{\mathcal W_2}(\rho_*,\delta_2')$, it holds that $\rho_t \in \mathcal B_{\mathcal W_2}(\rho_*,\delta_1')$ for all $t\in[0,1]$ and 
\[ \mathcal W_2(\rho_t,\rho_*)  \leqslant \frac{\delta_1'}2 \]
for all $t\geqslant 1$ such that $\rho_t\in \mathcal B_{\mathcal W_2}(\rho_*,\delta_1')$. Using that $t\mapsto \rho_t$ is continuous with respect to $\mathcal W_2$, we get by contradiction that $\rho_t\in \mathcal B_{\mathcal W_2}(\rho_*,\delta_1')$ for all $t\geqslant 0$ if $\rho_0 \in \mathcal B_{\mathcal W_2}(\rho_*,\delta_2')$, and in particular $T_{\mathcal A}(\rho_0) = +\infty$.

In the remaining of the proof we consider an initial condition  $\rho_0 \in \mathcal B_{\mathcal W_2}(\rho_*,\delta_2')$. As we used in  \eqref{eq:rho*min}, \eqref{eq:nonlinLSIkinetic}  implies that $\rho_*$ is a minimizer of $\mathcal F_k$ over $\mathcal B_{\mathcal W_2}(\rho_*,\delta)$, so that $\mathcal F_k(\rho_t) \geqslant \mathcal F_k(\rho_*)$ for all $t\geqslant 0$.  Applying Proposition~\ref{prop:FisherToEntropy} (but with initial condition $\rho_{t}$ instead of $\rho_0$), we get that, for all $t\geqslant 0$,
\begin{eqnarray*}
\frac{\varepsilon^3}2  \mathcal I\po \rho_{t+1}|\Gamma(\rho_{t+1})\pf & \leqslant &  \mathcal E(1,\rho_{t+1}) - \mathcal F_k(\rho_{t+1}) \\
& \leqslant &  \mathcal E(0,\rho_t) - \mathcal F_k(\rho_*)  \ = \ \mathcal F_k(\rho_t) - \mathcal F_k(\rho_*)\,,
\end{eqnarray*}
and thus, using that $t\mapsto \mathcal F_k(\rho_t)$ is non-increasing, there exists $R>0$ such that
\[\mathcal L(\rho_{t+1} ) \leqslant R \po \mathcal F_k(\rho_t) - \mathcal F_k(\rho_*)\pf\]
for all $t\geqslant 0$. Since Lemma~\ref{lem:dissipationVFP} together with \eqref{eq:nonlinLSIkinetic} implies an exponential decay of $\mathcal L(\rho_{t} )$ at some rate $c>0$, we can bound, for $t\geqslant 1$,
\begin{eqnarray*}
\mathcal F_k(\rho_t) - \mathcal F_k(\rho_*) & \leqslant & \mathcal L(\rho_t) \\
& \leqslant & e^{-c(t-1)} \mathcal L(\rho_1)  \ \leqslant \ R e^{-c(t-1)}  \po \mathcal F_k(\rho_0) - \mathcal F_k(\rho_*)\pf\,.
\end{eqnarray*}
For $t\in[0,1]$ we can simply use that $t\mapsto \mathcal F_k(\rho_t)$ is non-increasing, and this concludes the proof of \eqref{eq:thmVFPF}. The results for $\mathcal W_2$ and the relative entropy then follow by combining \eqref{eq:thmVFPF} with the various bounds relating the different quantities. More specifically, \eqref{eq:thmVFPW2} is obtained simply by \eqref{eq:demoTHMW2synchro} for $t\in[0,1]$ and, for $t\geqslant 1$, by combining  \eqref{eq:demoTHMW2}, \eqref{eq:thmVFPF}, \eqref{eq:demoTHM1} and Proposition~\ref{prop:W2ToEntropy_kin} to bound
\[\mathcal W_2^2 (\rho_t,\rho_*) \leqslant 4\oeta  C e^{-c(t-1)} \po \mathcal F_k(\rho_1) - \mathcal F_k(\rho_*)\pf \leqslant  C' e^{-ct}\mathcal W_2^2 (\rho_0,\rho_*)  \]
for some $C'$. Similarly, \eqref{eq:thmVFPH} is obtained by using first \eqref{eq:Fmu0Fmu1>}, then \eqref{eq:thmVFPW2} and \eqref{eq:thmVFPF} and finally \eqref{eq:demoTHM1} and \eqref{eq:demoTHMW2} to bound
\begin{eqnarray*}
\sigma^2 \mathcal H(\rho_t |\rho_*) & \leqslant &  \mathcal F_k(\rho_t) - \mathcal F_k(\rho_*)  + \lambda \mathcal W_2^2 (\rho_t,\rho_*) \\
& \leqslant & C e^{-ct} \po \mathcal F_k(\rho_0) - \mathcal F_k(\rho_*)  + \lambda \mathcal W_2^2 (\rho_0,\rho_*)  \pf  \\
& \leqslant & C' e^{-ct} \mathcal H(\rho_0 |\rho_*)
\end{eqnarray*}
for some $C'$. The proof of \eqref{eq:thmVFPminmax} uses the same ingredients.
\end{proof}

\begin{Example}
In the one-dimensional double-well case, Proposition~\ref{prop:double-well} remains true if the granular media equation~\eqref{eq:granular_media} is replaced by the Vlasov-Fokker-Planck equation~\eqref{eq:VFP}. In the subcritical case, it directly follows from Theorem~\ref{thm:mainVFP}, thanks to Proposition~\ref{prop:doublepuit} and Remark~\ref{rem:elliptic_to_kinetic}. In the critical case, the arguments are as in Remark~\ref{rem:double-puit-degenere}.
\end{Example}

\section{Fast free energy decay for interacting particles}\label{sec:particules}

\subsection{Settings and results}

We consider the particle approximations of \eqref{eq:granular_media_1} and \eqref{eq:VFP}, namely the overdamped Langevin process $(X_t)_{t \geq 0}$ solving
\begin{eqnarray}
\label{eq:particules_overdamped}
\forall i\in\cco 1,N\ccf\,,\qquad \dd X_t^i = - \na E_{\pi(X_t)}(X_t^i)\dd t + \sqrt{2}\sigma \dd B_t^i
\end{eqnarray}
and the kinetic Langevin process $(Y_t,V_t)_{t \geq 0}$ solving
\begin{equation}
\label{eq:particules_kinetic}
\forall i\in\cco 1,N\ccf\,,\qquad
\left\{\begin{array}{rcl}
 \dd Y_t^i &=& V_t^i \dd t \\
 \dd V_t^i &= &  - \na E_{\pi(Y_t)}(Y_t^i)\dd t - V_t^i \dd t + \sqrt{2}\sigma \dd B_t^i \,,
\end{array}\right.
\end{equation}
where in both cases $B^1,\dots,B^N$ are independent $d$-dimensional Brownian motions and
\[\pi(x) = \frac1N\sum_{i=1}^N \delta_{x_i}\]
stands for the empirical distribution of $x=(x_1,\dots,x_N)\in\R^{dN}$. Under~\eqref{eq:Lipshitz} and, respectively, \eqref{eq:onesided} for~\eqref{eq:particules_overdamped} and~\eqref{eq:M} for~\eqref{eq:particules_kinetic}, both equations admit a unique global strong solution.

Considering first the overdamped case, denote by $\rho^N_t$ the law of $X_t=(X_t^1,\dots,X_t^N)$ solving \eqref{eq:particules_overdamped}.
Assuming that
\begin{equation}
    \label{eq:ZN<infty}
    \int_{\R^{dN}} \exp\left(-\frac{N}{\sigma^2}\mathcal{E}(\pi(x))\right) \dd x < \infty\,,
\end{equation}
we consider the Gibbs measure
\begin{equation*}
    \rho^N_\infty \propto \exp\left(-\frac{N}{\sigma^2}\mathcal{E}(\pi(x))\right)\dd x\,.
\end{equation*}
Using that
\begin{equation*}
    \forall i \in \cco 1, N\ccf, \qquad \nabla_{x_i} N\mathcal{E}(\pi(x)) = \nabla E_{\pi(x)}(x)\,,
\end{equation*}
by the properties of the linear derivative, we see that \eqref{eq:particules_overdamped} is an overdamped Langevin diffusion on $\R^{dN}$ with potential $ N\mathcal{E}(\pi(x))$ and temperature $\sigma^2$. Its invariant measure is $\rho_\infty^N$, and $t \mapsto \mathcal{H}(\rho^N_t|\rho^N_\infty)$ is decreasing (by Jensen inequality) and goes to zero as $t\rightarrow \infty$ under mild assumptions. Defining the $N$-particle  free energy of $\rho^N \in \mathcal P_2(\R^{dN})$ as 
\begin{equation*}
    \mathcal F^N(\rho^N) = \sigma^2\mathcal H(\rho^N)+ N \int_{\R^{dN}} \mathcal E(\pi(x)) \rho^N(\dd x)\,,
\end{equation*}
we see that
\begin{equation*}
    \sigma^2 \mathcal{H}(\rho^N|\rho^N_\infty) = \mathcal{F}^N(\rho^N)-\mathcal{F}^N(\rho^N_\infty).
\end{equation*}
In particular, $\rho_\infty^N$ is the global minimizer of $\mathcal F^N$ and $t \mapsto \mathcal{F}^N(\rho^N_t|\rho^N_\infty)$ is non-increasing.

As studied in \cite{Pavliotis}, when the non-linear dynamics \eqref{eq:granular_media_1} admits several stable stationary solutions, the log-Sobolev constant of $\rho_\infty^N$ (assuming that a LSI holds, for instance if $x\mapsto \mathcal E(\pi(x))$ is uniformly convex outside a compact set) goes to infinity   as $N\rightarrow \infty$, and thus we cannot expect a fast convergence of $\mathcal F^N(\rho_t^N)$ to its infimum (which typically occurs at a time-scale of order $e^{aN}$ for some $a>0$). However, in short time, initialized with independent initial conditions, the law of the particle system stays close to the solution of the non-linear problem~\eqref{eq:granular_media_1} and thus the local convergence of the latter to some critical point $\rho_*\in\mathcal K$ drives the initial behavior of the particle system. This gives the following.

\begin{prop}\label{prop:particules_overdamped}
Under Assumption~\ref{assu:total}, assume furthermore~\eqref{eq:ZN<infty} and that there exists $\lambda'\geqslant 0$ such that  for all $\mu_0,\mu_1\in\mathcal P_2(\R^d)$ and $t\in[0,1]$,
\begin{equation}
\label{eq:concavity} 
\mathcal E(t\mu_0 + (1-t) \mu_1) \geqslant t \mathcal E(\mu_0) +(1-t)\mathcal E(\mu_1) - t(1-t) \lambda' \mathcal W_2^2(\mu_0,\mu_1) \,.
\end{equation}
Let $\rho_0 \in\mathcal P_2(\R^d)$ be such that the solution of \eqref{eq:granular_media_1} converges to $\rho_*$ in long time. There exist $C,\beta>0$ such that, for all $N\in\N$ and $t\geqslant 1$,  considering the particle system \eqref{eq:particules_overdamped} with initial distribution $\rho_0^N = \rho_0^{\otimes N}$,
\[\frac1N \mathcal F^N(\rho_t^N) - \mathcal F(\rho_*) \leqslant C \po N^{-\beta} + e^{-t/\oeta}\pf \,. \]
\end{prop}

Notice that, in view of \eqref{eq:courbureWa_kr_k}, \eqref{eq:concavity}  holds for instance in the case \eqref{eq:Wa_kr_k} if the functions $r_k$ are $L_k$-Lipschitz with $\sum_{k\in\N} L_k^2 <\infty$.

\begin{rem}
\label{rem:DelarueTse}
When the solution of the non-linear McKean-Vlasov equations converges to some $\rho_*$, under suitable  conditions (in particular of regularity), a quantitative weak convergence at rate $1/N$ for the empirical distribution of the corresponding system of $N$ interacting particles  is shown in 
\cite[Theorem 3.1]{delarue2021uniform} to hold over times of order $N^p$ for any $p$. This requires the initial condition to start close to $\rho_*$ in the total variation sense. However, thanks to Proposition~\ref{prop:Wang}, if the initial condition is close to $\rho_*$ in $\mathcal W_2$, then it becomes close in the total variation sense after a time $1$.
\end{rem}

The same occurs in the kinetic case. For $\rho^N \in \mathcal P_2(\R^{2dN})$, define the $N$-particle  kinetic free energy as 
\[\mathcal F^N_k(\rho^N) =  \sigma^2\mathcal H(\rho^N)  N \int_{\R^{2dN}} \co \mathcal E(\pi(x)) + \mathcal E_k(\pi(v))\cf \rho^N(\dd x\dd v)\,. \]

\begin{prop}\label{prop:particules_kinetic}
Under Assumption~\ref{assu:total_kinetic}, assume furthermore   that there exists $\lambda'\geqslant 0$ such that  \eqref{eq:concavity}  holds for all $\mu_0,\mu_1\in\mathcal P_2(\R^d)$ and $t\in[0,1]$. Let $\rho_0 \in\mathcal P_2(\R^{2d})$ be such that the solution of \eqref{eq:VFP} converges to $\rho_*$ in long time. There exist $C,\beta,\gamma>0$ such that, for all $N\in\N$ and $t\geqslant 1$,  considering the kinetic particle system \eqref{eq:particules_kinetic}  with initial distribution $\rho_0^N = \rho_0^{\otimes N}$,
\[\frac1N \mathcal F_k^N(\rho_t^N) - \mathcal F_k(\rho_*) \leqslant C \po N^{-\beta} + e^{-\gamma t}\pf \,. \]
\end{prop}

\subsection{Proofs}

We mostly focus on the overdamped case in this section, denoting by $\rho_t^N$ the law of \eqref{eq:particules_overdamped}. We start with the following variation of Proposition~\ref{prop:Wang}.

\begin{prop}\label{prop:W2entropyPoC}
For any $\rho_* \in \mathcal K$, there exists $C>0$ such that for all $N\in\N$, $t\geqslant 0$ and any initial distribution $\rho_0^N \in\mathcal P_2(\R^{dN})$, 
\[\mathcal H\po \rho_{t+1}^N | \rho_{*}^{\otimes N} \pf \leqslant C \mathcal W_2^2\po \rho_t^N,\rho_*^{\otimes N}\pf + C \,.\]
\end{prop}
Similarly  to Proposition~\ref{prop:Wang}, this follows from a Girsanov transform, see \cite[Theorem 2.3]{Huang} or \cite[Lemma 5.3]{Songbo2}.

To relate $\mathcal H\po \rho_{t}^N | \rho_{*}^{\otimes N} \pf$ to $\mathcal{F}^N(\rho_t^N)$, we can rely on the following.

\begin{lem}\label{lem:overdamped_particule}
Assume that there exists $\lambda'\geqslant 0$ such that \eqref{eq:concavity} holds for all $\mu_0,\mu_1\in\mathcal P_2(\R^d)$ and $t\in[0,1]$. Then, for all $\rho^N \in \mathcal P_2(\R^{dN})$ and $\rho_*\in\mathcal K$, 
\[\mathcal F^N(\rho^N) - N \mathcal F(\rho_*) \leqslant    \sigma^2 \mathcal H\po \rho^N|\rho_*^{\otimes N}\pf + 2 \lambda'\mathcal W_2^2 \po \rho^N,\rho^{\otimes N}_*\pf + 2\lambda' N \alpha(N) \]
where
\[\alpha(N) = \int_{\R^{dN}}    \mathcal W_2^2(\pi(x), \rho_*)   \rho_*^{\otimes N}(\dd x)\,.\]
\end{lem}
Thanks to \cite[Theorem 1]{FournierGuillin}, $\alpha(N) \leqslant C N^{-2/d}$ for some $C$ (but possibly for some particular $\rho_*$ it may vanish faster).

 \begin{proof}
   Dividing by $t$ and sending $t$ to zero in \eqref{eq:concavity} yields
\[\int_{\R^d} E_{\mu_1} (\mu_0-\mu_1) + \mathcal E(\mu_1) \geqslant  \mathcal E(\mu_0)  -    \lambda \mathcal W_2^2(\mu_0,\mu_1)\,.
\]
Then, using this with $\mu_0 = \pi(x)$ and $\mu_1=\rho_*$,
\begin{eqnarray*}
       \lefteqn{ \mathcal F^N(\rho^N) - N \mathcal F(\rho_*) }\\
        & = & N \int_{\R^{dN}} \co \mathcal E(\pi(x)) - \mathcal E(\rho_*)\cf \rho^N(\dd x) + \sigma^2\mathcal H(\rho^N) - N \sigma^2\mathcal H(\rho_*) \\
    & \leqslant & N \int_{\R^{dN}} \co   \lambda' \mathcal W_2^2(\pi(x),\rho_*) + \int_{\R^{2d}}E_{\rho_*}(y) (\pi(x)-\rho_*)(\dd y)   \cf \rho^N(\dd x) + \sigma^2\mathcal H(\rho^N) - N\sigma^2 \mathcal H(\rho_*).
\end{eqnarray*}
Using that $E_{\rho_*} = -\sigma^2(\ln \rho_* + \ln Z_{\rho_*})$, 
\begin{eqnarray*}
N \int_{\R^{dN}}      \int_{\R^{2d}}E_{\rho_*}(y) (\pi(x)-\rho_*)(\dd y)     \rho^N(\dd x) & = & N \sigma^2 \mathcal H(\rho_*) - \sigma^2\sum_{i=1}^N \int_{\R^{dN}}      \ln\rho_*(x_i)   \rho^N(\dd x) \\
 & = & N \sigma^2\mathcal H(\rho_*) - \sigma^2 \int_{\R^{dN}}      \ln\rho_*^{\otimes N}(x)   \rho^N(\dd x) \,.
\end{eqnarray*}
Plugging this in the previous inequality reads 
\[ \mathcal F^N(\rho^N) - N \mathcal F(\rho_*) \leqslant \sigma^2\mathcal H\po \rho_N|\rho_*^{\otimes N}\pf +  N \lambda'   \int_{\R^{dN}}    \mathcal W_2^2(\pi(x),\rho_*)   \rho^N(\dd x) \,.\]
Let $\mu(\dd x\dd y)$ be an optimal $\mathcal W_2$ coupling of $\rho^N$ and $\rho_*^{\otimes N}$. We bound
\begin{eqnarray*}
\int_{\R^{dN}}    \mathcal W_2^2(\pi(x),\rho_*)   \rho^N(\dd x) & \leqslant & 2 \int_{\R^{dN}}    \mathcal W_2^2(\pi(x),\pi(y))   \mu(\dd x\dd y) + 2 \alpha(N)  \\
 & \leqslant & \frac{2}{N}\sum_{i=1}^N \int_{\R^{dN}}    |x_i-y_i|^2  \mu(\dd x\dd y) + 2 \alpha(N)  \\
 &= & 2 \mathcal W_2^2 \po \rho^N,\rho^{\otimes N}_*\pf  + 2 \alpha(N)\,,
\end{eqnarray*}
which completes the proof.
\end{proof}

\begin{proof}[Proof of Proposition~\ref{prop:particules_overdamped}]
Thanks to \eqref{eq:convergence_temps-long} and standard finite propagation of chaos estimates (obtained by synchronous coupling), we bound 
\begin{eqnarray*}
\mathcal W_2^2\po \rho_t^N,\rho_*^{\otimes N}\pf &  \leqslant &  2 \mathcal W_2^2\po \rho_t^N,\rho_t^{\otimes N}\pf + 2 \mathcal W_2^2\po \rho_t^{\otimes N},\rho_*^{\otimes N}\pf \\
& \leqslant & C e^{Ct}  + 2 N C_0 e^{-t/\oeta}
\end{eqnarray*}
for some $C>0$. Proposition~\ref{prop:W2entropyPoC} and Lemma~\ref{lem:overdamped_particule} yield
  \[\mathcal F^N(\rho_t^N) - N \mathcal F(\rho_*)  \leqslant C' \po  e^{Ct}  + N  e^{-t/\oeta} + N \alpha(N)\pf \]
for all $t\geqslant 1$ for some $C'>0$. Conclusion follows by using that $t \mapsto \mathcal F^N(\rho_t^N)$ is non-decreasing (distinguishing whether $t$ is above or under $\ln N /(2C)$) and applying \cite[Theorem 1]{FournierGuillin} to bound $\alpha(N)$.
\end{proof}

\begin{proof}[Proof of Proposition~\ref{prop:particules_kinetic}]
The proof is the same as the proof of  Proposition~\ref{prop:particules_overdamped}. Proposition~\ref{prop:W2entropyPoC} still holds in the kinetic case (with $\rho_*\in\mathcal K_k$; this is \cite[Lemma 5.3]{Songbo2}), and so does Lemma~\ref{lem:overdamped_particule} (replacing $\mathcal F$ by $\mathcal F_k$, with the same proof, noticing that the kinetic energy is linear and thus satisfies \eqref{eq:concavity} with $\lambda'=0$). We use Theorem~\ref{thm:mainVFP} instead of Theorem~\ref{thm:main}.
\end{proof}

\subsection*{Acknowledgements}

The research of P. Monmarch\'e is supported by the projects SWIDIMS (ANR-20-CE40-0022) and CONVIVIALITY (ANR-23-CE40-0003) of the French National Research Agency. The research of J. Reygner is supported by the projects QUAMPROCS (ANR-19-CE40-0010) and CONVIVIALITY (ANR-23-CE40-0003) of the French National Research Agency.

\end{document}